\documentclass[10pt,a4paper]{report}

\pdfoutput=1

\usepackage{graphicx}
\usepackage{subfigure}

\usepackage{enumerate}

\usepackage{amsmath,amssymb}

\usepackage{amsthm}

\theoremstyle{plain}
\newtheorem{theorem}{Theorem}[chapter]
\newtheorem{lemma}[theorem]{Lemma}

\newtheorem{corollary}[theorem]{Corollary}
\newtheorem{conjecture}[theorem]{Conjecture}

\newtheorem{definition}[theorem]{Definition}

\theoremstyle{definition}
\newtheorem{remark}[theorem]{Remark}
\newtheorem{example}[theorem]{Example}

\newtheorem{construction}[theorem]{Construction}
\newtheorem{claim}[theorem]{Claim}

\newtheorem{rules}[theorem]{Rule}

\newcommand{\HRule}{\rule{\linewidth}{0.5mm}}

\begin{document}
  \begin{titlepage}
  \begin{center}
    \textsc{\LARGE University of Durham} \\[1.5cm]

    \textsc{\Large Final year project} \\[0.5cm]

    \HRule \\[0.4cm]
    
    { \huge Empire Maps on Surfaces } \\[0.4cm]
    
    { \large An exploration into colouring empire maps on the sphere and higher genus surfaces } \\[0.4cm]

    \HRule \\[1.5cm]
    
\begin{minipage}{0.4\textwidth}
\begin{flushleft} \large
\emph{Author:}\\
Caspar de Haes
\end{flushleft}
\end{minipage}
\begin{minipage}{0.4\textwidth}
\begin{flushright} \large
\emph{Supervisor:} \\
Dr Vitaliy Kurlin
\end{flushright}
\end{minipage}

    \vfill
    
    \begin{figure}[h]
	  \begin{center}
	  	\includegraphics[scale=0.3]{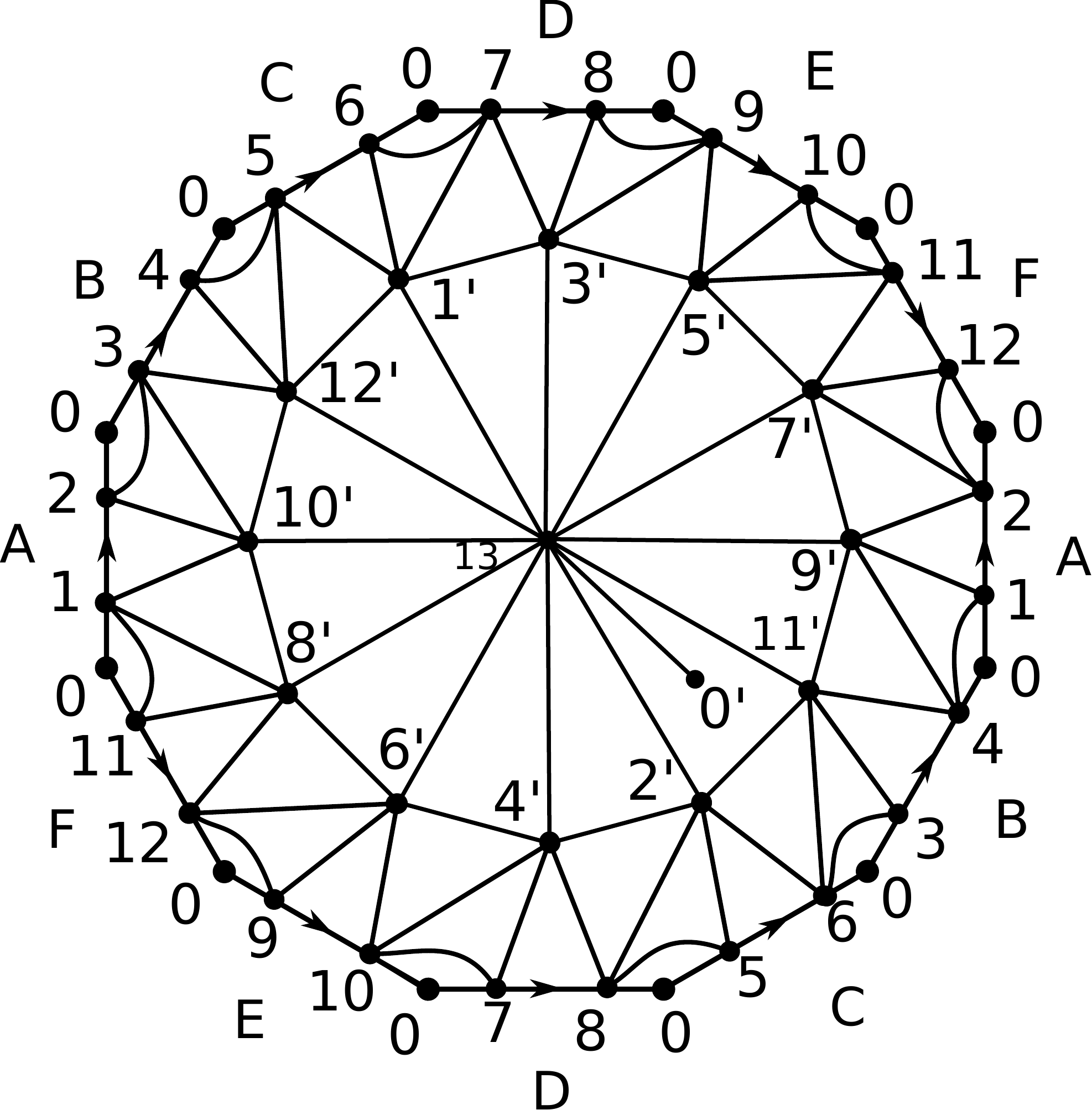}
	  \end{center}
	\end{figure}
	
	\vfill
  
    {\large April 28, 2011 }

  \end{center}
\end{titlepage}
  
  \begin{abstract}
    This report is an introduction to mathematical map colouring and the problems posed by Heawood in his paper of 1890.
	
	There will be a brief discussion of the Map Colour Theorem; then we will move towards investigating empire maps in the plane and the recent contributions by Wessel. Finally we will conclude with a discussion of all known results for empire maps on higher genus surfaces and prove Heawood’s Empire Conjecture in a previously unknown case.
  \end{abstract}
  
  \tableofcontents

  \chapter{Introduction} \label{chp:introduction}
    \section{Overview of the Problem}
  To get a feel for the problems discussed in this report is very simple; take an as yet un-coloured map of the counties of England and find yourself a box of coloured pencils. How many colours do you need to colour the map?

  Now, to make things clear, you could colour each county a different colour. However even with the counties of England one might struggle to have enough colours to hand, and consider what would happen if I gave you a map of the world! So this isn’t very practical, how about if we just coloured neighbouring counties a different colour?

  This should then reduce our problem, but how do we decide when two counties are neighbouring? What if they share any point in common? Well this is one possibility and indeed it would work fine for the counties of England, however consider a map where say 100 counties all met at a single point (like the centre of a spiral), you would then need 100 colours! This seems a little more far fetched but certainly it would be easier if one didn't have to consider this possibility.
  
  So instead we consider neighbouring counties to be ones that have a common line, this seems more sensible, how many colours do we need now? (The colouring in Figure~\ref{fig:counties-of-england} uses 4, can you do it with fewer colours?) Well this is the very problem we will be discussing in this report. Well, almost! We are not going to be considering the counties of England but instead we will be taking the role of a cartographer who has the job of travelling around colouring any possible type of map that people want him to. Since he travels a lot, he would like to carry round as few colours as possible, so given any possible map of the world that anyone could ever think of, how many colours will he need? It is these kind of questions that this report will be discussing.
	
	\begin{figure}[htb]
	  \begin{center}
		\includegraphics[scale=0.4]{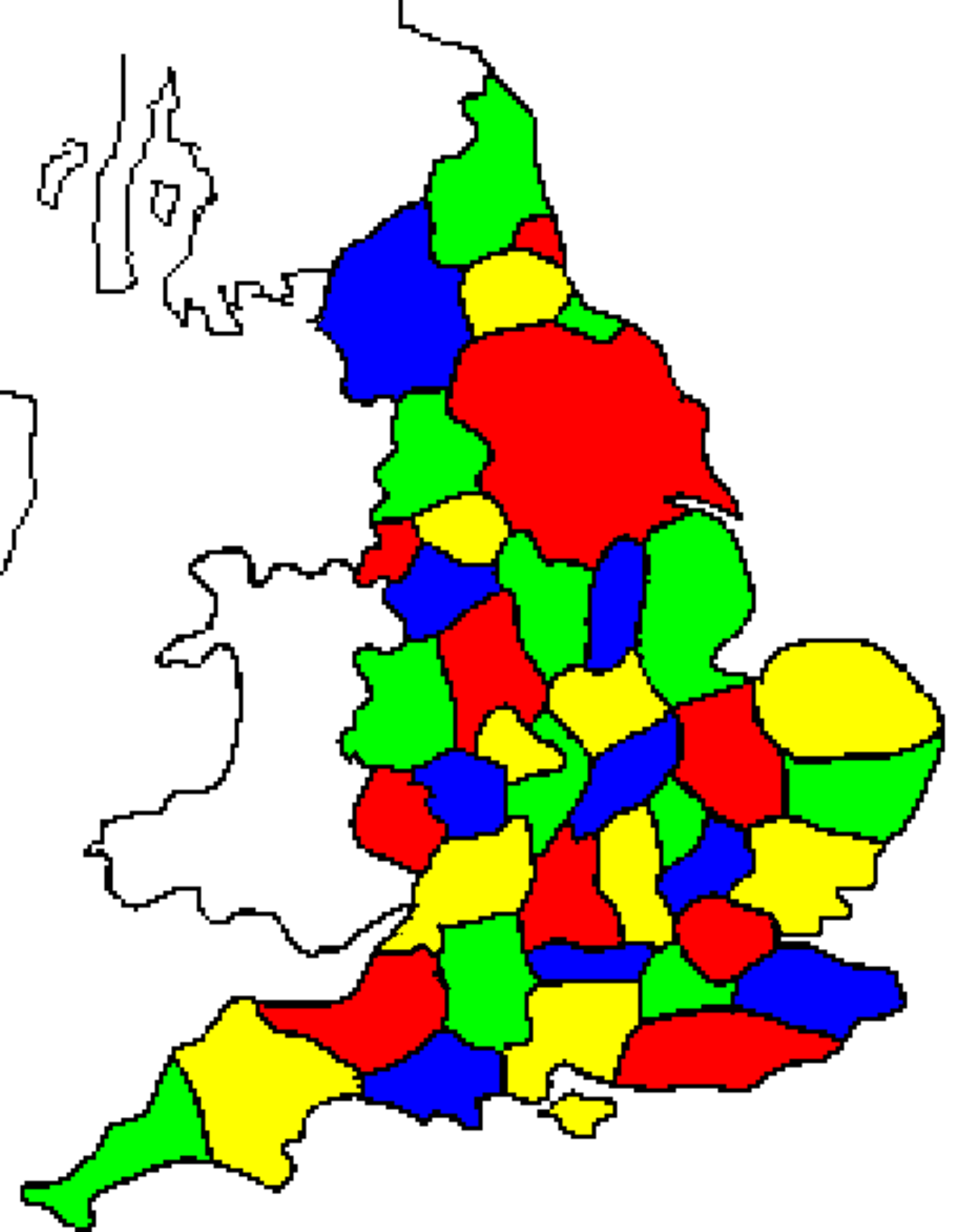}
	  \end{center}
	  \caption{A colouring of the counties of England using 4 colours. Copyright Mark Berry, \cite{windmill}.}
	  \label{fig:counties-of-england}
	\end{figure}

\section{History of the Problem} \label{sec:history}
  Map Colouring as a mathematical problem has a rather convoluted beginning. Initially it was noticed by Francis Guthrie when colouring a map of the counties of England, but Francis had no mathematical training (at the time) and so passed it onto his brother Frederick Guthrie. At the time Frederick was a student of de Morgan in London and so he posed the problem to de Morgan who in turn then mentioned it in a letter to a fellow mathematician of the time, William Hamilton:
  \begin{quote}
    ``If a figure be anyhow divided. . . four colours may be wanted but not more"
  \end{quote}
  However from this point on, neither of the Guthrie brothers, de Morgan nor Hamilton had anything to do with the problem! The conjecture was first published by Cayley in his paper of 1879, although he did attribute the conjecture to de Morgan.
  
  The first ``proof" was also published in 1879 by Alfred Kempe and it met with wide acclaim and it wasn't until 11 years later in 1890 that Heawood highlighted fatal flaws in Kempe's proof.
  
  In his paper of 1890 \cite{heawood}, Heawood found faults in Kempe's proof but was unable to fix them and so instead he tried to explore other areas of map colouring; it is these ideas that Heawood thought of back then that will form the basis for our discussions in this report. It was a prolific paper in which Heawood posed several extensions including:
  \begin{itemize}
    \item What if we change the surface? (i.e. a doughnut instead of a sphere)
    \item What if we allow ``empires" or colonies? (such as USA and Alaska)
    \item What if we allow colonies on another body such as the moon?
  \end{itemize}
  In fact Heawood managed to make considerable progress in many of these areas in the same paper, producing an upper bound on the number of colours required to in both the first two cases.
  
  Work on the matter took somewhat of a back seat for a while, with small contributions coming here are there but largely indirectly. It wasn't until Ringel arrived on the scene in the mid 20\textsuperscript{th} century that results were finally achieved. In fact Ringel was very prolific in the area proving a result for one class of surfaces (non-orientable) in 1954, then another (orientable) in 1968 with the help of Youngs but it took Appel and Haken until 1976 to complete the last case, the one where it all had started, the sphere. Even then they required the use of computers which was revolutionary at the time and which many people distrusted although the proof has gained wider acceptance over the years.
  
  All of the work mentioned above focused on the classic problem as posed by Francis Guthrie, but the extension mentioned above which allowed ``empires" also proved to be a rich area of mathematics. Once again Heawood provided an upper bound in 1890 but this time it wasn't until 1984 that a proof that the upper bound was sharp was finally given, again by the prolific Ringel, this time with the help of Jackson. They conclusive settled the empire problem on the sphere, however Heawood had even thought about combining both the first and the second change to allow empire maps on different surfaces and this problem has yet to be solved in its entirety. It is this unknown area which we will be working towards in this report and providing a proof of a previously unknown case.

\section{Chapter Plan}
  A brief summary of each of the chapters in this report is as follows:
  \begin{enumerate} [(a)]
    \item Chapters \ref{chp:graphs}- \ref{chp:colouring} concentrate on a mathematical introduction to graphs, surfaces and maps
    \item Chapter \ref{chp:heawood-higher-genus} proves Heawood’s upper bound for maps on higher genus surfaces
    \item Chapter \ref{chp:empire-plane} introduces empire maps and looks at Hutchinson’s proof and includes the proof of an essential claim that was left in \cite[p. 215]{hutchinson} as an exercise
	\item Chapter \ref{chp:empire-conjecture} provides a complete review of Wessel’s recent contribution to empire maps in the plane
	\item Chapter \ref{chp:torus} provides an overview of all known results for empire maps on higher genus surfaces and also proves Heawood’s Empire Conjecture in the previously unknown case of 2-pire maps on a triple torus
	\item Chapter \ref{chp:conclusion} is a brief conclusion highlighting key points from the report
	\item Appendix \ref{app:tools} provides some tools and idea
  \end{enumerate}

  \chapter{Graphs} \label{chp:graphs}

Before we can embark upon our exploration into map colouring we must first set up some mathematical objects that will allow us to discuss more precisely exactly what we mean by a map and colouring. We start by discussing mamthematical objects called graphs which will ultimately provide us with a means of mathematically describing a map.

\section{Terminology and Notation}
  We begin by defining an abstract combinatorial graph (not to be confused with a ``normal" graph with axis etc.):
  \begin{definition} \label{def:combinatorial-graph}
	  A finite general \textbf{graph} $G$ consists of three things:
	  \begin{itemize}
	    \item A finite set of \textit{vertices}
	    \item A finite set of \textit{edges}
	    \item An \textit{incidence} relation between vertices and edges
	  \end{itemize}
	  If $v$ a vertex and $e$ an edge and $(v,e)$ is a pair in the incidence relation, then we say that $v$ is incident with $e$ and $e$ is incident with $v$.
	\end{definition}

  \begin{remark}
    We will be dealing entirely with \emph{undirected graphs}, that is to say an edge between vertices $v$ and $w$ is equivalent to an edge between vertices $w$ and $v$. Thus for the purposes of this report a \emph{graph} will always refer to a \emph{finite undirected graph}.
  \end{remark}
  
	We are now in a position to cover some basic terminology relating to graphs:
	\begin{definition}
	  Two vertices $v, w$ with $v \neq w$ are \textbf{adjacent} if they are both incident to a common edge.
	\end{definition}
	\begin{definition}
		A \textbf{loop} is an edge which is incident with only one vertex.
	\end{definition}
	\begin{definition}
		A graph is said to have a \textbf{multiple edge} (or a \textbf{double edge}, \textbf{triple edge}, etc.) if there exists two vertices $v$ and $w$ such that the number of edges incident to both $v$ and $w$ is greater than 1.
	\end{definition}
	
	One common notion that consequently appears as a result of the above definitions is that of a simple graph:
	\begin{definition}
		If a graph contains no loops and no multiple edges then the graph is said to be a \textbf{simple graph}.
	\end{definition}
	
	\begin{figure}[htb]
	  \begin{center}
		\includegraphics[scale=0.7]{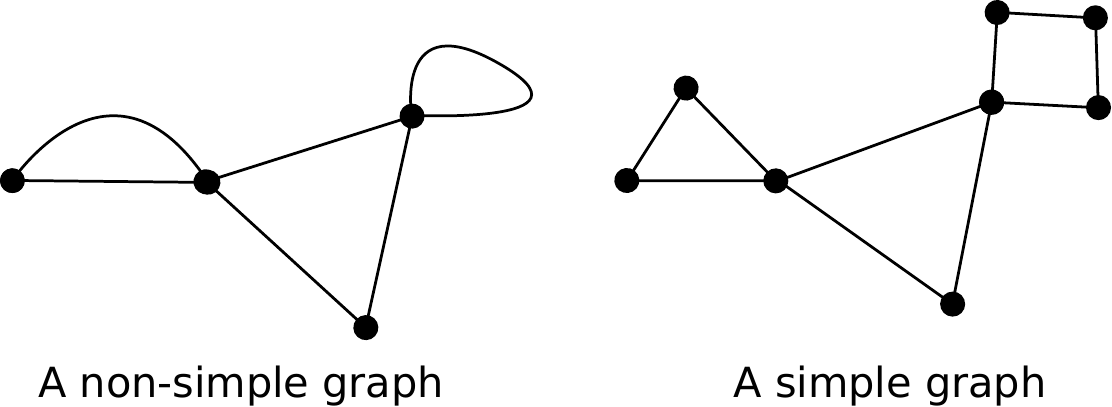}
	  \end{center}
	  \caption{An example of the differences between simple and non-simple graphs.}
	  \label{fig:simple-graph}
	\end{figure}
	
	\begin{example}
	  Figure~\ref{fig:simple-graph} shows two graphs. The left-hand graph demonstrates a non-simple graph since it contains a loop and a double edge, although only one of these features would be enough to prevent the graph from being simple. The right-hand graph is simple since it contains no loops and no multiple edges.
	\end{example}
	
	We will also be interested in how many edges are attached to each vertex:
	\begin{definition}
		For a vertex $v$ the \textbf{degree} of $v$ (denoted $\deg(v)$) is defined as the number of edges which $v$ is incident with, with the exception of a loop where the edge which forms the loop is counted twice.
	\end{definition}
	
	\begin{figure}[htb]
	  \begin{center}
		\includegraphics[scale=0.4]{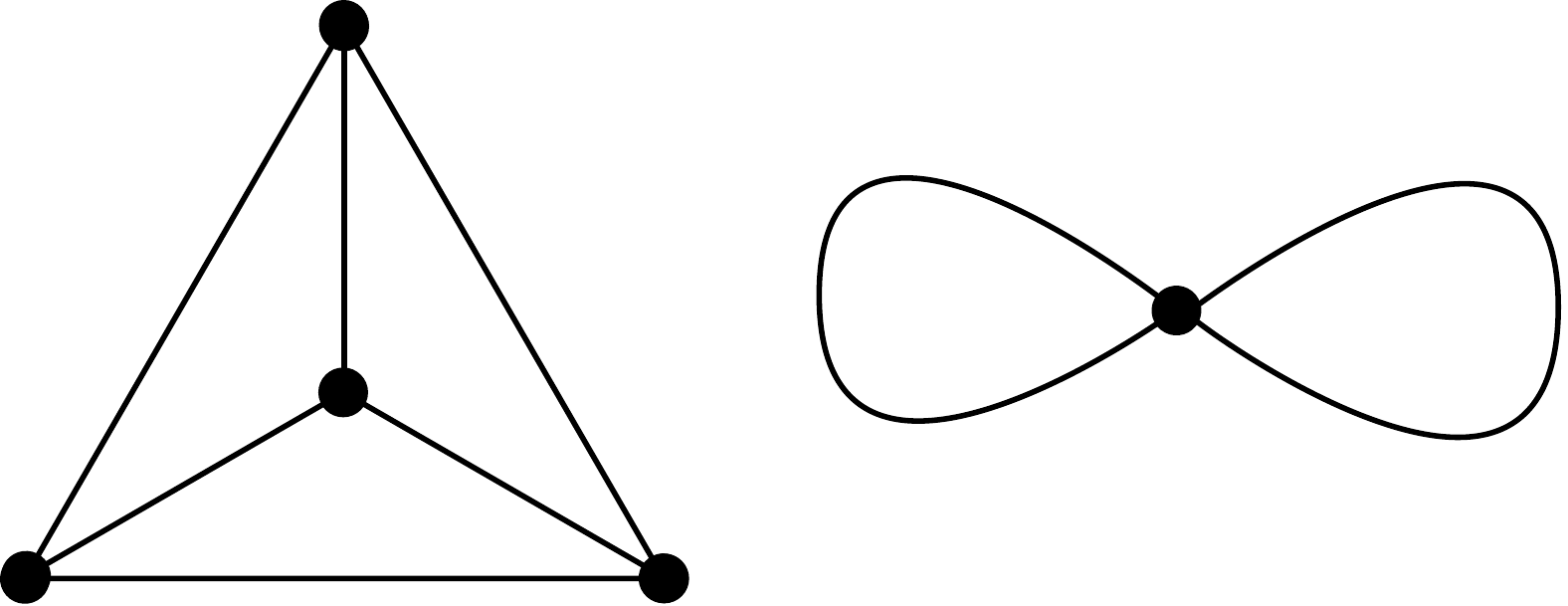}
	  \end{center}
	  \caption{In the left-hand graph all vertices have degree 3 and on the right-graph there is a single vertex of degree 4.}
	  \label{fig:vertex-degree}
	\end{figure}
	
	\begin{definition} \label{def:trivial-vertex}
		In a graph $G$ a vertex $v$ is called a \textbf{trivial vertex} if it has exactly two edges incident to it and has degree 2.
	\end{definition}
	
	\begin{figure}[htb]
	  \begin{center}
		\includegraphics[scale=0.7]{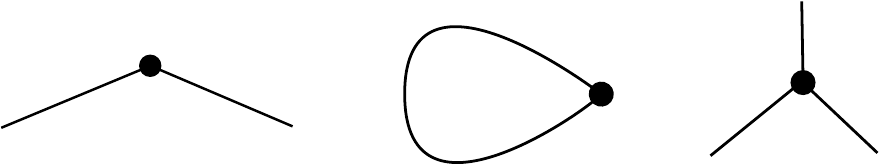}
	  \end{center}
	  \caption{Trivial vertex, non-trivial vertex and non-trivial vertex respectively.}
	  \label{fig:trivial-vertices}
	\end{figure}
	
	Recall that a homeomorphism is a bijective bicontinuous function and that if a homeomorphism exists between two topological objects then they are said to be homeomorphic. Thus we will consider a topological graph which is a way of thinking about a graph as a topological object, namely we let each edge be homeomorphic to an arc and vertices as the endpoints of arcs. There are two immediate consequences when thinking about topological graphs:
	\begin{remark}
	  Vertices are single points and edges are just arcs which are (infinite) collections of single points, so in reality an embedding would not have large circles for the vertices as is the case in our pictures. I will however persist with using such representations because it provides a much clearer picture.
	\end{remark}
	\begin{remark} \label{rem:no-trivial-vertices}
	  We do not have trivial vertices on a topological graph since a vertex incident with only two edges is equivalent to a single edge which is the union of the two edges and the trivial vertex.
	\end{remark}

\section{Complete Graphs} \label{sec:complete-graphs}
	We will define a common set of graphs called complete graphs:
	\begin{definition}
		A \textbf{complete graph on $n$ vertices} (denoted $K_n$) is a graph consisting of $n$ vertices such that every vertex is connected to every other vertex by a single edge.
	\end{definition}
	
	\begin{definition}
		A \textbf{complete bipartite graph} is a graph whose vertex set is partitioned into two disjoint sets such that every vertex in one set is joined to every vertex in the other set by a single edge, if the sizes of the two sets are $m, n$ then the graph is denoted $K_{m,n}$.
	\end{definition}
	
	\begin{figure}[htb]
	  \begin{center}
			\includegraphics[scale=0.7]{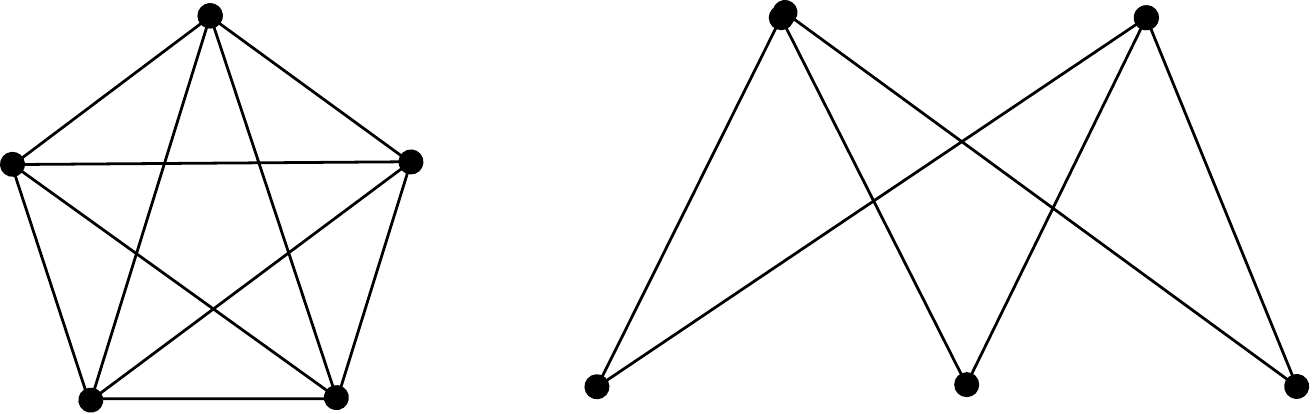}
	  \end{center}
	  \caption{$K_5$ (complete graph) and $K_{3,2}$ (complete bipartite graph).}
	  \label{fig:complete-and-bipartite}
	\end{figure}
  
  \chapter{Surfaces} \label{chp:surfaces}

We now introduce the concept of a surface onto which we can draw our maps. We will initially introduce the formal definition of a surface and then move onto ways to distinguishing between surfaces and calculating properties for each surface.

\section{Surfaces and Embeddings}
  We start by introducing an abstract surface which informally is a 2-dimensional object which looks everywhere locally like the plane:
  \begin{definition}
    A Hausdorff topological space $X$ is called a \textbf{surface} (without boundary) if at every point in $a \in X$ there exists an open neighbourhood which is homeomorphic to an open disc $\{(x,y) \in \mathbb{R}^2 | x^2+y^2 < 1\}$.
  \end{definition}
  \begin{remark}
    If a surface is closed and bounded then it is called a compact surface and for the purposes of this report we will only be considering compact surfaces thus from now on ``surface" will refer to a ``compact surface".
  \end{remark}
  
  We can then discuss the concept of an embedding of a graph on a surface $S$ as found in \cite[Section 1.3]{BWGT} recalling that a homeomorphism is a bijective bicontinuous function:
  \begin{definition}
		An \textbf{embedding} of a graph $G$ on a surface $S$  is a continuous function $f:G \to S$ taking $G$ homeomorphically to its image $f(G)$. \newline
	\end{definition}
	Intuitively an embedding is a drawing of a graph on a surface in which no edges cross.
	
	\begin{figure}[htb]
	  \begin{center}
	  	\includegraphics[scale=0.6]{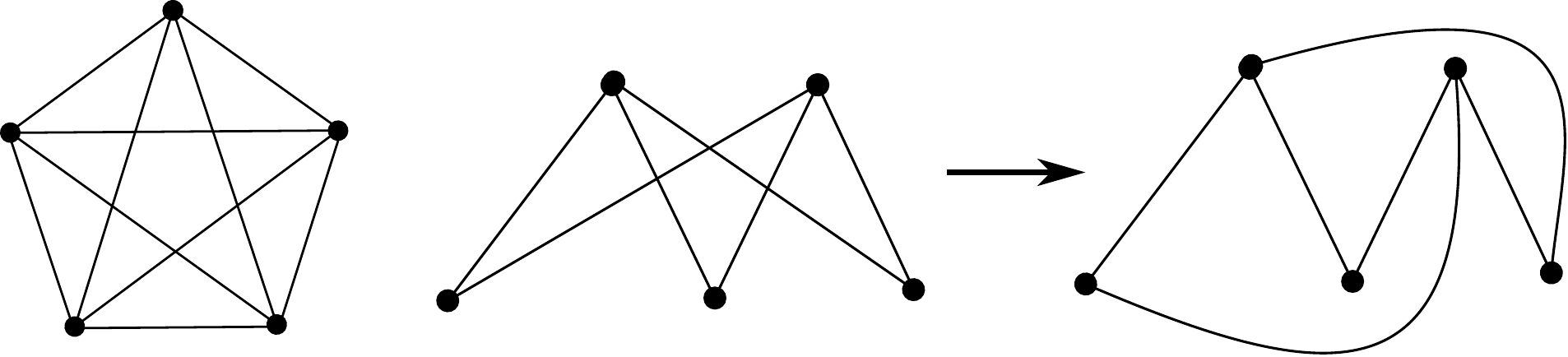}
	  \end{center}
	  \caption{$K_5$ and $K_{3,2}$ which cannot and can be embedded in the plane respectively.}
	  \label{fig:embeddable-graphs}
	\end{figure}
	
	\begin{example}
	  Figure \ref{fig:embeddable-graphs} shows two graphs we have seen before, $K_5$ and $K_{3,2}$. We can see that the natural representation of $K_{3,2}$ is not the image of an embedding in the plane since there are edges which cross, however we use the fact that we can deform these edges so that we obtain a representation where there are no edges which cross, thus $K_{3,2}$ is embeddable in the plane. We see also that the natural representation of $K_5$ is also not the image of an embedding but it is harder to see how to deform the edges to obtain one and in fact we shall see later on that no such embedding exists.
	\end{example}
	
	One question that springs to mind is whether all maps can be embedded into a surface - the answer is quite clearly no because for more complicated graphs with many edges it will become impossible to draw them on a simple surface without them crossing. However we will be interested largely in graphs which \textit{can} be embedded in a surface:
	\begin{definition}
	  A graph $G$ is called \textbf{planar} if it can be embedded in the plane.
	\end{definition}
	\begin{remark}
	  Embedding a graph in the plane is equivalent to embedding a graph on on the sphere.
	\end{remark}
	
	We can also now define the natural counterpart to vertices and edges which is the space on the surface enclosed by them:
	\begin{definition}
		The connected components of the complement of the image of an embedded graph are called \textbf{countries}.
	\end{definition}

\section{Euler Characteristic}
  We can now introduce a different way of thinking about a surface in a way that will make it easier for us to calculate the properties of the surface and compare it to other surfaces:
  \begin{definition}
  	A \textbf{combinatorial surface} is a surface $S$ with a fixed embedded graph $G$ such that each connected component in $S \setminus G$ is a topological open disk with at least 3 sides. A combinatorial surface can be denoted as the pair $(S,G)$.
  \end{definition}
  
  We can then introduce a concept called the Euler characteristic which is a way of assigning to any surface an integer which can then be used to distinguish between surfaces:
  \begin{definition} \label{def:combinatorial-surface}
    For a combinatorial surface $(S,G)$ we define the \textbf{Euler characteristic}, $\chi(S)$, to be:
    \begin{equation}
      \chi(S) = V - E + C
    \end{equation}
    Where $V, E, C$ are the number of vertices, edges and countries in $G$ respectively.
  \end{definition}
  
  We then provide the following theorem without proof:
  \begin{theorem}
    The Euler characteristic of a combinatorial surface, $(S, G)$ is independent of the choice of embedded graph, $G$.
  \end{theorem}

\section{Orientability}
  We next introduce the concept of orientability. We begin with an abstract definition of orientability:
  \begin{definition}
    A combinatorial surface $(S, G)$ is \textbf{orientable} if for any two connected components in $S \setminus G$ which share a common edge you can choose a compatible orientation. That is for each edge in the perimeter pick a consistent direction such that for the common edge the directions are opposite.
  \end{definition}
	
	\begin{figure}[htb]
	  \begin{center}
	  	\includegraphics[scale=0.6]{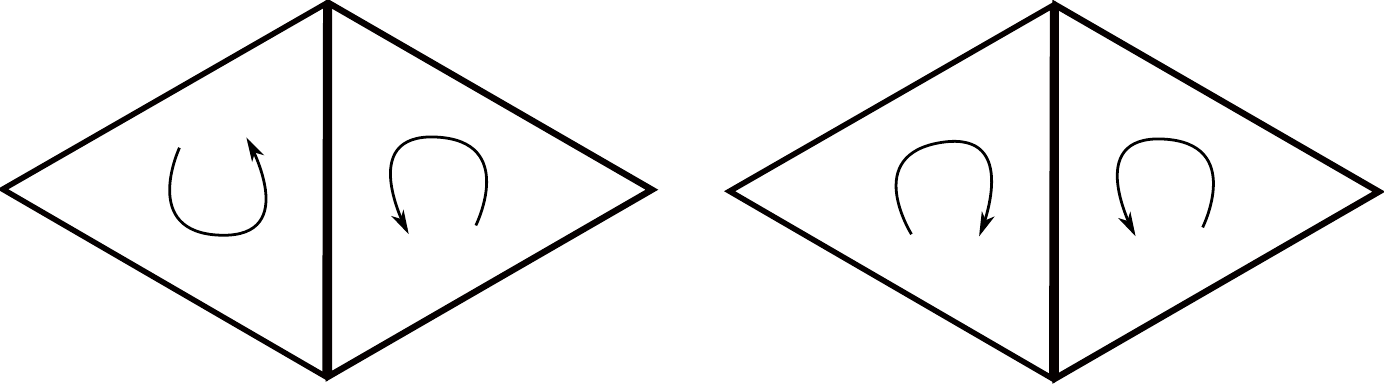}
	  \end{center}
	  \caption{An example of a choice of orientation which is acceptable (LHS) and not acceptable (RHS).}
	  \label{fig:orient-triangulation}
	\end{figure}
  
  Figure~\ref{fig:orient-triangulation} shows the two possibilities for orienting components that share a common edge with the one on the left being a compatible orientation whereas the one of the right is not a compatible orientation.
  \begin{remark}
    For orientable surfaces it will always be possible to orient them and for non-orientable surfaces it will never be possible.
  \end{remark}
  
\section{Classification of surfaces} \label{sec:classification-of-surfaces}
  We are aiming to find a complete topological invariant of surfaces, that is to say some property which satisfies two conditions:
  \begin{enumerate} [(i)]
    \item If any two surfaces have possess this property they are equivalent
    \item If one surface possesses the property and one doesn't then they are not equivalent
  \end{enumerate}
  This way we know exactly what surfaces we are dealing with at any point in time. For this we introduce the concept of ``combining" surfaces:
  \begin{definition}
    For two surfaces $S$ and $T$ we define the \textbf{connected sum} (denoted $S\#T$) as the surface obtained by removing a disc from both $S$ and $T$ and then ``gluing" them back together along the new boundaries.
  \end{definition}
  
  We then need to introduce a couple of special surfaces, the first of which is the torus, or as it is more commonly known a doughnut. It can be thought of in 3 ways:
  \begin{enumerate}
    \item A sphere with a handle attached
    \item A doughnut
    \item The quotient of a square
  \end{enumerate}
  The last one of these is the least intuitive but if you look at Figure~\ref{fig:making-torus-quotient} you can see the process of representing the torus as the quotient of a square. Essentially you start with a square where the sides have a direction and a label, then for sides with the same label you attach them so that the direction matches. So initially we attach to the two edges labelled A which gives us a hollow cylinder. We then attach the two sides labelled B which are now the ends of the cylinder by wrapping it round to make a doughnut shape and thus we have the torus.
	
  \begin{figure}[htb]
    \begin{center}
   	  \includegraphics[scale=0.5]{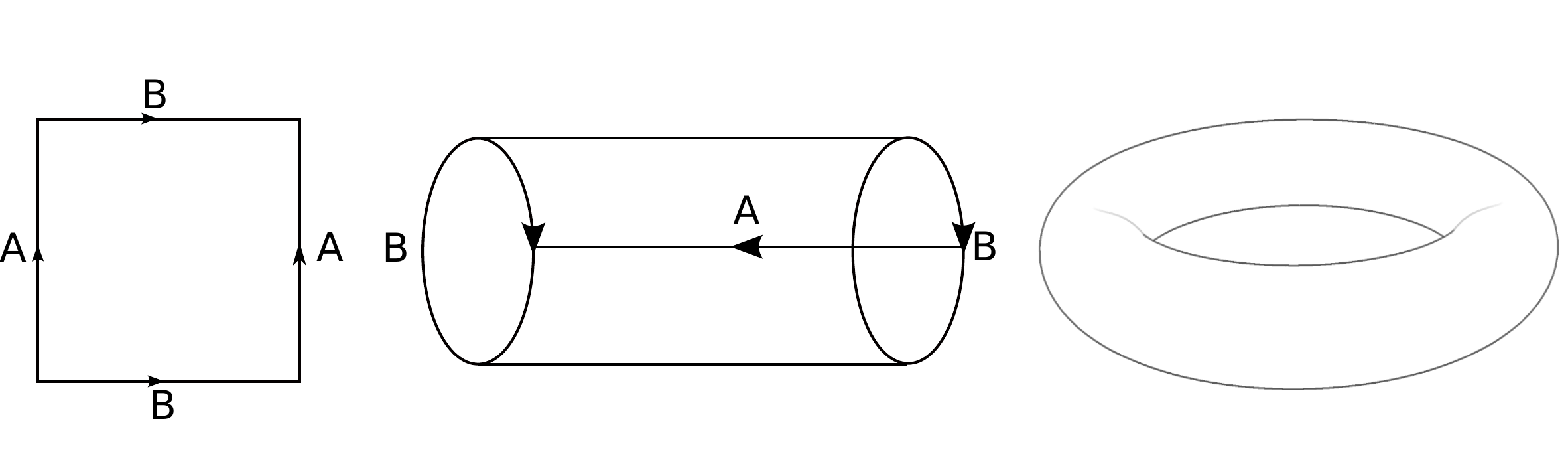}
	\end{center}
	\caption{Representing the torus as the quotient of a square.}
	\label{fig:making-torus-quotient}
  \end{figure}
  
  The second shape we are going to introduce is the projective plane, this is a non-orientable surface which cannot actually be drawn in 3 dimensions without self-intersecting (it can be drawn in 4-dimensions without self intersecting) but to give an idea of what it looks like we give two representations as shown in Figure~\ref{fig:projective-plane-representations}, the first is an attempt at a drawing in 3 dimensions and the second is representing it as the quotient of a square in the same way as we did with the torus, I have not included the process of actually making it since it can't be accurately drawn!
	
  \begin{figure}[htb]
    \begin{center}
   	  \includegraphics[scale=0.2]{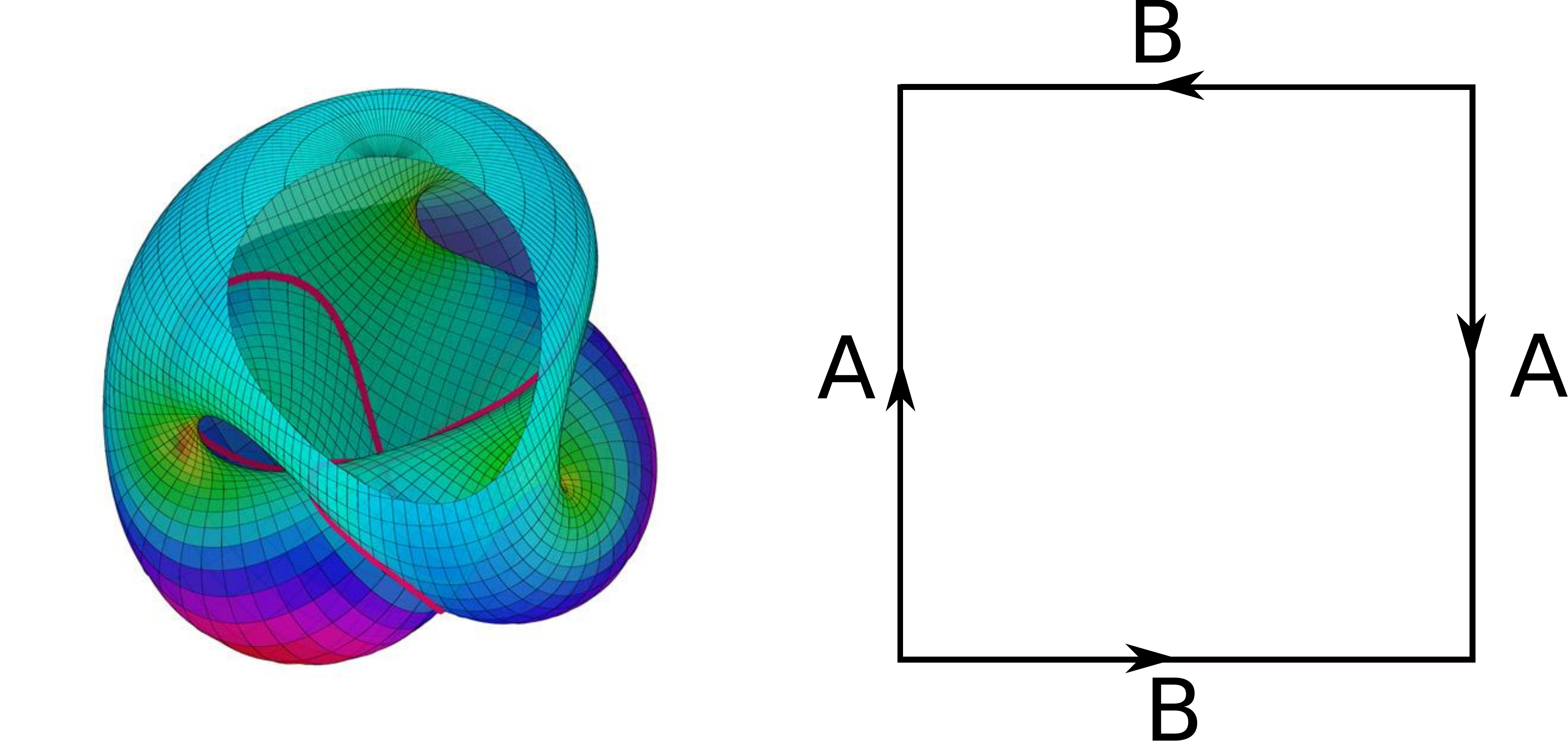}
	\end{center}
	\caption{Two representations of the projective plane. The left-hand image is taken from Surrey University \cite{surrey-maths}}
	\label{fig:projective-plane-representations}
  \end{figure}
  
  We then state (without proof) that all surfaces are equivalent to one of the following:
  \begin{enumerate}
    \item The sphere
    \item The connected sum of $g$ tori ($g \geq 1$)
    \item The connected sum of $n$ projective planes ($n \geq 1$)
  \end{enumerate}
  
  Furthermore we can define the Euler characteristic in another way by considering it as the sphere with tori or projective planes ``added", informally we start with the Euler characteristic of the sphere which is 2 and then for every tori we add we take 2 from the Euler characteristic and for every projective plane we take 1.
  \begin{remark}
    Whilst you can add say $g$ tori and $n$ projective planes to the sphere at the same time with both $n, g \geq 1$ the surface is homeomorphic to the surface obtained by adding $n+2g$ projective planes to the sphere. Thus we need not consider cases where we add both tori and projective planes to the sphere as separate cases from adding only projective planes.
  \end{remark}
  
  Thus you can get the following equation for the Euler characteristic:
  \begin{lemma}
    For a surface $S$ which is the connected sum of $g$ tori and $n$ projective planes we have that the Euler characteristic is given by:
    \begin{equation} \label{eqn:euler-char-genus}
      \chi(S) = 2 - 2g - n
    \end{equation}
  \end{lemma}
  
  \begin{remark}
    Throughout this report we will only be considering orientable surfaces. However since a lot of the results given later in this report also hold for non-orientable surfaces these cases will often be discussed in related literature.
  \end{remark}
  
  \begin{remark}
    The other important reason for this discussion is to convince ourselves that we are not missing any surfaces in our discussions. Or indeed to ensure that we are not duplicating any work by considering the same surface in two different ways.
  \end{remark}

  \chapter{Map Colouring} \label{chp:colouring}

We now have enough terminology to begin to define our maps in a more mathematical way, we begin by defining the concept of a map and move onto defining how to colour a map.

\section{Maps and Colouring} \label{sec:maps-intro}
  For us a map if just a more convenient and intuitive way to refer to the embedding of a graph, recall that we are only considering orientable, connected surfaces without boundary:

	\begin{definition} \label{def:map}
		A \textbf{map} is an embedding ($f:G \to S$) of a graph. We will specify ``map on a sphere" which shall refer to the embedding of a graph $G$ on a sphere.
	\end{definition}
	\begin{remark}
	  Two maps are considered equivalent if there exists a homeomorphism between their images.
	\end{remark}
	
	It is natural to think of a cartographer colouring each Country (in the typical sense, e.g. England) a different colour on a typical Map (also in the typical sense e.g. of say the world), or at least not colouring two neighbouring Countries the same colour. We will however be making one slight simplification from a general Map of the world, we will assume that each Country in the typical sense is the same as a country defined above - i.e. does not consists of ``colonies" or ``empires". For example the situation with Alaska and the USA (same country but not attached) will not be allowed. \newline
	
	We now formally define the concepts above:
	\begin{definition}
		Two countries are said to be \textbf{neighbouring countries} if they share a common edge.
	\end{definition}
	
	\begin{figure}[htb]
	  \begin{center}
			\includegraphics[scale=1]{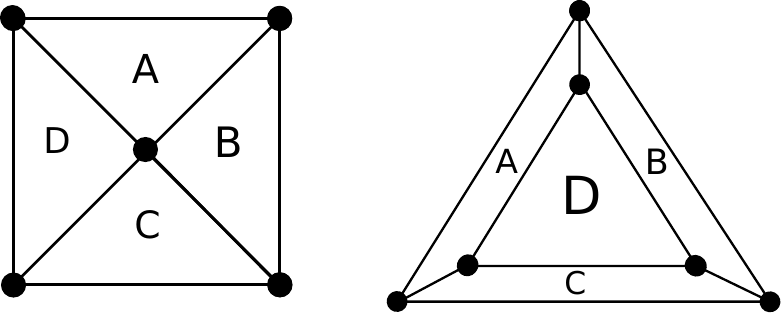}
	  \end{center}
	  \caption{Examples of different maps.}
	  \label{fig:neighbouring-example}
	\end{figure}
	
	\begin{example}
	  In Figure \ref{fig:neighbouring-example} we have two examples of maps. In the left-hand map country A neighbours countries B and D because they share a common edge but not country C since they only share a vertex. \\
	  In the right-hand map we see an example of where every country neighbours every other country.
	\end{example}

	\begin{definition}
		A \textbf{map colouring} consists of assigning to each country a colour (or equivalently a natural number) such that no country is assigned the same colour as any of its' neighbouring countries. \newline
	\end{definition}
	
	\begin{remark}
	  Note that one colour may be assigned to many different countries, providing no two of these countries are neighbouring. This does not however signify any relationship between these countries.
    \end{remark}
	
	\begin{figure}[htb]
	  \begin{center}
			\includegraphics[scale=1]{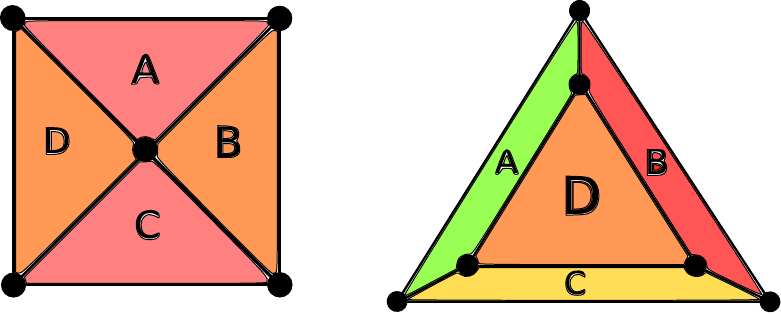}
	  \end{center}
	  \caption{An example of map colourings.}
	  \label{fig:colouring-example}
	\end{figure}

\section{Dual Graphs}
	When colouring maps, the important information that we are interested in is which countries neighbour which other countries. Unfortunately this isn’t encoded very well in the way we have described maps so far. To overcome this problem we introduce a new concept called a ``dual graph”. This is a means to convert a map back into a graph (different from the graph of which the map
is embedding). Informally, for each country you create a ``capital” and then for neighbouring countries you create a railway between capitals such that it doesn’t cross any other railway and has exactly one border crossing. To make this a little more formal we describe the process by which you can convert between a map and it’s dual graph, based on a definition by Fritsch \cite[Section 4.4]{fritsch}.

	For a map $M$ you can obtain the dual graph $G'$ via the following process:
	\begin{enumerate} [(i)]
	  \item For each country $c$ place a point (``capital") in the interior of $c$, this point becomes a vertex in $G'$
	  \item For neighbouring countries $c_1, c_2$ create an edge $r$ between the corresponding vertices of $G'$ such that $r$ intersect all the edges of $G'$ at a single point.
	\end{enumerate}
	
	\begin{figure}[htb]
	  \begin{center}
			\includegraphics[scale=0.7]{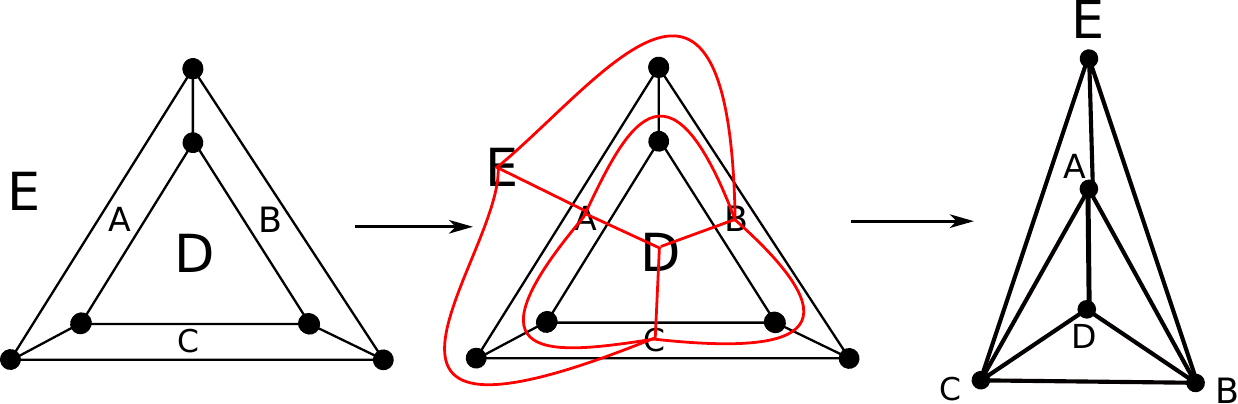}
	  \end{center}
	  \caption{Converting from a graph to its' dual graph.}
	  \label{fig:graph-to-dual}
	\end{figure}
	
	\begin{example}
	  Figure \ref{fig:graph-to-dual} shows how you can convert a map into its' dual graph by first making each country (note that we have countries A through D but also since this is an embedding in the plane we also have ``the rest" which also counts as a country, in this case E) into a vertex. Then for neighbouring countries join the corresponding vertices to be left with a graph like the one on the right.
	\end{example}
	
	\begin{remark} \label{rem:dual-graph-quirks}
	  There are a few points to notice about the construction of dual graphs (illustrated in Figure \ref{fig:dual-graph-quirks}):
	  \begin{enumerate} [(i)]
		  \item Vertices in $G'$ correspond to countries in $G$
		  \item There cannot exist loops in $G'$ since a country cannot neighbour itself
		  \item If two countries share more than one edge in common there still only exists one edge between the corresponding vertices in $G'$
	  \end{enumerate}
	\end{remark}
	
	\begin{figure}[htb]
	  \begin{center}
			\includegraphics[scale=0.4]{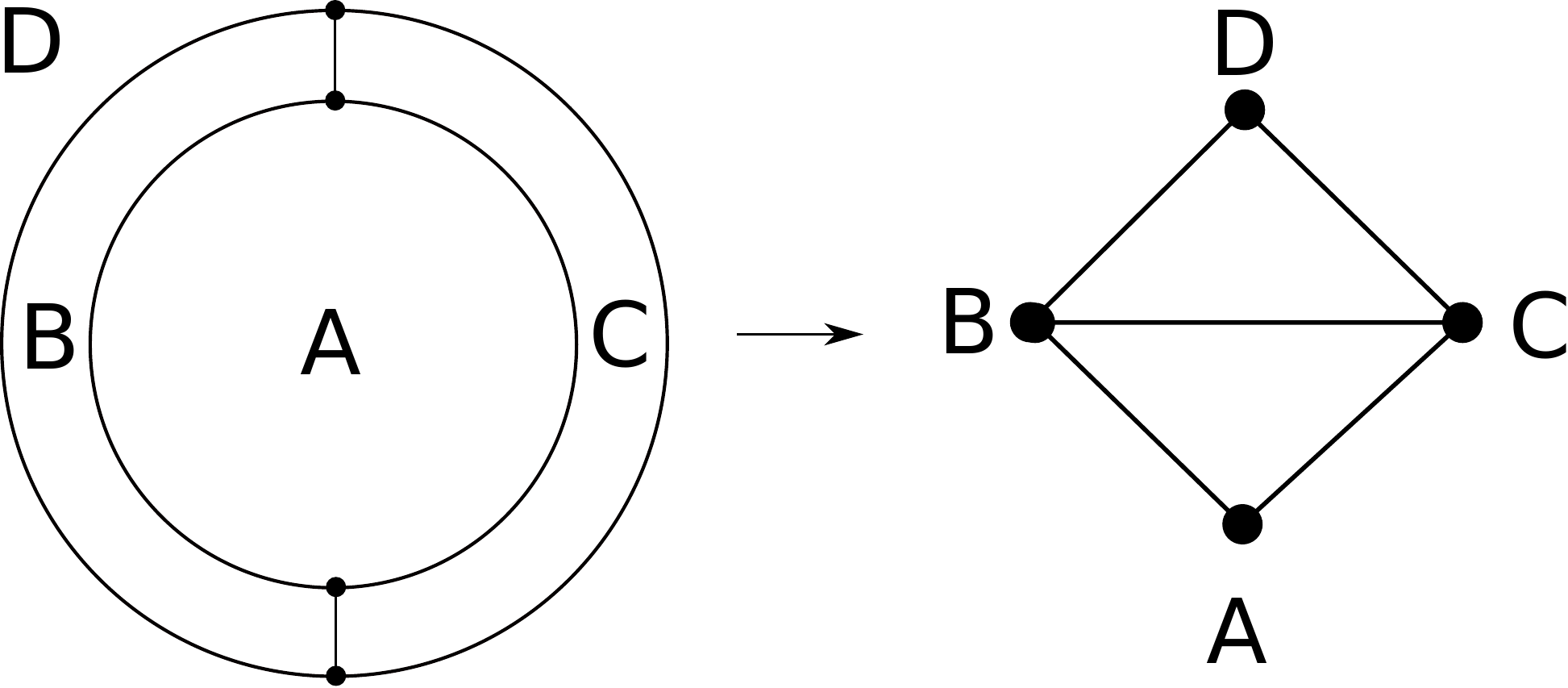}
	  \end{center}
	  \caption{An example of some of the quirks of dual graphs outlined in Remark \ref{rem:dual-graph-quirks}.}
	  \label{fig:dual-graph-quirks}
	\end{figure}
	
	We then state without proof (although it is reasonably intuitive to see from the construction of a dual graph) a very important theorem:
    \begin{theorem}
      For a graph, $M$, embedded on a surface, $S$, if you construct the dual graph of $M$, call it $G$, then $G$ can also be embedded on $S$.
    \end{theorem}
    
	One important result which we will use later in this report is that a dual graph is in fact by it's nature simple:
	\begin{lemma} \label{lem:dual-graph-simple}
	  For a map $M$ with dual graph $G'$ we have that $G'$ is simple.
	\end{lemma}
	\begin{proof}
	  We need to check that dual graphs contain no loops or multiple edges. Firstly, as mentioned in Remark \ref{rem:dual-graph-quirks} we join two vertices only by a single edge even if they have more than one common border thus we will never obtain multiple edges. With regards loops this would only arise if a country bordered itself but this is absurd so that can never occur either. So therefore a dual graph can contain neither multiple edges or loops and is therefore simple.
	\end{proof}
	
	We then naturally get:
	\begin{lemma} \label{lem:maps-equal-dual-graphs}
		For a map $M$ with dual graph $G'$, a map colouring on $M$ is equivalent to colouring the vertices of $G'$ such that any two vertices that joined by an edge do not share the same colour.
	\end{lemma}
	\begin{proof}
	  Since vertices in $G$ represent countries in $M$ we have parity there. Also since we create an edge between two vertices in $G$ whenever two countries are neighbours in $M$ we have parity there also. Thus we have an equivalent problem.
    \end{proof}
    
	\begin{remark} \label{rem:map-dual-graph-parity}
	  We also have a similar reverse correspondance that if we have an arbitrary graph $G'$ which we can embed in a surface $S$ then we can create a map $M$ on $S$ such that the dual graph of $M$ is exactly $G'$. Thus we have complete parity between graphs and maps, i.e. if we could find a planar a graph which required 5 colours to colour properly then we have shown that there exists at least one map on the plane which requires 5 colours to colour properly. Note also that dual graph encode much more neatly the important information since each edge in the dual graph represents an important border that we need to consider when colouring.
    \end{remark}

\section{Colouring Graphs}
  Now that we have established that colouring graphs is equivalent to colouring maps we look at a few results about graphs and introduce some new terminology which we borrow from David Gay \cite[Chapter 12]{gay}.
  \begin{definition}
    For a graph $G$, we define the chromatic number, $c(G)$, to be the number of colours required to colour $G$ properly.
  \end{definition}
  
  When thinking about colouring a graph $G$ it is useful to only have to consider vertices and edges which have a direct impact on $c(G)$ so now we introduce a new concept:
  \begin{definition}
    If removing any edges and/or vertices from a graph $G$ always results in a graph $G'$ such that $c(G') < c(G)$ then $G$ is called a \textbf{critical graph}.
  \end{definition}
  
  We can now then produce some Lemmas about critical graphs (again taken from Gay, \cite[Chapter 12]{gay}) which we will be using in later chapters to help prove some more significant results:
  \begin{lemma} \label{lem:critical-graph-vertex-degree}
    If a graph $K$ is critical with chromatic number $c(K)$ then every vertex of $K$ has degree at least $c(K)-1$.
  \end{lemma} 
  \begin{proof}
    We use a proof by contradiction as found in Gay, \cite[Chapter 12]{gay}. \newline
    So suppose that there is a vertex $v \in K$ such that $\deg(v) < c(K) - 1$. Then from the graph $K$ remove $v$ and all the edges attached to $v$ to obtain a new graph $K'$, then because $K$ was critical we have that $K'$ can be coloured in at most $c(K) - 1$ colours, so produce such a colouring. \newline
    We then colour $K$ in exactly the same way as we coloured $K'$, then we only have $v$ left to colour but since it is joined to at most $c(K) - 2$ other vertices there is guaranteed to be a colour spare so no new colours are needed. This means that $K$ is has been coloured in $c(K) - 1$ colours which is a contradiction since the chromatic number of $K$ is $c(K)$.
  \end{proof}
  
  \begin{lemma} \label{lem:simple-graph-three-sides}
    If a graph $G \subset S$ is simple then:
    \begin{displaymath}
      3C \leq 2E
    \end{displaymath}
    where $C$ and $E$ are the number countries and edges in (the embedding of) $G$ respectively.
  \end{lemma}
  \begin{proof}
    It is sufficient to show that each country in the embedding of $G$ is surrounded by at least 3 edges. We have only two cases to check, namely that a country is surrounded by 1 or 2 countries.
	
	\begin{figure}[htb]
	  \begin{center}
		\includegraphics[scale=0.7]{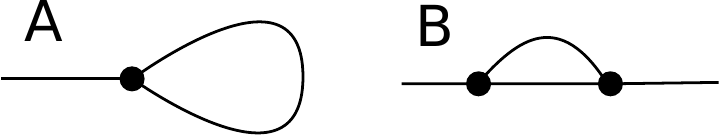}
	  \end{center}
	  \caption{The two cases that require checking for Lemma \ref{lem:simple-graph-three-sides}.}
	  \label{fig:critical-lemma-cases}
	\end{figure}
    
    Figure \ref{fig:critical-lemma-cases} demonstrates the two cases we need to check but you can see in both one of the requirements of simple graph is broken, in A we have a loop and in B we have a multiple edge. Thus if $G$ is a simple graph every country in its' embedding must be surrounded by at least three edges.
  \end{proof}
  
  We are now in a position to prove that $K_5$ isn't planar or equivalently:
  \begin{theorem} \label{thm:K5-not-planar}
    $K_5$ cannot be embedded on the sphere.
  \end{theorem}
  \begin{proof}
    We use proof by contradiction, so assume that $K_5$ can be embedded on the sphere. Then by Definition \ref{def:combinatorial-surface} we can calculate the Euler characteristic via an embedding of $K_5$ on the sphere. Giving us:
    \begin{displaymath}
      \chi(S) = C - E + V \Rightarrow 2 = C - \frac{5 \cdot 4}{2} + 5 \Rightarrow C = 7
    \end{displaymath}
    But then since $K_5$ is simple by construction we can apply Lemma~\ref{lem:simple-graph-three-sides} to get:
    \begin{displaymath}
      3C = 21 < 20 = 2E
    \end{displaymath}
    which is clearly a contradiction and so $K_5$ cannot be embedded on the sphere.
  \end{proof}
  
  \begin{lemma} \label{lem:critical-equals-simple}
    If a graph $K$ is critical then it is also simple.
  \end{lemma}
  \begin{proof}
    So we need to prove that we cannot have loops or multiple edges. For this we use again Figure~\ref{fig:critical-lemma-cases} which contains the things we need to check. However we can see that in A we can simply remove the loop which will not change the chromatic number and in B we can remove one (but not both) of the multiple edges. Thus we can say a critical graph contains no loops or mutiple edges and is therefore simple.
  \end{proof}
  
  \begin{corollary} \label{cor:critical-graph-three-sides}
    If a graph $K \subset S$ is critical then:
    \begin{displaymath}
      3C \leq 2E
    \end{displaymath}
    where $C$ and $E$ are the number countries and edges in (the embedding of) $K$ respectively.
  \end{corollary}
  \begin{proof}
    This follows immediately as a result of Lemma~\ref{lem:simple-graph-three-sides} and Lemma~\ref{lem:critical-equals-simple}.
  \end{proof}
  
  \chapter{Heawood's Conjecture for higher genus surfaces} \label{chp:heawood-higher-genus}

We now begin to look more closely at some of the problems that Heawood introduced in his paper of 1890. In this chapter we will consider maps on orientable surfaces of higher genus. We will look at an upper bound that Heawood gave for the number of colours needed to colour any map (on an orientable surface) which depended only on the surface’s genus. The ideas outlined below follow a similar line to that given in Gay \cite[Chapter 12]{gay}.

\section{Initial results}
  Before we actually state and prove Heawood's upper bound we will first introduce and prove some smaller Lemmas which will be useful in the proof and also later in the report. We start with some definitions to establish some notation starting with a concept similar to one we have already introduced for graphs:
  \begin{definition}
    For maps and surfaces we define the \textbf{chromatic number} as follows:
    \begin{enumerate}
      \item The number of colours, $c(M)$ needed to colour a map properly is called the \textbf{chromatic number} of $M$
      \item If $S$ is a surface then $c(S)$, the \textbf{chromatic number} number of $S$, is the maximum of all $c(M)$ where $M$ is a map on $S$
    \end{enumerate}
  \end{definition}
  
  Intuitively, $c(S)$ is an upper bound for the number of colours required for any map that can be drawn on the surface $S$.

  We also introduce some notation for any graphs:
  \begin{definition} \label{def:average-degree}
    For any graph $G$ we have that the average degree of the vertices, $A(G)$ is:
    \begin{displaymath}
      A(G) = \frac{1}{V} \sum_{i=1}^{V} i V_i
    \end{displaymath}
    where $V$ is the number of vertices in $G$ and $V_i$ is the number of vertices of degree $i$ in $G$.
  \end{definition}
  
  We can then produce a series of Lemmas which we will use later:
  \begin{lemma} \label{lem:edge-count-vertices}
    For any graph $G$ we have the following result:
    \begin{displaymath}
      2E = \sum_{i=1}^{V} i V_i
    \end{displaymath}
  \end{lemma}
  \begin{proof}
    It is obvious that each vertex of degree $i$ has $i$ edges incident to it so if therefore we can count the number of edges incident to each vertex by: $\sum_{i=1}^{V} i V_i$. However each edge is incident to two vertices so this actually double counts all the edges hence we get:
    \begin{displaymath}
      2E = \sum_{i=1}^{V} i V_i
    \end{displaymath}
    as required.
  \end{proof}
  
  \begin{lemma} \label{lem:bound-chromatic-edge-vertices}
    For a critical graph $K$ we have that:
    \begin{displaymath}
      \frac{2E}{V} \geq c(K) - 1
    \end{displaymath}
    where $E$ and $V$ are the number of edges and vertices in $K$ respectively.
  \end{lemma}
  \begin{proof}
    It is clear to see that we can combine Definition~\ref{def:average-degree} and Lemma~\ref{lem:edge-count-vertices} to get the following:
    \begin{displaymath}
      A(K) = \frac{2E}{V}
    \end{displaymath}
    Then Lemma~\ref{lem:critical-graph-vertex-degree} tells us that every vertex has order at least $c(K) - 1$ thus the average degree must be at least $c(K) - 1$ giving:
    \begin{displaymath}
      \frac{2E}{V} = A(K) \geq c(K) - 1
    \end{displaymath}
    as required.
  \end{proof}
  
  \begin{lemma} \label{lem:bound-edges-vertices-euler}
    For a simple graph $G$ embedded in a surface $S$ we have:
    \begin{displaymath}
      2E \leq 6V - 6 \chi(S)
    \end{displaymath}
    where $E$ and $V$ are the number of edges and vertices in $G$ respectively.
  \end{lemma}
  \begin{proof}
    We have that the Euler characteristic of $S$ is given by :
    \begin{displaymath}
      \chi(S) = V - E + C
    \end{displaymath}
    Then combining this with Lemma~\ref{cor:critical-graph-three-sides} we get the following:
    \begin{displaymath}
      \begin{aligned}
        3\chi(S) &= 3V - 3E + 3C \\
        &\leq 3V - 3E + 2E \\
        &= 3V - E \\
        \Rightarrow 2E &\leq 6V - 6\chi(S)
      \end{aligned}
    \end{displaymath}
    as required.
  \end{proof}
  
  \begin{lemma} \label{lem:bound-chromatic-vertices-euler}
    For a critical graph $K$ embedded in a surface $S$ we have:
    \begin{displaymath}
      c(K) - 1 \leq 6 - 6 \frac{\chi(S)}{V}
    \end{displaymath}
    where $V$ is the number of vertices in $K$.
  \end{lemma}
  \begin{proof}
    We have from Lemma~\ref{lem:bound-edges-vertices-euler} that $6V - 6 \chi(S) \geq 2E$ and since $V \neq 0$ we can divide through by $V$ giving:
    \begin{displaymath}
      6 - 6 \frac{\chi(S)}{V} \geq \frac{2E}{V}
    \end{displaymath}
    Then we can use Lemma~\ref{lem:bound-chromatic-edge-vertices} to extend this giving:
    \begin{displaymath}
      6 - 6 \frac{\chi(S)}{V} \geq \frac{2E}{V} \geq c(K) - 1
    \end{displaymath}
    as required.
  \end{proof}
  
  \begin{remark}
    It is important to note at this point is that all of the above holds for any critical graph embedded on any surface $S$ whereas the theorem we prove below only holds for surfaces with $\chi(S) \leq 0$.
  \end{remark}
    
\section{Heawood's Inequality}
  We now move towards the main aim of the chapter and that is to state and prove Heawood's upper bound on the chromatic number of surfaces:
  \begin{theorem} \label{thm:heawood-upper-normal}
    For an orientable, connected, compact surface without boundary, $S$, with Euler characteristic $\chi(S) \leq 0$ we have the following inequality:
    \begin{equation} \label{eqn:heawood-upper-normal}
      c(S) \leq \lfloor \frac{7 + \sqrt{49 - 24\chi(S)}}{2} \rfloor
    \end{equation}
    where $\lfloor x \rfloor$ is the floor function and is equal to the largest integer less than or equal to $x$.
  \end{theorem}
  \begin{proof}
    For this proof we consider a map $M$ on the surface $S$, then if we can provide an upper bound on the chromatic number, $c(M)$, for a general map $M$ we can also provide an upper bound on $c(S)$.
    
    Furthermore, following on from Lemma~\ref{lem:maps-equal-dual-graphs} we will consider the dual graph $G'$ of a general map $M$ and then produce a critical graph $K$ from this graph $G'$ so that $c(M) = c(G') = c(K)$ and it is this $K$ that we will be working with.
    
    Our basic strategy will be to find relationships between the different components of $G'$ and then to eliminate variables until we are left with an equation involving only $c(K)$ and $\chi(S)$.
    
    We start with the result for a general critical graph $K$ provided by Lemma~\ref{lem:bound-chromatic-vertices-euler} which states:
    \begin{displaymath}
      c(K) - 1 \leq 6 - 6 \frac{\chi(S)}{V}
    \end{displaymath}
    We then note that $c(K) \leq V$ since you cannot possibly need more colours than there are vertices! Combining this with the assumption that $\chi(S) \leq 0$ we get:
    \begin{displaymath}
      -6 \frac{\chi(S)}{V} \geq 0 \Rightarrow -6 \frac{\chi(S)}{V} \leq -6 \frac{\chi(S)}{c(K)}
    \end{displaymath}
    So combining both of the above we get:
    \begin{displaymath}
      \begin{aligned}
        c(K) - 1 &\leq 6 - 6\frac{\chi(S)}{V} \\
        &\leq 6 - 6\frac{\chi(S)}{c(K)}
      \end{aligned}
    \end{displaymath}
    So we have now achieved our aim of getting an equation involving only $c(K)$ and $\chi(S)$, so now we can play around with this to isolate $c(K)$:
    \begin{displaymath}
      \begin{aligned}
        &c(K) - 1 \leq 6 - 6\frac{\chi(S)}{c(K)} \\
        &\Rightarrow c(K) - 7 + 6\frac{\chi(S)}{c(K)} \leq 0 \\
        &\Rightarrow c(K)^2 - 7c(K) + 6 \chi(S) \leq 0 \\
        &\Rightarrow \left(c(K) - \frac{7 + \sqrt{49-24\chi(S)}}{2}\right)\left(c(K) - \frac{7 - \sqrt{49-24\chi(S)}}{2}\right) \leq 0
      \end{aligned}
    \end{displaymath}
    At this point note that since $\chi(S) \leq 0$ we get the following:
    \begin{displaymath}
      \begin{aligned}
        &\chi(S) \leq 0 \\
        &\Rightarrow \sqrt{49 - 24\chi(S)} \geq 7 \\
        &\Rightarrow 7 - \sqrt{49 - 24\chi(S)} \leq 0 \\
        &\Rightarrow c(K) - \frac{7 - \sqrt{49-24\chi(S)}}{2} \geq 0
      \end{aligned}
    \end{displaymath}
    Thus the whole of the right-hand factor is positive, this implies that the whole of the left hand factor must be negative yielding:
    \begin{equation}
      c(K) \leq \frac{7 + \sqrt{49-24\chi(S)}}{2}
    \end{equation}
    Then since we know $c(K) \in \mathbb{Z}$ we can say that \eqref{eqn:heawood-upper-normal} also holds and thus we are done.
  \end{proof}    

\section{Heawood's Conjecture} \label{sec:heawood-conjecture-normal}
  Heawood proved this result in his paper of 1890 \cite{heawood} and at the same time he conjectured that this inequality was sharp, i.e. that for orientable surfaces $S$ with $\chi(S) \leq 0$ you have:
  \begin{equation} \label{eqn:heawood-conjecture-normal}
    c(S) = \lfloor \frac{7 + \sqrt{49 - 24\chi(S)}}{2} \rfloor
  \end{equation}
  This later become known as the ``Map Colour Theorem".
  
  Although Heawood didn't manage to prove this result he did suggest a way to attack the problem, namely if you can create a map on a particular surface such that you have $n$ mutually neighbouring regions then you must require at least $n$ colours to colour the map properly. Thus if you can construct a map where $n$ is equal to the upper bound for $c(S)$ then you have shown that the upper bound is in fact an equality. Heawood even went as far as producing such a map for the torus (see Figure~\ref{fig:heawood-torus-construction}) but this was sufficiently irregular that he was unable to generalise it.
  
  \begin{figure}[htb]
	\begin{center}
      \includegraphics[scale=0.8]{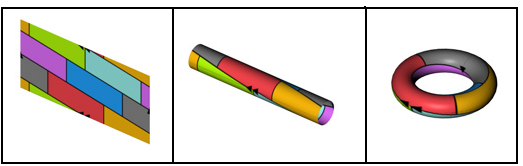}
	\end{center}
	\caption{A representation of the map Heawood constructed to prove his conjecture for the Torus.}
	\label{fig:heawood-torus-construction}
  \end{figure}
	
	The work to actually prove Heawood's conjecture began with Heffter a few years after Heawood's paper was published and he began in the manner to which Heawood suggested, however he noticed it was similar to a different problem called the $n$-cities problem which Heffter managed to solve for some cases. It wasn't until Dirac came along in 1952 though that it was proved that the solution to the $n$-cities problem would be equivalent to solving the Map Colour Theorem, but now that it was, work began to continue Heffter's work to extend to all cases. In 1968 Ringel and Youngs managed to cover all the cases that Heffter didn't quite manage and thus proved the Map Colour Theorem, details of the proof can be found in the book by Ringel \cite{ringel}.

\section{Four Colour Problem}
  For completeness we mention here some problems which are related to or came out of the Map Colour Theorem. The first and main partner to the Map Colour Theorem is the more famous ``Four Colour Theorem" which states that four colours are sufficient to colour any map on the sphere. Since the proof of the Map Colour Theorem requires that $\chi(S) \leq 0$ it clearly does not cover the case of the sphere since $\chi(S^2) = 2$, the Four Colour Theorem was however proved by Appel and Haken \cite{appel-haken} in its' own right in 1976 (note that is was actually solved 8 years \emph{after} the Map Colour Theorem!) and its' proof was controversial in that it required heavily on the use of computers to check properties of some 1936 maps. Doubts were initially quite high because the proof could not be feasibly checked by hand although it gradually gained wider acceptance; a simpler proof without the use of computers has yet to be found.
  
\section{Non-orientable surfaces}
  The initial proof on the upper bound and the subsequent conjecture by Heawood all focused on orientable surfaces, however in 1910 Tietze made an equivalent conjecture for non-orientable surfaces and surprisingly the formula is exactly the same. In fact it turned out that non-orientable surfaces were much easier than their orientable counterparts and consequently Ringel managed to prove the Map Colour Theorem for non-orientable surfaces in 1954 for all surfaces except the Klein bottle. The Map Colour Theorem states that $c(K) = 7$ where $K$ is the Klein bottle, but in fact Franklin proved in 1934 that $c(K) = 6$, this is the only case for which the Map Colour Theorem does not hold for non-orientable surfaces.
  
  \chapter{Empire maps on the plane} \label{chp:empire-plane}
When introducing maps (Section \ref{sec:maps-intro}) we placed on them the restriction that you cannot have two countries as part of the same empire (e.g. Alaska and the USA). In this chapter we move away from that restriction and towards another extension outlined by Heawood which covers the cases where a map can contain empires and colonies, i.e. two distinct countries which belong to a single empire and should therefore be coloured the same colour. Throughout this chapter we loosely follow the terminology and arguments as produced in the paper by Joan Hutchinson \cite{hutchinson}, but the terminology is slightly different.

\section{Introduction to Empire Maps} \label{sec:intro-empire-maps}
	As ever we need to formally set up some notation and describe exactly what we mean by the new restrictions, firstly we will introduce the the concept of an empire map:
	\begin{definition} \label{def:empire-map}
	  An \textbf{empire map} is a triple, ($M$, $A$, $g$), where:
	  \begin{itemize}
	    \item $M$ is a standard map (c.f. Definition \ref{def:map})
	    \item $A$ is a non-empty set of empires
	    \item $g$ is a surjective function which associates to each country in (the image of) $M$ an empire from $A$
	  \end{itemize}
	  for two countries, $\alpha, \beta$, if $g(\alpha) = g(\beta)$, then $\alpha, \beta$ are said to be part of the same \textbf{empire}.
	\end{definition}
	
	This is fine, however we would like to place some restrictions on how complicated our empire maps can be, or more specifically, how many countries we can have in each $m$-pire.
	\begin{definition} \label{def:m-pire-map}
	  For an empire map, $(M, A, g)$, if each empire contains at most $m$ countries then it is said to be an $m$-pire map.
	\end{definition}
	
	\begin{figure}[htb]
	  \begin{center}
        \includegraphics[height=0.15\textheight]{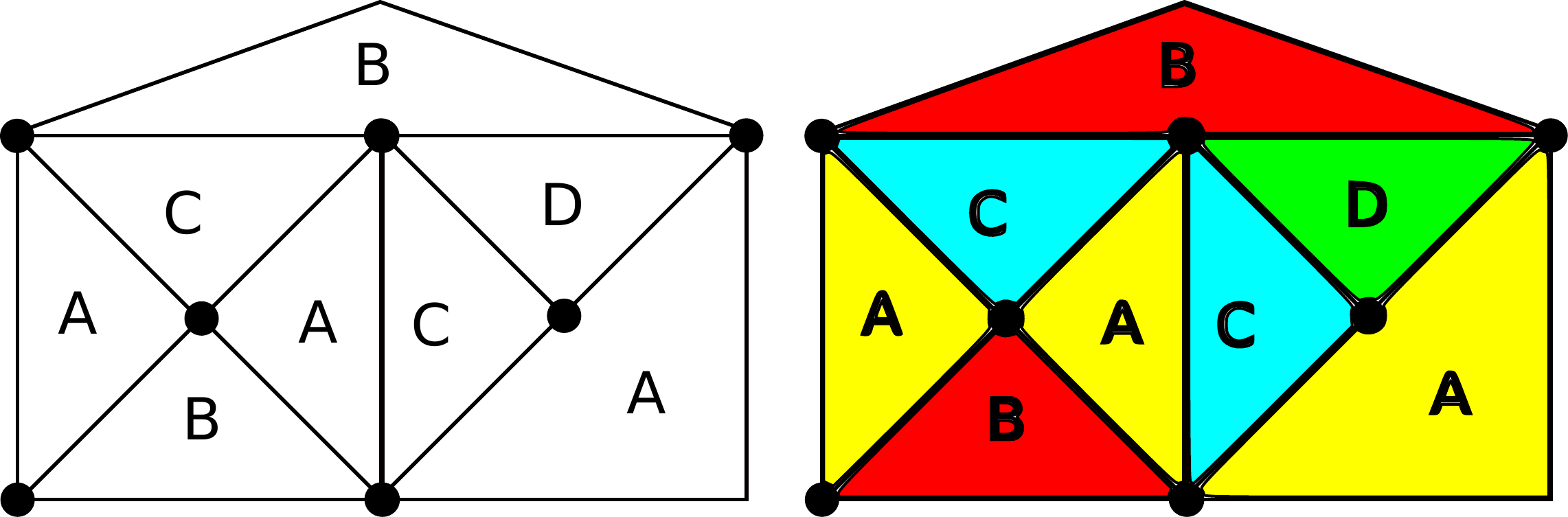}
	  \end{center}
	  \caption{An example of an empire map and an acceptable colouring.}
	  \label{fig:empire-example}
  \end{figure}
	
	\begin{example}
	  Figure \ref{fig:empire-example} shows an example of an empire map, in this example the empire A consists of 3 countries, the B and C empires each consists of 2 countries and the D ``empire" consists of a single country. Thus we can refer to this as a 3-pire map since none of the empire contain more then 3 countries.
	\end{example}
	\begin{remark} \label{rem:3pire-4pire-etc}
	  Note that, for the purposes of proving Lemma~\ref{lem:remove-vertex-still-mpire} later in the chapter, it is convenient to use this relaxed definition of an $m$-pire map which allows that the map in Figure~\ref{fig:empire-example} is also a 4-pire map, 5-pire map, etc. since none of the empires contain more than 4, 5 etc. counties. However,for general purposes, one normally refers to an empire map as being an $m$-pire map where $m$ is the smallest such number for which the statement still holds. Thus we would normally refer to the map in Figure~\ref{fig:empire-example} as a 3-pire map.
	\end{remark}
	
	We then look towards colouring these empire maps and intuitively you simply colour the map as before with the restriction that all countries belonging to the same empire must be coloured the same empire. We can define this more formally as follows:
	\begin{definition}
	  An \textbf{empire map colouring} of an empire map, ($M$, $A$, $g$), consists of assigning each element of $A$ a colour (or equivalently a natural number) such that for any two neighbouring countries, $\alpha, \beta $, with $g(\alpha) \neq g(\beta)$ then $g(\alpha)$ is not assigned the same colour as $g(\beta)$.
	\end{definition}
  
  \begin{example}
	  Figure \ref{fig:empire-example} shows a valid colouring, since each empire neighbours every other empire somewhere each one needs a different colour to be coloured properly.
	\end{example}

\section{Empire Graphs}
  As we did with standard maps it will be beneficial to deal with graphs rather than maps. So we define an empire graph which is the dual to an empire map:
  \begin{definition} \label{def:empire-graph}
    An \textbf{empire graph} is a triple, $(G, A, h)$, where:
    \begin{itemize}
      \item $G$ is a simple connected graph
      \item $A$ is a non-empty set of empires
      \item $h$ a surjective function which associates to each vertex in $G$ an empire from $A$
    \end{itemize}
    for vertices $v, w \in G$, if $h(v) = h(w)$, then $v$ and $w$ are said to be part of the same \textbf{vertex empire}.
  \end{definition}
  
  Similarly we will want to place a restriction on the number of vertices in any vertex empire:
  \begin{definition} \label{def:m-pire-graph}
	  For an empire graph, $(G, A, h)$, if each vertex empire contains at most $m$ vertices then it is said to be an $m$-pire graph.
	\end{definition}
  
  Similarly we can define a colouring of an empire graph:
  \begin{definition}
    An \textbf{empire graph colouring} of an empire graph, $(G, A, h)$, consists of assigning each element of $A$ a colour (or equivalently a natural number) such that for any two adjacent vertices, $v, w$, with $h(v) \neq h(w)$ then $h(v)$ is not assigned the same colour as $h(w)$.
	\end{definition}
  
  \begin{remark}
    It should be obvious that if we construct the standard dual graph of an $m$-pire map we will obtain an $m$-pire graph. Similarly if we construct a dual map from an $m$-pire graph we will obtain a valid $m$-pire map so once again we have comlete duality between maps and graphs.
  \end{remark}
  
  When colouring empire graphs we have the problem that we have different vertices which need to be coloured the same colour which is awkward and difficult to deal with. We can get round this though by simply putting all those vertices together! Or more formally:
  \begin{definition} \label{def:collapsed-empire-graph}
    You can obtain a \textbf{collapsed empire graph}, $G^*$, from an empire graph, $(G, A, h)$, via the following steps:
    \begin{enumerate} [(i)]
      \item Combine all the vertices which belong to the same empire to obtain a new graph $G^*$
      \item For any pair of vertices in $G^*$ which have more than one edge adjacent to both, remove all but one of these edges
      \item Remove any loops from $G^*$
    \end{enumerate}
  \end{definition}
  
  \begin{definition}
    If a collapsed empire graph can be obtained from an $m$-pire graph, then it is referred to a \textbf{collapsed $m$-pire graph}.
  \end{definition}
  
  \begin{figure}[htb]
    \begin{center}
      \includegraphics[width=\textwidth]{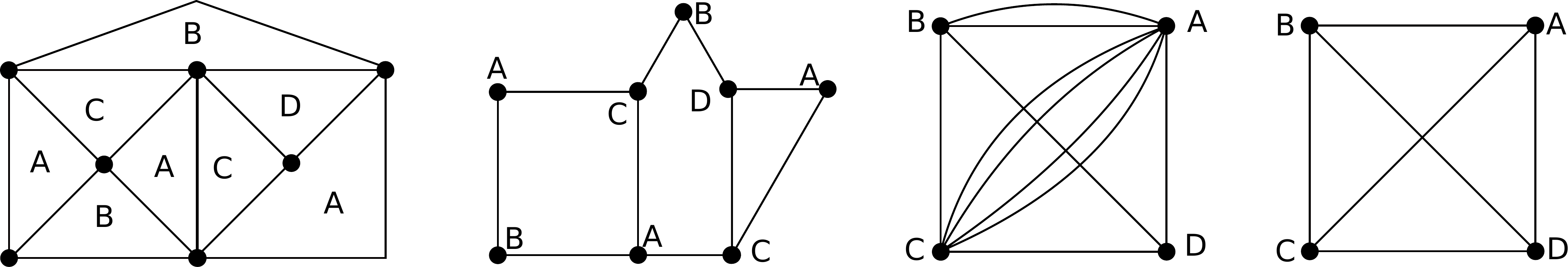}
	  \end{center}
	  \caption{A demonstration of creating a collapsed empire graph from an empire map.}
	  \label{fig:empire-graph-example}
  \end{figure}
  
  \begin{example}
	  Figure \ref{fig:empire-graph-example} shows the process of obtaining a collapsed empire graph from an empire map. The images from left to right represent the following:
	  \begin{enumerate} [(i)]
	    \item The original empire map which we have seen before
	    \item The empire graph which is the dual of the empire map
	    \item The empire graph with all vertices in the same empire identified
	    \item The final empire graph once all extraneous edges have been removed
	  \end{enumerate}
	\end{example}
	
  \begin{remark}
    It should be clear that if an empire map requires $n$ colours to colour properly, say, then the collapsed empire graph obtained from this empire map (via the dual empire graph) will also require $n$ colours to colour properly. Thus our aim for this chapter will be to look at producing an upper bound on the number of colours needed for all possible collapsed empire graphs and therefore produce an equal upper bound on empire maps.
  \end{remark}
  
  \begin{remark}
    It is also important to realise that there is a one to one correspondence between empire maps and empire graphs. However the same does
not hold for collapsed empire graphs. For any given empire map there exists a unique collapsed empire graph which represents it; at the same time any one collapsed empire graph may represent many different empire maps.
  \end{remark}
  
  We now prove a lemma about collapsed empire graphs on the sphere which we will use to prove an important theorem later in the chapter:
  \begin{lemma} \label{lem:empire-graph-average-bound}
    For an $m$-pire map on the sphere, $(M, A, g)$, with (dual) $m$-pire graph $(G', B, h)$ and collapsed $m$-pire graph $G^*$ there exists at at least one vertex $v$ in $G^*$ with $\deg(v) \leq 6m - 1$.
  \end{lemma}
  \begin{proof}
    We follow the argument as outlined in Hutchinson \cite{hutchinson}.
    
    Our aim will be to show that $A(G^*) < 6m$, then if this is true there must exist at least one vertex with degree less that $6m$ as required. Throughout this proof we shall refer to the vertices, edges and countries in $G'$ ($G^*$) as $V', E', C'$ ($V^*, E^*, C^*$) respectively. Also $V_i^*$ will denote the number of vertices in $G^*$ of degree exactly $i$.
    
    So we notice that we only remove edges from $G'$ to obtain $G^*$ we have that:
		\begin{equation} \label{eqn:hlem2-1}
		  E^* \leq E
		\end{equation}
		Then we have from Lemma \ref{lem:edge-count-vertices} that for any graph (so in our case consider $G^*$):
		\begin{displaymath}
		  \sum_{i=1}^{V^*} iV_i^* = 2E^* \leq 2E'
		\end{displaymath}
		Then since $G'$ is a dual graph, by Lemma \ref{lem:dual-graph-simple} we have that it is also simple, so we can apply Lemma \ref{lem:bound-edges-vertices-euler} to get that:
		\begin{displaymath}
		  2E' \leq 6V' - 6\chi(S) = 6V' - 12
		\end{displaymath}
		Then, since $G^*$ is a collapsed $m$-pire graph, each vertex of $G^*$ can have been obtained from at most $m$ vertices of $G$ thus we get:
		\begin{equation} \label{eqn:hlem2-2}
		  V' \leq mV^* \Rightarrow 6V' - 12 \leq 6mV^* - 12
		\end{equation}
		So combining the three inequalities above we obtain:
		\begin{displaymath}
		  \sum_{i=1}^{V^*} iV_i^* \leq 6mV^* - 12
		\end{displaymath}
		Then if we divide through by $V^*$ and noticed that the LHS looks like the average degree of the vertices we get:
		\begin{displaymath}
		  A(G^*) \leq 6m - \frac{12}{V^*}
		\end{displaymath}
		But then since $\frac{12}{V^*} > 0$ this clearly yields the strict inequality:
		\begin{displaymath}
		  A(G^*) < 6m
		\end{displaymath}
		as required.
  \end{proof}
  
\section{Proof of Claim} \label{sec:proof-of-claim}
  The first thing to say is that throughout the rest of this Chapter we will be consider removing vertices from various graphs. However when you remove a vertex from a graph you are left with edges which have one of their ends missing! To solve this, whenever you remove a vertex you also always remove all the edges incident to that vertex. So from now it will be assumed that one removes all incident edges when removing a vertex.
  
  In the paper by Hutchinson \cite{hutchinson}, in the proof of Theorem 1 the following claim is left as an exercise:

  \begin{claim} \label{clm:original-claim}
    For a collapsed $m$-pire graph, $G^*$, you can remove a vertex of degree less than $6m$ to leave another collapsed $m$-pire graph.
  \end{claim}
  
  We will not be proving this claim but instead proving a very similar claim which is sufficient to complete the main proof. We will prove the following claim:
  \begin{claim} \label{clm:modified-claim}
    For a collapsed $m$-pire graph, $G^*$, if you remove any vertex to leave a new graph, $H$, then this graph is a subgraph of a collapsed $m$-pire graph which contains the same number of vertices as $H$.
  \end{claim}

  Firstly, we need to introduce the concept similar to that of an annulus, which we are going to call a ``squeezed annulus", this occurs when countries belonging to the same empire form a kind of circle around some other empires. An example is given in Figure~\ref{fig:squeezed-annulus} where countries from the A completely enclose the two unlabelled countries in the middle.
  
  \begin{figure}[htb]
    \begin{center}
      \includegraphics[width=0.5\textwidth]{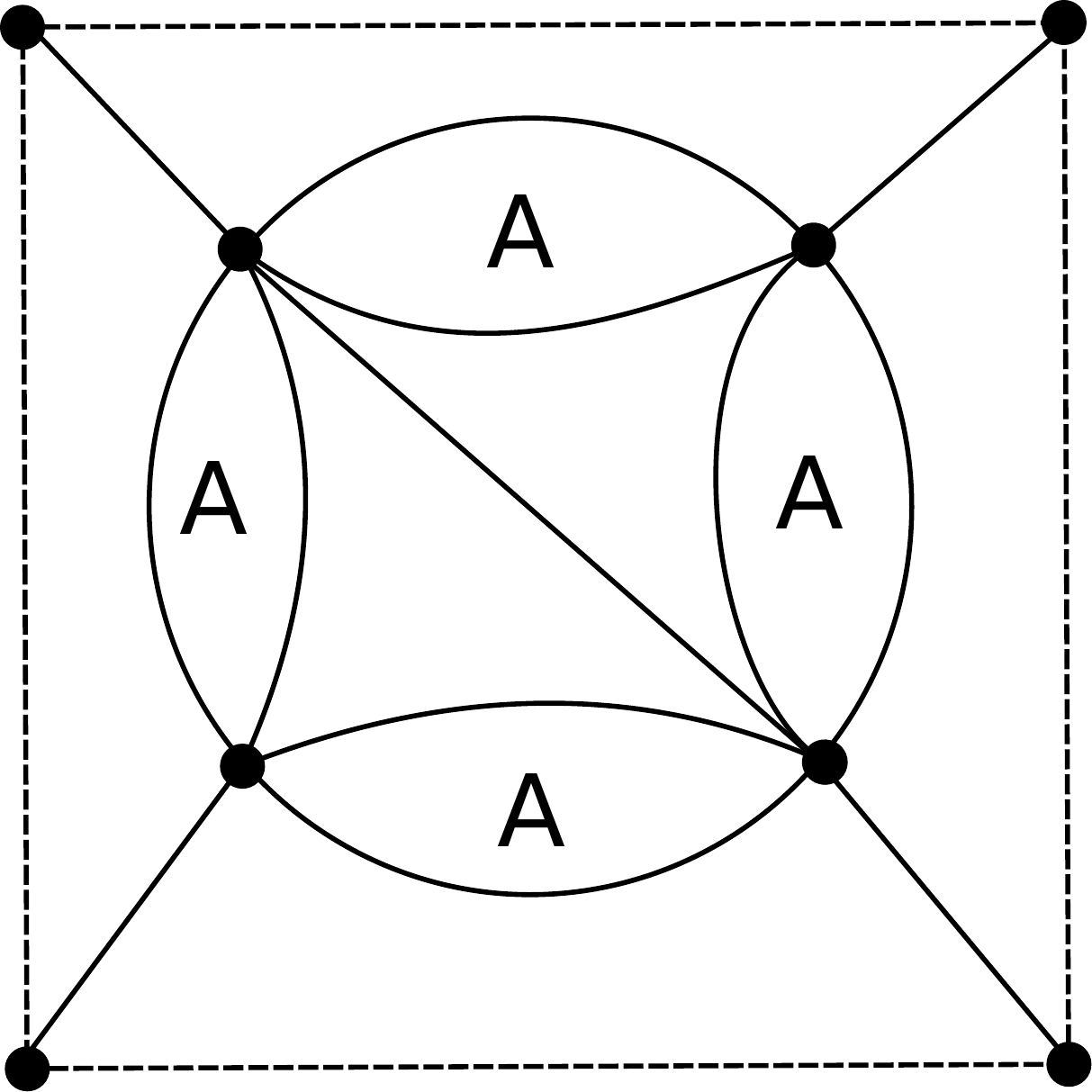}
	  \end{center}
	  \caption{An example of a ``squeezed annulus".}
	  \label{fig:squeezed-annulus}
  \end{figure}

  So we can define more precisely what it means for a graph to contain no annuli or squeezed annuli:
  \begin{definition} \label{def:annulus-free}
    An empire map on the sphere, $(M, A, g)$, with (dual) empire graph $(G, B, h)$ is called \textbf{annulus free} if the dual graph $G$ remains connected after removing any single vertex empire (and all the incident edges associated to each vertex).
  \end{definition}
  
  First we prove a small result about annulus free maps:
  \begin{lemma} \label{lem:annulus-free-remove-connected}
    For an annulus free $m$-pire map on the sphere, $(M, A, g)$, with (dual) m-pire graph, $(G, B, h)$ and collapsed $m$-pire graph $G^*$, if you remove any vertex from $G^*$ to obtain a new graph, $H$, then $H$ is connected.
  \end{lemma}
  \begin{proof}
    This follows trivially from Definition~\ref{def:annulus-free} since removing a vertex from $G^*$ is equivalent to removing an entire vertex empire from $G$ . Then, since $M$ is annulus free, we have that after removing this vertex empire $G$ is still connected and therefore the collapsed version is also connected but that is simply $H$.
  \end{proof}
  
  Then we can prove the original claim (Claim~\ref{clm:original-claim}) if we only consider annulus-free maps:
  \begin{lemma} \label{lem:original-claim-proof}
    For an annulus free m-pire map on the sphere, $(M, A, g)$, with (dual) m-pire graph, $(G, B, h)$ and collapsed m-pire graph $G^*$, if you remove any vertex from $G^*$ to obtain a new graph, $H$, then $H$ is also a collapsed $m$-pire graph.
  \end{lemma}
  \begin{proof}
    Since we have restricted $M$ to being annulus free we see that $H$ must be connected by Lemma~\ref{lem:annulus-free-remove-connected}. Thus all we need to do is find an $m$-pire map which when collapsed (via a dual graph) gives exactly $H$. To do this we notice that removing a vertex from $G^*$ is equivalent to removing an entire vertex empire from $G$ which is the same as removing an entire empire from $M$. So assume, without loss of generality, that we are removing an empire labelled $A$, from $M$. Then we can think of this as ``removing" countries from $M$ which belong to $A$ and leaving ``blank spaces". Ideally we would like to fill these ``blank spaces" without creating any new boundaries, then if we collapsed this new map we would be left with $H$. Figure~\ref{fig:remove-lemma-removing-face-all} shows how this can be done in a variety of cases.
  
    \begin{figure}[htb]
      \begin{center}
        \includegraphics[width=\textwidth]{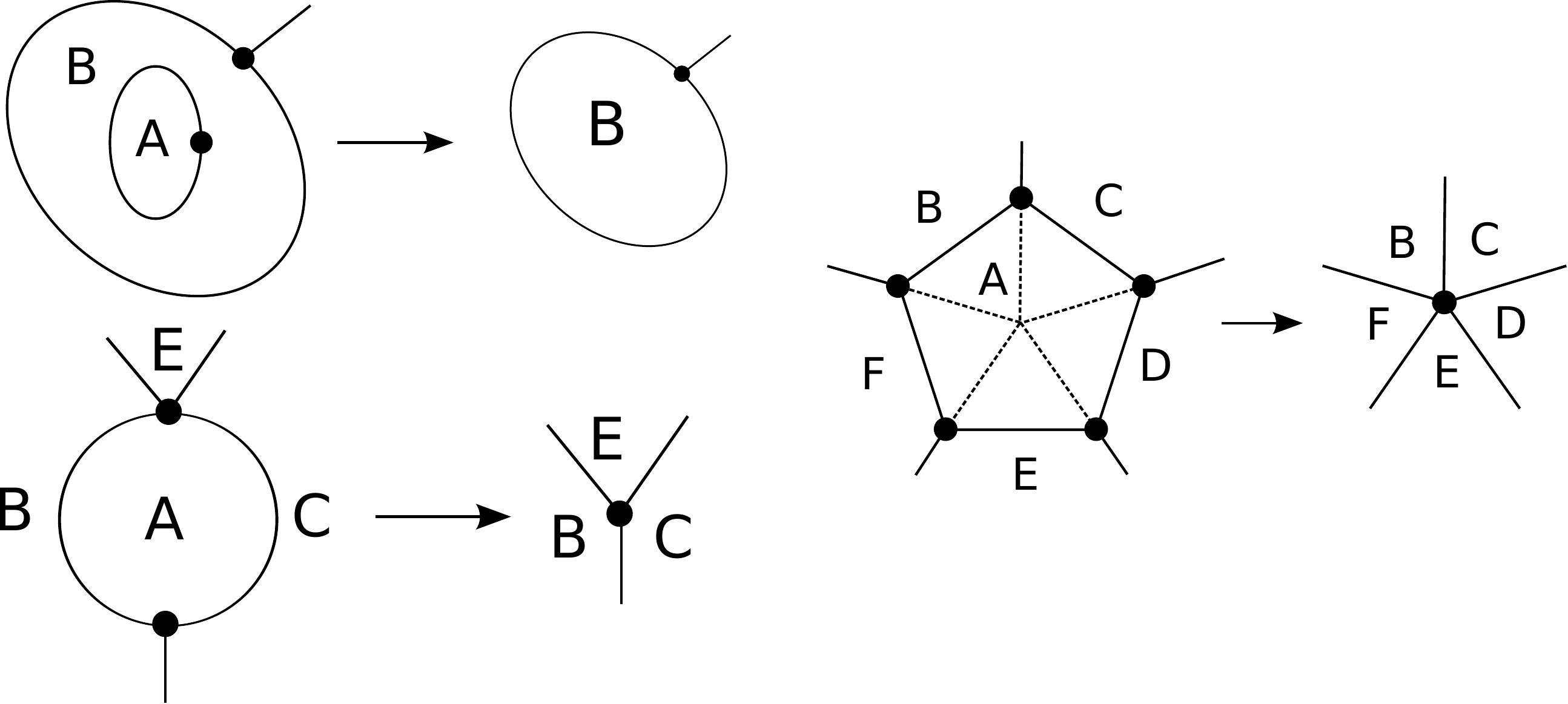}
	  \end{center}  
	  \caption{Filling ``blank spaces" left by ``removing" a country $(A)$ from a map.}
	  \label{fig:remove-lemma-removing-face-all}
    \end{figure}
  
    Thus we can simply fill in all the gaps in $M$ to create a new map, and since we have created no new boundaries we have created no new edges in $H$ and since we have created no new countries each empire must still contain at most $m$ countries and so $H$ is a collapsed $m$-pire graph.
  \end{proof}
  
  So we now relax the condition for $M$ to be annulus free, this time we can only prove the weaker claim (Claim~\ref{clm:modified-claim}) but it does hold for all maps:
  \begin{lemma} \label{lem:proof-modified-claim}
    For an $m$-pire map on the sphere, $(M, A, g)$, with (dual) $m$-pire graph, $(G, B, h)$ and collapsed m-pire graph $G^*$ . Remove any vertex from $G^*$ to obtain a new graph, $H$. Then $H$ is a subgraph of a collapsed $m$-pire graph which contains the same number of vertices (as $H$).
  \end{lemma}
  \begin{proof}
    So, using the same ideas as Lemma~\ref{lem:original-claim-proof}, we want to fill the ``blank spaces". There are several instances where this can go wrong but we can reduce it to a single case, namely when the following two conditions hold:
    \begin{itemize}
      \item Every country is homeomorphic to an open disc
      \item Every country which is part of a squeezed annulus has at most 2 neighbours
    \end{itemize}
    An example of such a map is the left-most map in Figure~\ref{fig:must-create-boundaries}. To show that we can in fact reduce to this single case we need to consider ways in which we can deviate from this type of map. There are 3 ways which we can deviate from this:
    \begin{enumerate}
      \item \textbf{There is a country which is homeomorphic to an annulus:} We can convert this to a squeezed annulus by simply ``pinching it" somewhere as shown in Figure~\ref{fig:annulus-to-squeezed}.
      \item \textbf{A country has more than two neighbours:} In this instance we can simply collapse any extra edges to nothing as shown in Figure~\ref{fig:only-two-neighbours}. (Note that, in this instance, A is not part of a squeezed annulus, however the image merely illustrates different cases of how you can collapse edges.)
      \item \textbf{There is a country homeomorphic to a double-annulus, triple-annulus etc.:} A double-annulus is an open disc with 2 smaller open discs removed, a triple has 3 smaller open discs removed, etc. We can convert these to multiple squeezed annuli by again “pinching" as shown in Figure~\ref{fig:double-to-squeezed}.
    \end{enumerate}
  
  \begin{figure}[htb]
    \begin{center}
      \subfigure[Converting from an annulus to a squeezed annulus.]
        {
          \includegraphics[width=0.4\textwidth]{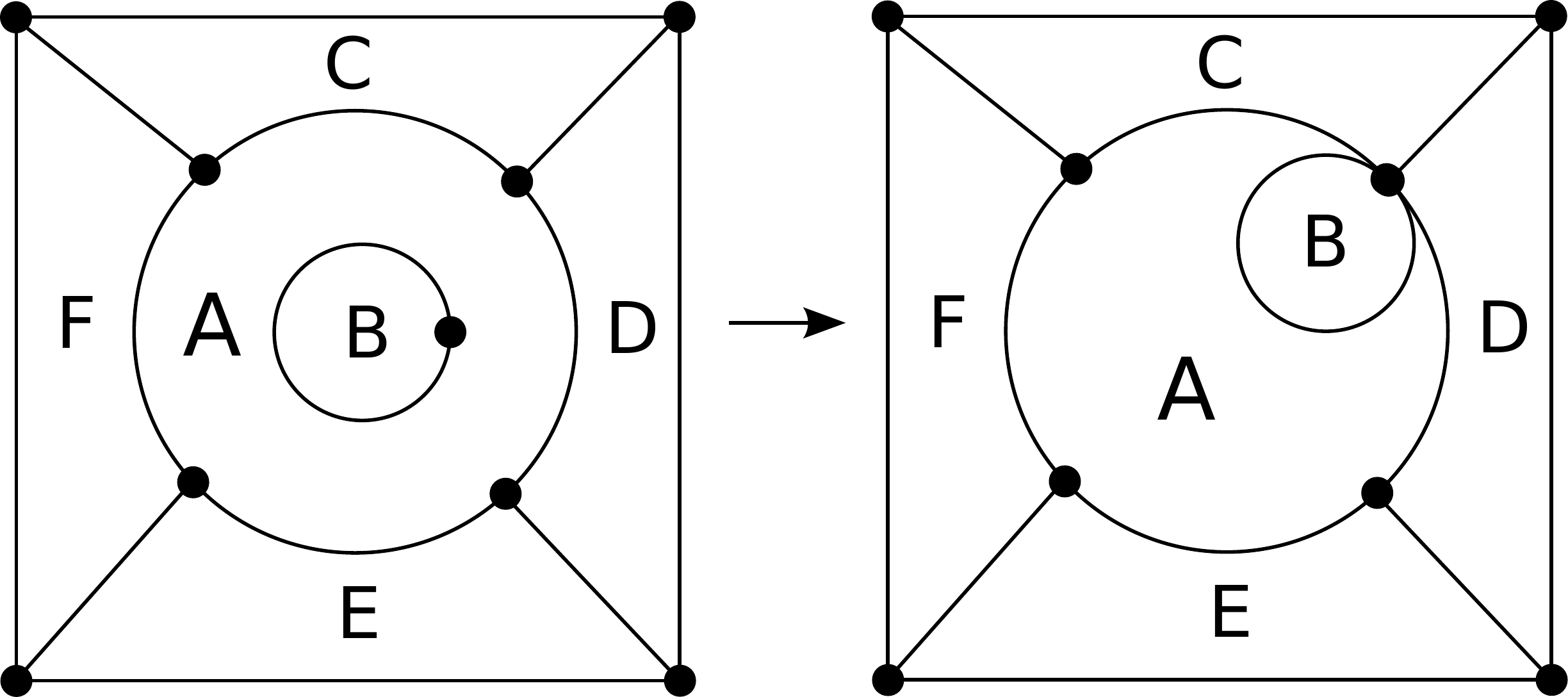}
          \label{fig:annulus-to-squeezed}
        }
      \subfigure[Converting from having multiple neighbours to just two.]
        {
          \includegraphics[width=0.4\textwidth]{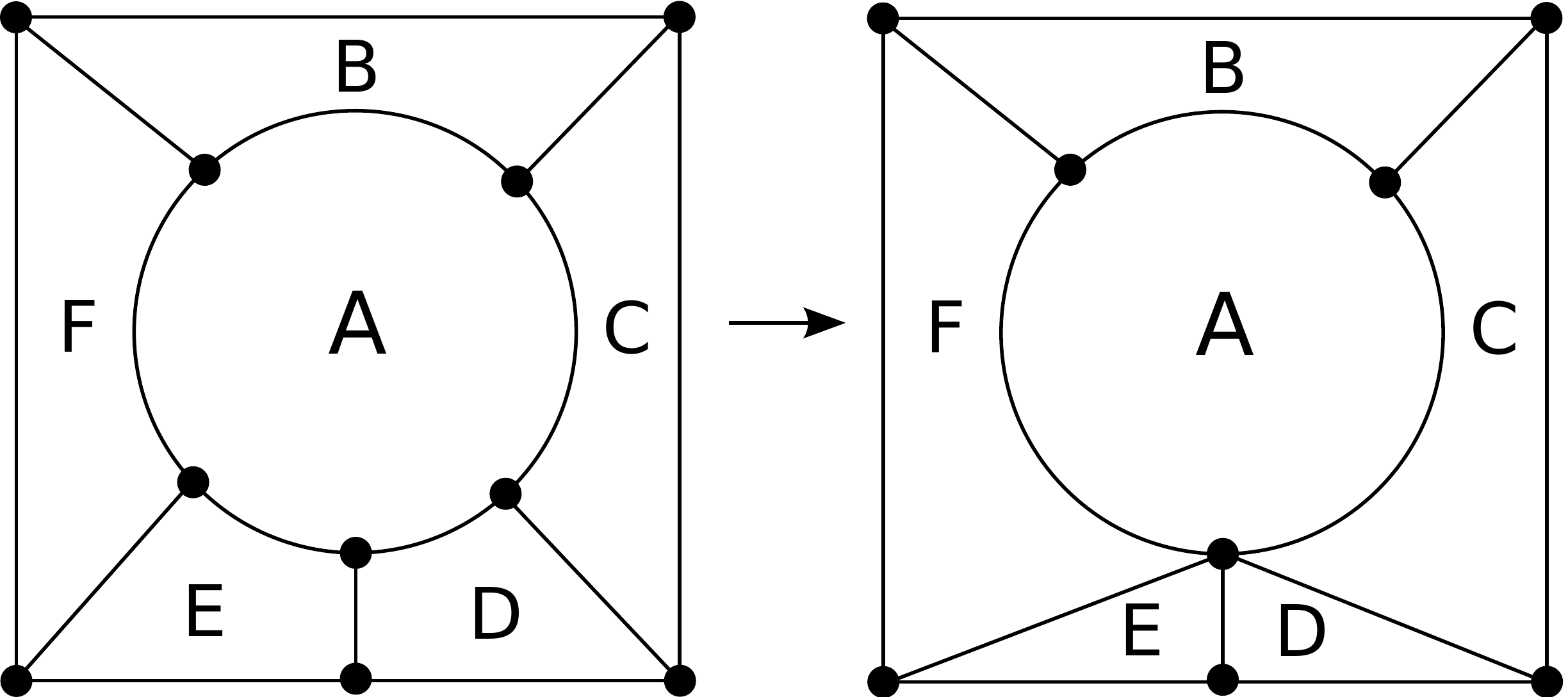}
          \label{fig:only-two-neighbours}
        }
      \subfigure[Converting from a double annulus to two squeezed annuli.]
        {
          \includegraphics[width=0.8\textwidth]{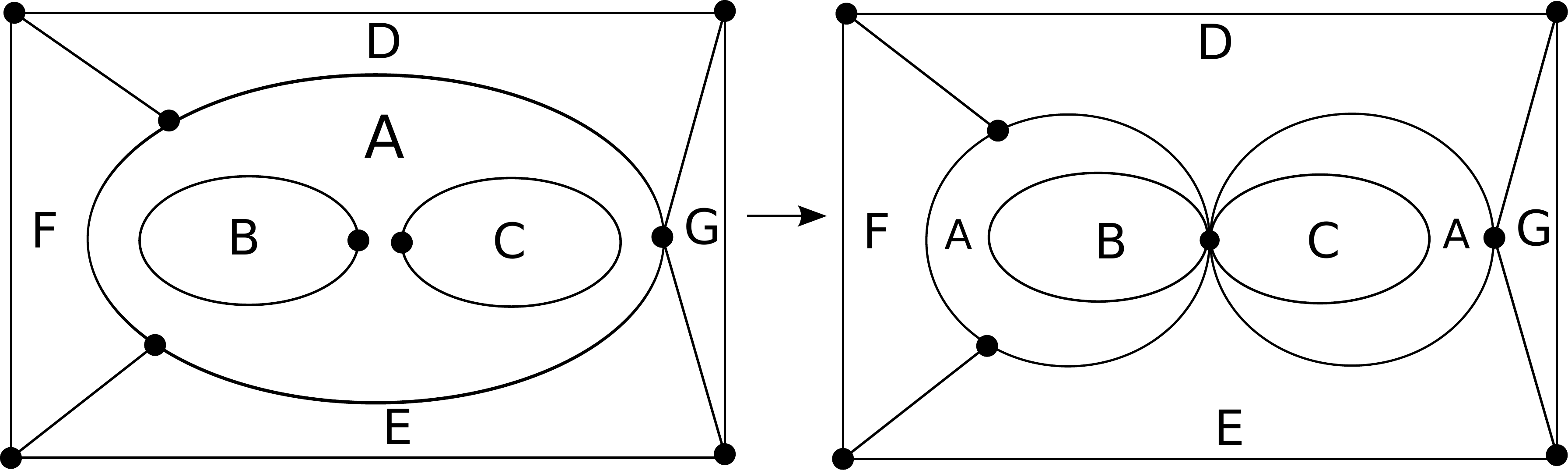}
          \label{fig:double-to-squeezed}
        }
	\end{center}
	\caption{How to convert back to cases covered in Lemma~\ref{lem:proof-modified-claim}.}
	\label{fig:convert-back-cases}
  \end{figure}
  
  \begin{figure}[htb]
    \begin{center}
      \includegraphics[width=\textwidth]{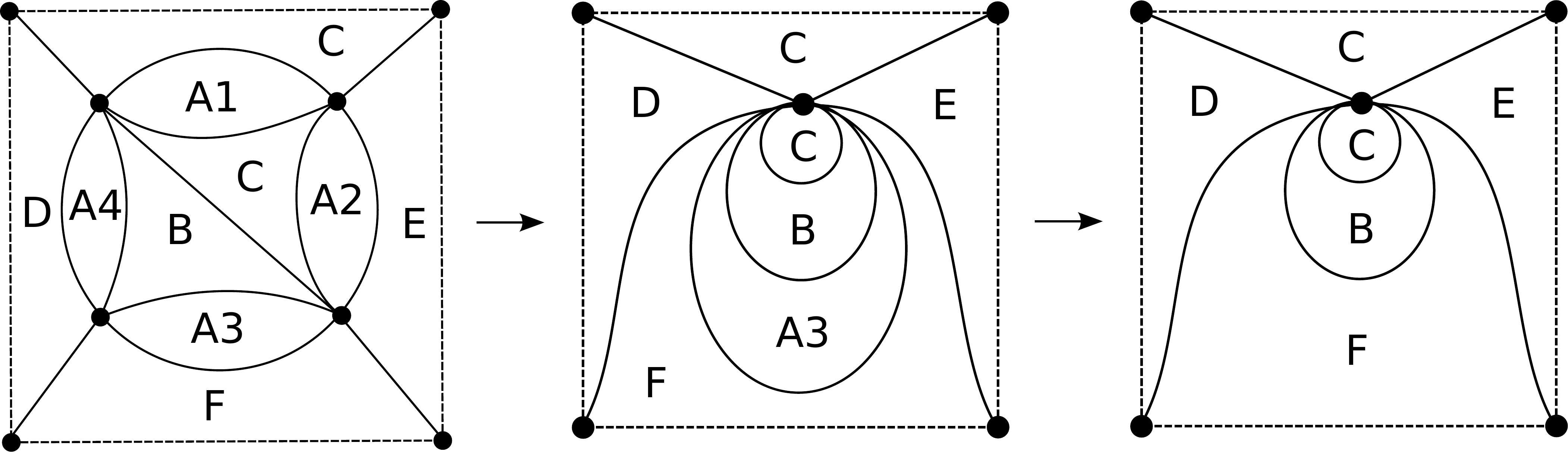}
	  \end{center}
	  \caption{When you are forced to create an extra boundary between empires.}
	  \label{fig:must-create-boundaries}
  \end{figure}

  \begin{remark}
    Note that all of the changes we have demonstrated above can only be done because we are ultimately going to remove the country labelled
A. Otherwise the changes above can create and/or remove boundaries which is not allowed. However since in each case we have only created or removed boundaries relating to the $A$ country, when we remove the associated vertex in the collapsed empire graph we also remove all incident edges. These edges represent the boundaries relating to $A$ and thus on the level of collapsed empire graphs the above changes make no difference.
  \end{remark}
  
  So, now that we have reduced it to this one case, we can show how we would deal with this one case which is demonstrated in Figure~\ref{fig:convert-back-cases}. In this instance we can close all but one (i.e. we close $A1$, $A2$, $A4$) of the countries from the $A$ empire.
  
  \begin{remark}
    In this simple example we could have chosen to collapse all but $A1$ in which case we would have been fine since creating a boundary between
countries from the same empire doesn’t change the collapsed empire graph. However we cannot guarantee that in generality there will be an instance where one of the countries in the $A$ empire will be bounded on each side by a country from the same empire.
  \end{remark}
  
  So we consider the situation as shown where $A3$ separates the $B$ and $F$ empires which may not be neighbouring anywhere else. In this instance we simply remove the $A3$ country completely and create a new boundary. When we collapse this new graph we will not necessarily get $H$, we will get $H$ with a potential extra edge between the vertices representing $B$ and $F$. But note that we have the same number of vertices as $H$ (i.e. we didn’t create any new empires), also we didn’t create any countries so this graph is still a collapsed $m$-pire graph of which H is a subgraph, as required.
  
  \end{proof}

\section{Heawood's Upper Bound}
	In his rather prolific paper of 1890 Heawood proposed the following upper bound for $m$-pire graphs:
	\begin{theorem} \label{thm:heawood-upper-empire}
	  Any $m$-pire map, $M$, on the sphere can be coloured with at most $6m$ colours.
	\end{theorem}
	Unsurprisingly he also produced a proof of this upper bound, as we shall also do:
	\begin{proof}
	  Let $G^*$ be the collapsed $m$-pire graph obtained from $M$. If we can provide the same upper bound the the colours required to colour G* then we are done.
	  
	  We use proof by induction on the number of vertices, $V^*$, in $G^*$ as outlined in Hutchinson \cite{hutchinson}.
	  
	  \textbf{Base Statement:} If $V^* \leq 6m$ then the result is trivial since we can simply colour each vertex a different colour.
	  
	  \textbf{Inductive Step:} So then assume that the theorem holds for all collapsed $m$-pire graphs with fewer than $V^*$ vertices and let $V^* > 6m$.
	  
	  Then find in $G^*$ a vertex $v$ which has degree at most $6m - 1$ (existence guaranteed by Lemma~\ref{lem:empire-graph-average-bound}) and remove it and all it’s incident edges from $G^*$ to obtain a new graph $H$. Then by Lemma~\ref{lem:proof-modified-claim} we can say that $H$ is a subgraph of a collapsed $m$-pire graph, call it $L$. Crucially though, $H$ has fewer than $V^*$ vertices and $L$ has the same number of vertices as $H$, thus by the inductive assumption we can colour $L$ in $6m$ colours, do so. Then note that, since $L$ contains more edges than $H$, any valid colouring of $L$ will also work on $H$ which is simpler to colour, so colour $H$ accordingly.
      Then to this $6m$ colouring of $H$ we add back in $v$ and all its’ incident edges, then since $v$ is adjacent to at most $6m - 1$ other vertices there must be a spare colour which we can use to colour $v$, thus we have coloured $G^*$ in $6m$ colours as required.
  \end{proof}
	
  \begin{corollary} \label{cor:map-upper-bound-6}
    Every map in the plane can be coloured using 6 colours.
  \end{corollary}
  
  \begin{remark} \label{rem:map-upper-bound-5}
    The above Corollary does indeed provide an upper bound which is non-trivial although in the same paper that Heawood produced Theorem~\ref{thm:heawood-upper-empire} and it's proof he also produced a proof that every map in the plane could be coloured using 5 colours thus rendering it immediately obsolete!
  \end{remark}

  \chapter{Heawood's Empire Conjecture} \label{chp:empire-conjecture}
As demonstrated by Corollary \ref{cor:map-upper-bound-6} and Remark \ref{rem:map-upper-bound-5} the upper bound from Theorem \ref{thm:heawood-upper-empire} as proved by Heawood is not sharp for the case $m=1$ i.e. a ``normal" map with no empires. However Heawood did propose another conjecture that for all $m \geq 2$ the inequality was sharp, he even went to far as to produce an example for $m=2$ however in a similar manner to his previous conjecture the example was so irregular he could not generalise is to all $m>2$. Over the years examples were produced for $m=2,3,4$ and in 1984 Jackson and Ringel proved all cases for $m \geq 5$ thus proving Heawood's conjecture. \newline
In 1997 Walter Wessel managed to produce a shorter and nearly uniform proof for all $m \geq 2$ and it is this that we will be discussing. Inevitably though we will start by introducing soe new concepts and some smaller results that will help us in proving the conjecture.

\section{Complete Graphs II}
  We have already introduced the concept of complete graphs back in Section~\ref{sec:complete-graphs} and Figure~\ref{fig:complete-and-bipartite} shows an example of $K_5$ which is a complete graph on 5 vertices. We now introduce the idea of a decomposition of complete graphs into Hamiltonian paths.
  \begin{definition}
    For a graph $G$, if you take a sequence of vertices such that any two neighbouring vertices (in the sequence) are both incident to a common edge then the union of those vertices and the common edges is called a \textbf{path}. The first and last vertices in the sequence are called the \textbf{endpoints} of the path.
  \end{definition}
  
  \begin{remark}
    Although definitions of what a path is vary from book to book, for the purposes of this report we will only be considering edge-disjoint paths. That is to say no edge appears twice in a single path.
  \end{remark}
  
  \begin{definition}
    A path in a graph $G$ is called a \textbf{Hamiltonian path} if it includes all the vertices of $G$.
  \end{definition}
  
  \begin{figure}[htb]
    \begin{center}
      \includegraphics[width=\textwidth]{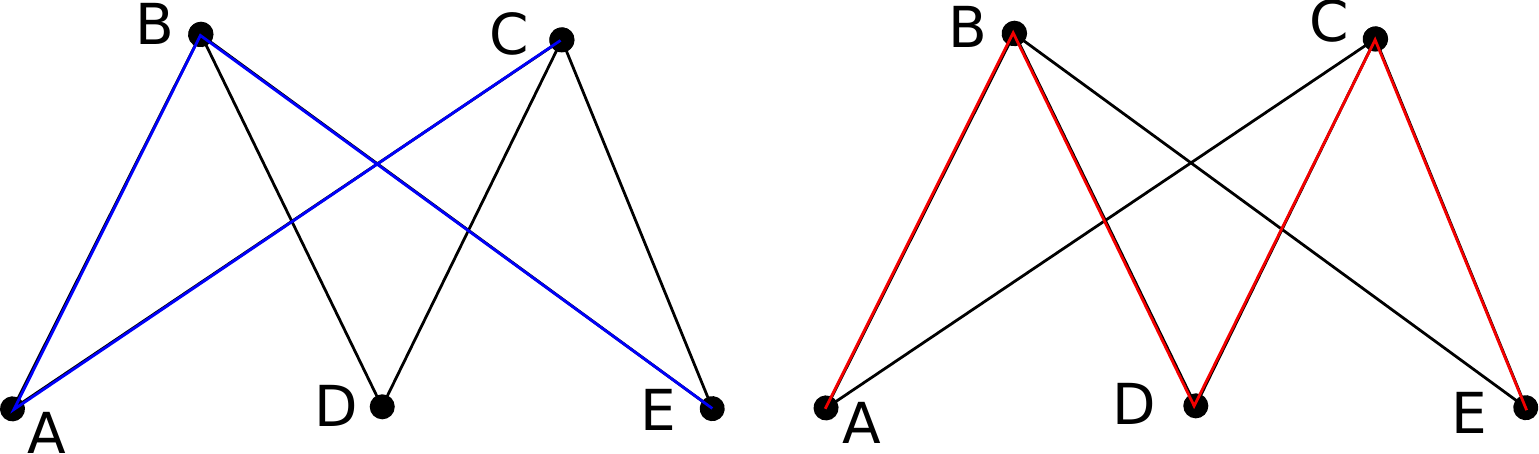}
	\end{center}
	\caption{An example of a path (LHS) and a Hamiltonian path (RHS).}
	\label{fig:path-hamiltonian-path}
  \end{figure}
  
  \begin{example}
    Figure~\ref{fig:path-hamiltonian-path} demonstrates the two kind of paths in the graph $K_{3,2}$. The left-hand graph demonstrates a standard path with endpoints C and E whereas the right-hand graph is an example of a Hamiltonian path with endpoints A and E.
  \end{example}
  
  \begin{remark}
    Note that we do not care about the direction of the path, so to borrow again the example in Figure \ref{fig:path-hamiltonian-path}, the first path can start at C and end at E or equivalently can start at E and end at C.
  \end{remark}
  
  We then look at how you can represent a graph as a union of paths:
  \begin{definition}
    For a graph $G$, a \textbf{graph decomposition} into paths consists of a series of paths $p_1, p_2, ..., p_n$ such that:
    \begin{displaymath}
      \bigcup_{i=1}^{n} p_i = G
    \end{displaymath}
  \end{definition}
  
  \begin{example}
    Once again using the examples in Figure \ref{fig:path-hamiltonian-path}, the two paths would form a decomposition of $G$ since the union of the two paths contain the whole of $G$.
  \end{example}
  
  The most interesting decompositions of graphs are edge-disjoint decompositions, i.e. no two paths in the decomposition share a common edge. Thus in the decomposition of $K_{3,2}$ represented by the two paths in Figure \ref{fig:path-hamiltonian-path} are not edge-disjoint since they share the edge between A and B.
  
  We then introduce a theorem which is proved in \cite[Theorem 11]{bollobas} concerning complete graphs (we provide the theorem without proof):
  \begin{theorem} \label{thm:complete-graph-decomposition}
    All complete graphs on an even number of vertices, $K_{2n}$, can be decomposed into a $n$ edge-disjoint Hamiltonian paths.
  \end{theorem}
  
  \begin{figure}[htb]
    \begin{center}
      \includegraphics[width=\textwidth]{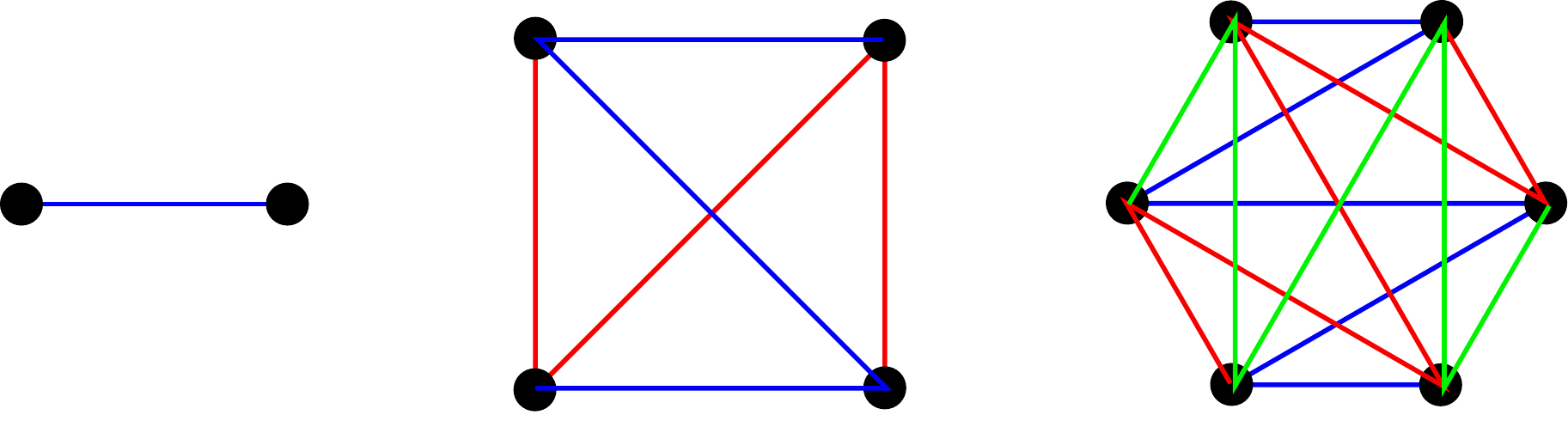}
	\end{center}
	\caption{Decompositions of $K_2, K_4, K_6$ into edge-disjoint Hamiltonian paths.}
	\label{fig:K2n-path-split}
  \end{figure}
  
  \begin{corollary} \label{cor:endpoint-exactly-once}
    In a decomposition of $K_{2n}$ into $n$ edge-disjoint Hamiltonian paths each vertex of $K_{2n}$ is the endpoint of a path exactly once.
  \end{corollary}
  \begin{proof}
    We use proof by contradiction. First assume that a vertex $v$ is the endpoint of at least two of the $n$ paths. \newline
    We have that if a vertex is an endpoint of a path it is incident to exactly one edge in the path. If it isn't an endpoint it is incident to exactly two edges in the path. So if we look at the degree of $v$:
    \begin{displaymath}
      \deg(v) \leq 2 + 2(n-2) = 2n-2
    \end{displaymath}
    But all the vertices in $K_{2n}$ have degree $2n-1$ which is a contradiction so $v$ cannot be the endpoint of more than 1 path. So assume that it is not the endpoint of any paths then similarly to above you have:
    \begin{displaymath}
      \deg(v) = 2n \neq 2n-1
    \end{displaymath}
    So therefore $v$ must be the endpoint of exactly 1 path.
  \end{proof}
  
  \begin{corollary} \label{cor:touch-every-other-once}
    In a decomposition of $K_{2n}$ into $n$ edge-disjoint Hamiltonian paths each vertex of $K_{2n}$ is adjacent to every other vertex exactly once.
  \end{corollary}
  \begin{proof}
    We first note that, by construction, every vertex in $K_{2n}$ is adjacent to every other vertex exactly once so it suffices to show that no edges are repeated or missed out. But since the decomposition is edge disjoint we can never repeat an edge. We can then count the number of edges, $E$, in $K_{2n}$:
    \begin{displaymath}
      E = \frac{1}{2}2n(2n-1) = n(2n-1)
    \end{displaymath}
      Then considering the number of edges in the path decomposition we get that each path has $2n-1$ edges and there are $n$ paths giving $n(2n-1)$ across all paths. Thus since no edge is repeated we must have the same edges and therefore Corollary \ref{cor:touch-every-other-once} holds.
  \end{proof}
  
\section{Introduction to the Proof}
  The intuitive result we are going to prove is that Theorem \ref{thm:heawood-upper-empire} is actually strict, that is there are examples of planar empire maps which require $6m$ colours, i.e. 
  \begin{theorem} \label{thm:planar-empire-maps}
    For all $m \geq 2$ there exists an $m$-pire map which requires $6m$ colours to colour properly and can be embedded on the sphere.
  \end{theorem}
  
  Since any graph which can be embedded on the sphere can also be embedded on the plane (and vice versa), throughout this chapter we will be discussing planar graphs. This is partly for consistency with Wessel and also for brevity of description! Also as before we will deal with graphs as they are easier to work with so we introduce a empire analogue of complete graphs:

  \begin{definition} \label{def:complete-empire-graph}
    A \textbf{complete $m$-pire graph on $n$ empires} is an $m$-pire graph, $(G, A, h)$ which satisfies the following additional conditions:
    \begin{enumerate} [(a)]
      \item $|A| = n$ ($|A|$ denotes the number of elements in $A$)
      \item any two empires (elements of $A$) are neighbouring, i.e. for any $\alpha, \beta \in A$, there exists adjacent vertices $v, w \in G$ such that $h(v) = \alpha$ and $h(w) = \beta$
    \end{enumerate}
    for ease of notation, these graphs will be referred to as $J(n, m)$ graphs.
  \end{definition}
  
  \begin{remark}
    It is important to point out that whilst we have now restricted these graphs quite a lot they are still just a type of graph and do not represent a specific graph in the same way that complete graphs, e.g. $K_5$, do.
  \end{remark}
  
  The reason we have introduced these extra restrictions is that, in the same as $K_n$ always required $n$ colours to colour properly, so we can say the same about $J(n, m)$ graphs:
  \begin{lemma} \label{lem:complete-empire-graphs-number-colours}
    All $J(n, m)$ graphs require $n$ colours to colour properly.
  \end{lemma}
  \begin{proof}
    Since each empire neighbours every other empire each empire will require a different colour and there are exactly $n$ empires.
  \end{proof}
  
  So now we would like to draw some graphs in the plane and see if they are $J(n, m)$ graphs. However we would like some sufficient conditions that will allow us to check more easily whether a particular graph is a $J(n, m)$ graph:
  \begin{lemma} \label{lem:complete-empire-sufficient-conditions}
    A graph $G$ is a complete $m$-pire graph on $n$ empires if it satisfies the following conditions:
    \begin{enumerate} [(a)]
      \item $G$ is simple
      \item The vertices of $G$ can be partitioned into exactly $n$ empires
      \item Each vertex empire contains at most $m$ vertices
      \item For any two empires there exist two vertices, one from each, which are adjacent
    \end{enumerate}
  \end{lemma}
  \begin{proof}
    This follows directly from Definitions \ref{def:empire-graph}, \ref{def:m-pire-graph} and \ref{def:complete-empire-graph}.
  \end{proof}
  
  \begin{remark}
    In reality conditions (a) through (c) from Lemma~\ref{lem:complete-empire-sufficient-conditions} are easy to check but it is condition (d) which is hardest to satisfy and also to verify.
  \end{remark}
  
  We now begin our assualt on Theorem~\ref{thm:planar-empire-maps} by proving a few intermediary statements; the first is a special case and also helps illustrate Definition~\ref{def:complete-empire-graph} and Lemma~\ref{lem:complete-empire-sufficient-conditions}:
  \begin{lemma} \label{lem:empire-planar-case-3}
    There exist a $J(18,3)$ graph which is planar.
  \end{lemma}
  \begin{proof}
    We simply provide the graph (Figure~\ref{fig:heawood-plane-case-3}) as given in Wessel \cite{wessel-plane}. Note that any vertex which is not incident to any edges should be seen as being adjacent to all the vertices in the boundary of the face it is lying in. Also two vertices which are labelled the same should be viewed as belonging to the same empire.
      
    \begin{figure}[htb]
	    \begin{center}
	      \includegraphics[width=0.7\textwidth]{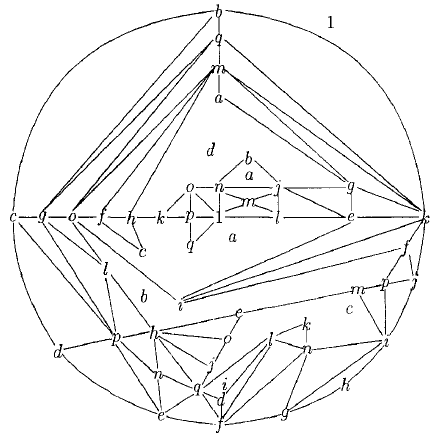}
		  \end{center}
		  \caption{The graph required to prove Theorem \ref{thm:planar-empire-maps} for the case $m=3$.}
		  \label{fig:heawood-plane-case-3}
	  \end{figure}
		
		We should check the conditions of Lemma~\ref{lem:complete-empire-sufficient-conditions}. We can see by inspection that the graph in Figure~\ref{fig:heawood-plane-case-3} is simple, also we can see that it is partitioned into 18 empires, namely $a$ through $q$ and 1. It is also easy to check that each empire contains no more than 3 vertices (in fact each empire contains exactly 3 apart from empire 1 which only contains 2).
		
		The last condition is tedious to check thus we provide an example by considering the $a$ empire. We look at each of the vertices labelled $a$ and check that between them they are adjacent to all the others. So we have the following adjacencies for the vertices:
		\begin{displaymath}
		  (b, n, j); (o, f, c, h, k, p, q, 1, l, e, i); (d, m, g)
		\end{displaymath}
		and we can see that each of the other empires appears there exactly once. Similarly it can be checked for all the other vertices.
		
		Thus we have that the graph provided in Figure~\ref{fig:heawood-plane-case-3} is indeed a planar $J(18, 3)$ graph.
	\end{proof}
  
  We now provide a general way to construct a planar graph out of 6 paths as provided by Wessel \cite{wessel-plane} and for illustrative purposes we show how it can be done in the case that each path contains exactly 4 vertices:
  \begin{construction} \label{con:empire-graph-case-2}
    We take 6 paths, $A, A', B, B', C, C'$, note that $A$ and $A'$ need not be related but are called so for convenience. We then describe a general procedure to construct a general planar graph; note that in all cases the adding of extra edges must be done without crossing:
    \begin{enumerate} [(i)]
      \item Represent these 6 paths in the plane without crossing (Figure \ref{fig:empire-graph-1})
      \item Add extra edges by taking one endpoint from each of the $A$ paths and joining them both to all the vertices in one of the $B$ paths (Figure \ref{fig:empire-graph-2})
			\item Add edges by taking the other endpoint from each of the $A$ paths and joining them both to all the vertices in the other $B$ path (Figure \ref{fig:empire-graph-3})
			\item Repeat the previous two steps but with the endpoints of $B$ joining to the $C$ paths and the endpoints of $C$ joining to the $A$ paths (Figure \ref{fig:empire-graph-4})
	  \end{enumerate}
	  
	  \begin{figure}[htb]
	    \begin{center}
	      \includegraphics[height=0.25\textheight]{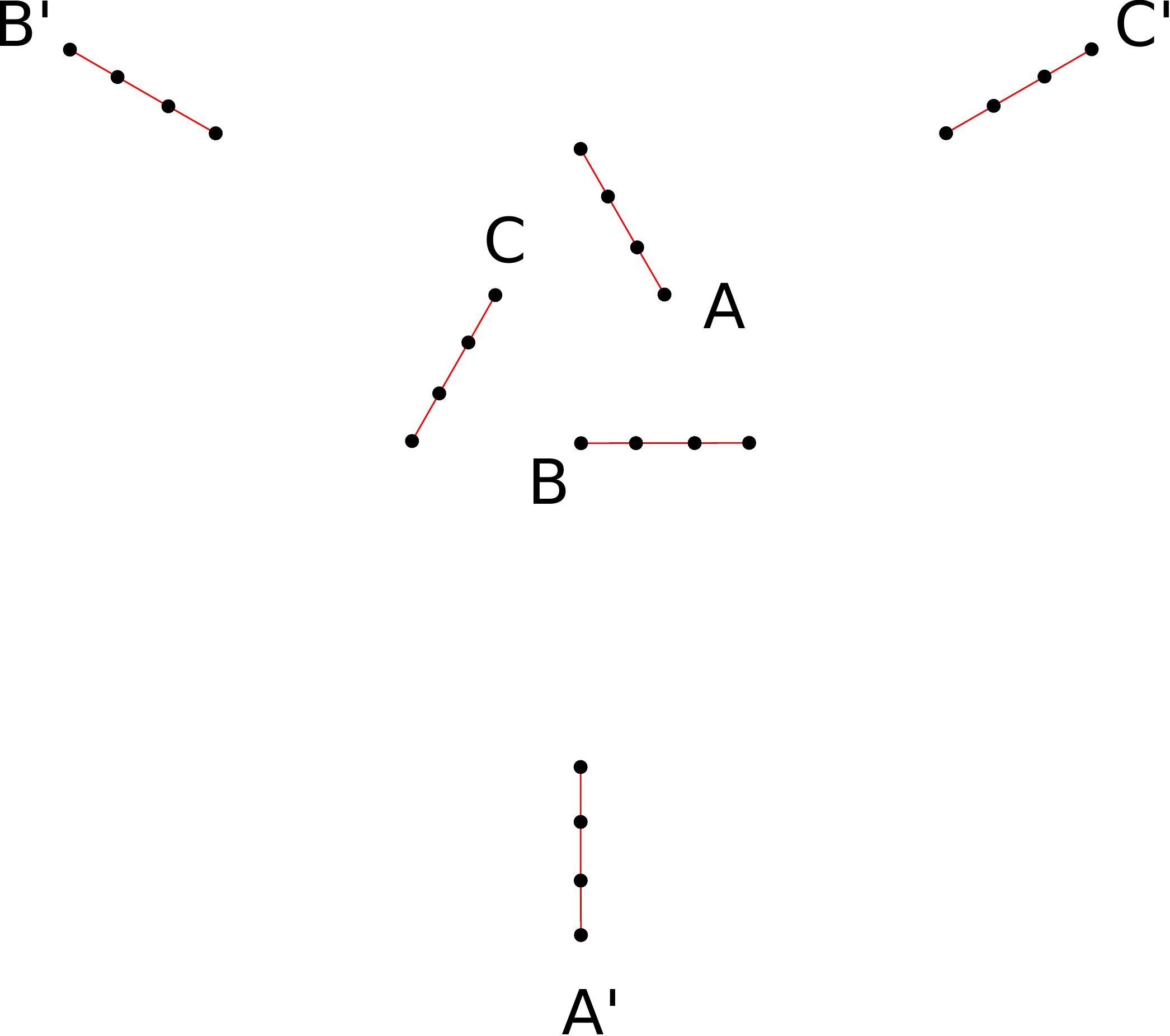}
		\end{center}
		\caption{(Construction \ref{con:empire-graph-case-2}) Representation of 6 paths in the plane without crossing.}
		\label{fig:empire-graph-1}
	  \end{figure}
		
		\begin{figure}[htb]
	    \begin{center}
	      \includegraphics[height=0.25\textheight]{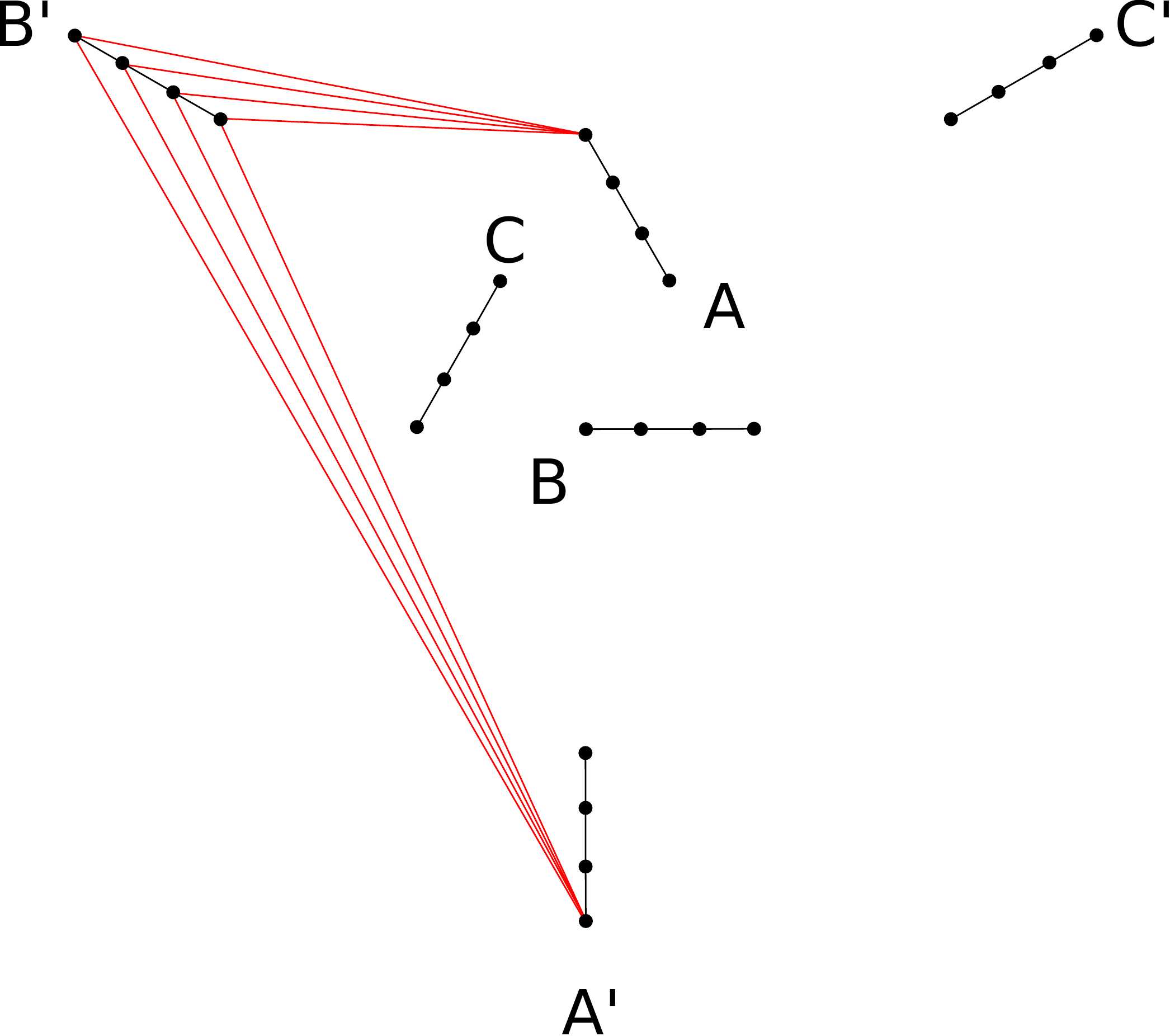}
		\end{center}
		\caption{(Construction \ref{con:empire-graph-case-2}) Joining the endpoints of two $A$ paths with one of the $B$ paths.}
		\label{fig:empire-graph-2}
	  \end{figure}
		
		\begin{figure}[htb]
	    \begin{center}
	      \includegraphics[height=0.25\textheight]{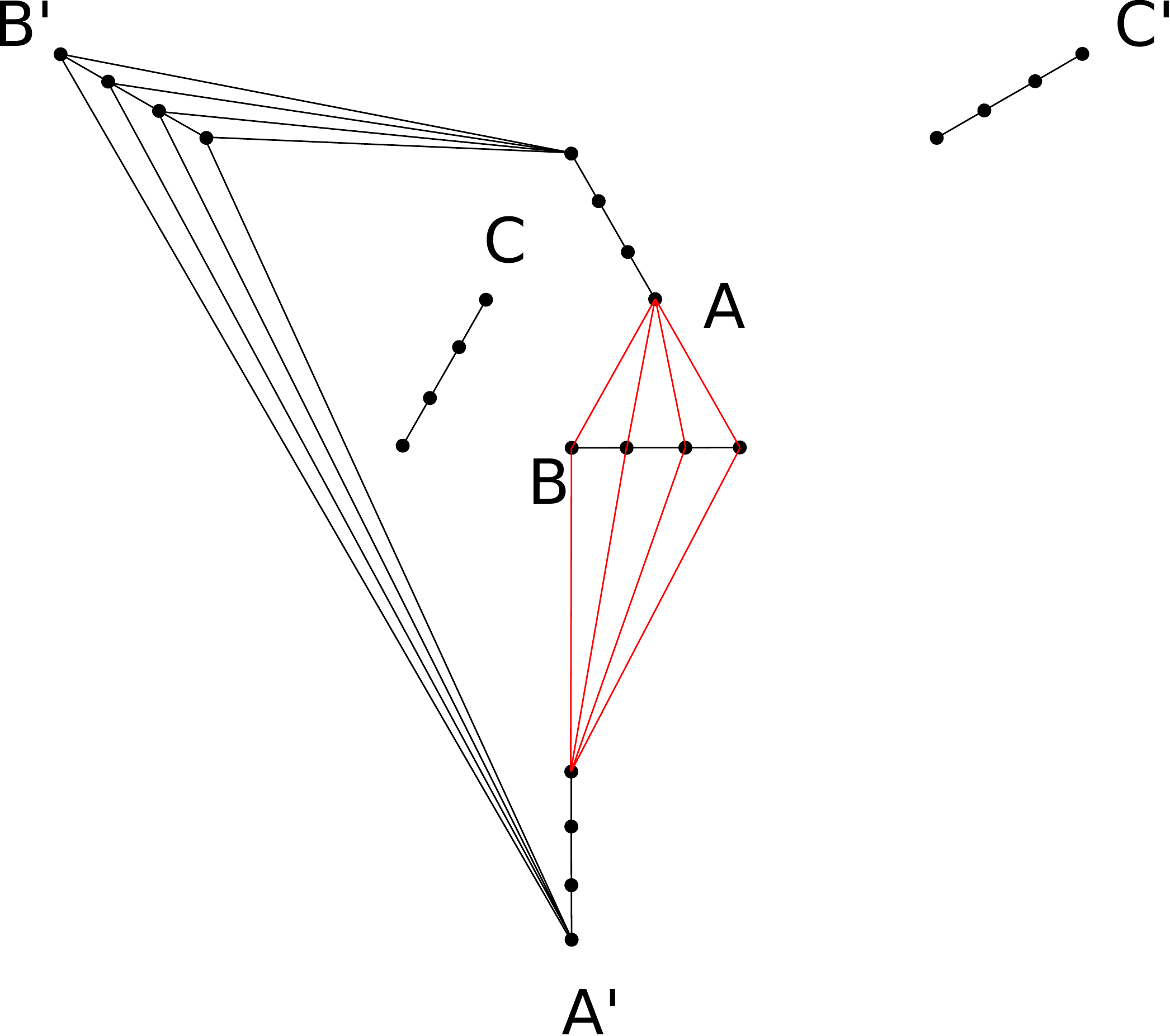}
		\end{center}
		\caption{(Construction \ref{con:empire-graph-case-2}) Joining the other enpoints of the $A$ paths to the other $B$ path.}
		\label{fig:empire-graph-3}
	  \end{figure}
		
		\begin{figure}[htb]
	    \begin{center}
	      \includegraphics[height=0.25\textheight]{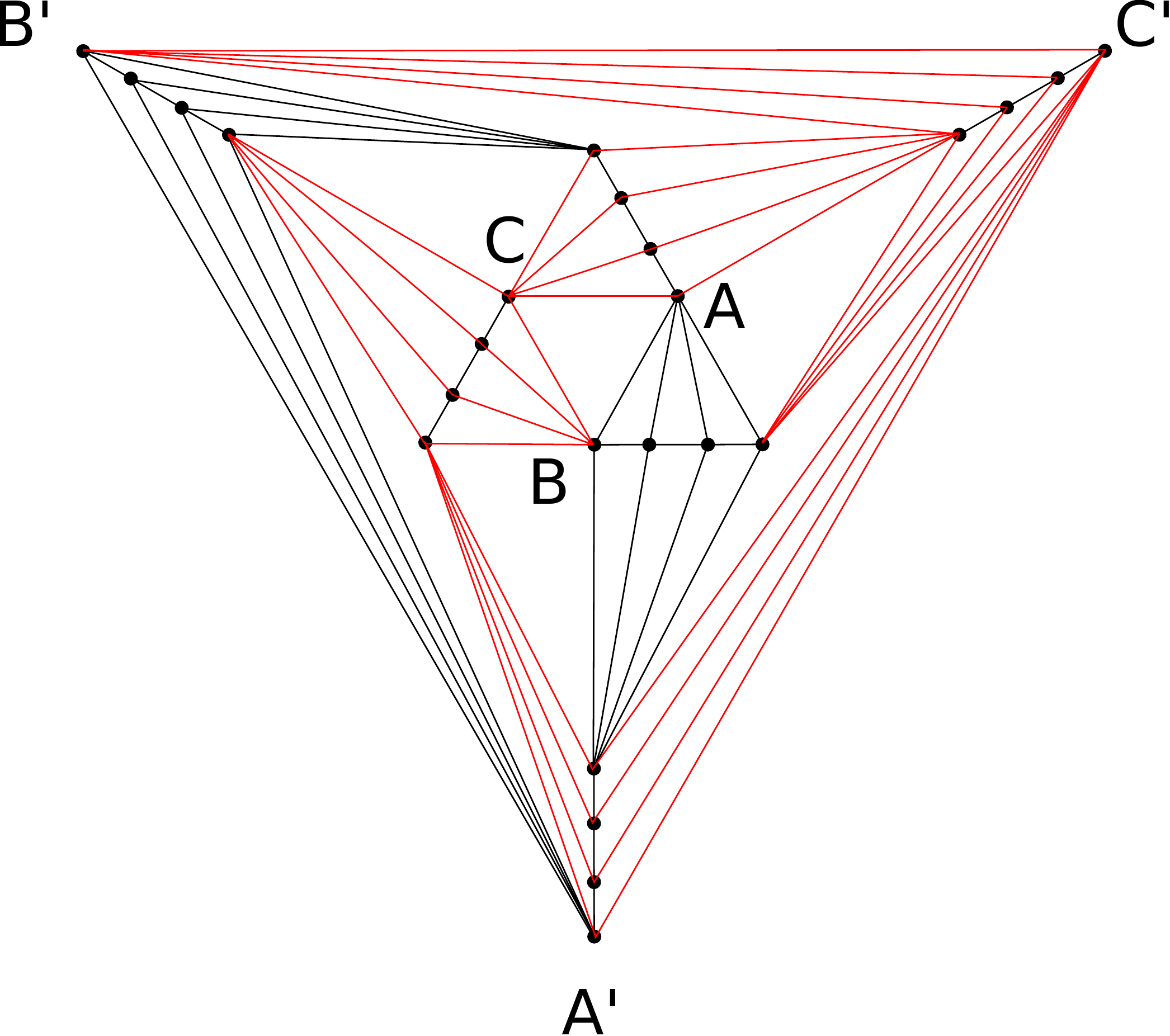}
		\end{center}
		\caption{(Construction \ref{con:empire-graph-case-2}) Repeating the process for $B$ to $C$ and $C$ to $A$.}
		\label{fig:empire-graph-4}
	  \end{figure}
	\end{construction}
	
	\begin{remark}
	  The process as outlined in Construction \ref{con:empire-graph-case-2} is clearly planar. Also the process will work irrespective of the length of any of $A, A', B, B', C, C'$ i.e. each can be arbitrarily large (although finite!) and they can each have a different number of vertices and the construction will still work.
	\end{remark}
  
  We can then use Construction \ref{con:empire-graph-case-2} to prove another case:
  \begin{lemma} \label{lem:empire-planar-case-2}
    There exists a $J(12,2)$ graph which is planar.
  \end{lemma}
  \begin{proof}
    We start with $K_4$, then by Theorem \ref{thm:complete-graph-decomposition} we have that we can decompose this into 2 edge-disjoint Hamiltonian paths as demonstrated in Figure \ref{fig:empire-graph-0}.
			
	\begin{figure}[htb]
	  \begin{center}
	    \includegraphics[scale=0.3]{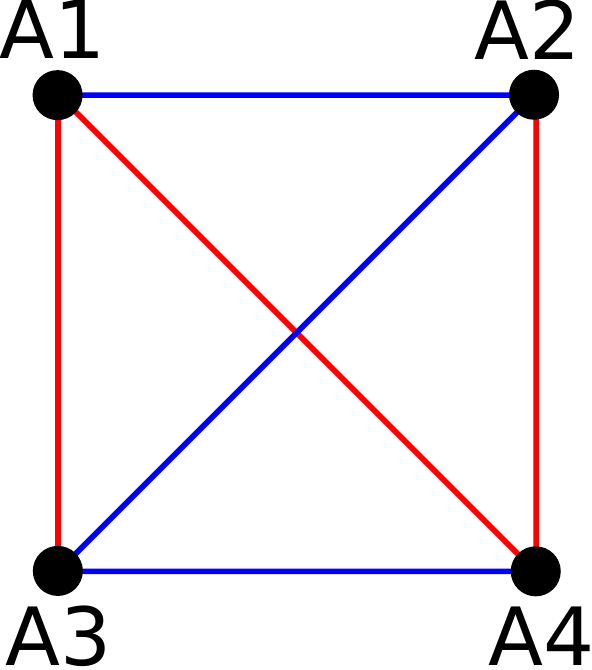}
	  \end{center}
	  \caption{Decomposition of $K_4$ into 2 edge-disjoint Hamiltonian cycles.}
	  \label{fig:empire-graph-0}
	\end{figure}
		  
	We call these two paths $A$ and $A'$ which are both paths on vertices $A_1, A_2, A_3, A_4$. We then take two more copies of $K_4$ and decompose these and call the paths $B, B', C, C'$ and label the vertices similarly to before. We then use Construction \ref{con:empire-graph-case-2} on these 6 paths to produce the graph as in Figure \ref{fig:empire-graph-case-2}.
			
			\begin{figure}[htb]
		    \begin{center}
		      \includegraphics[height=0.25\textheight]{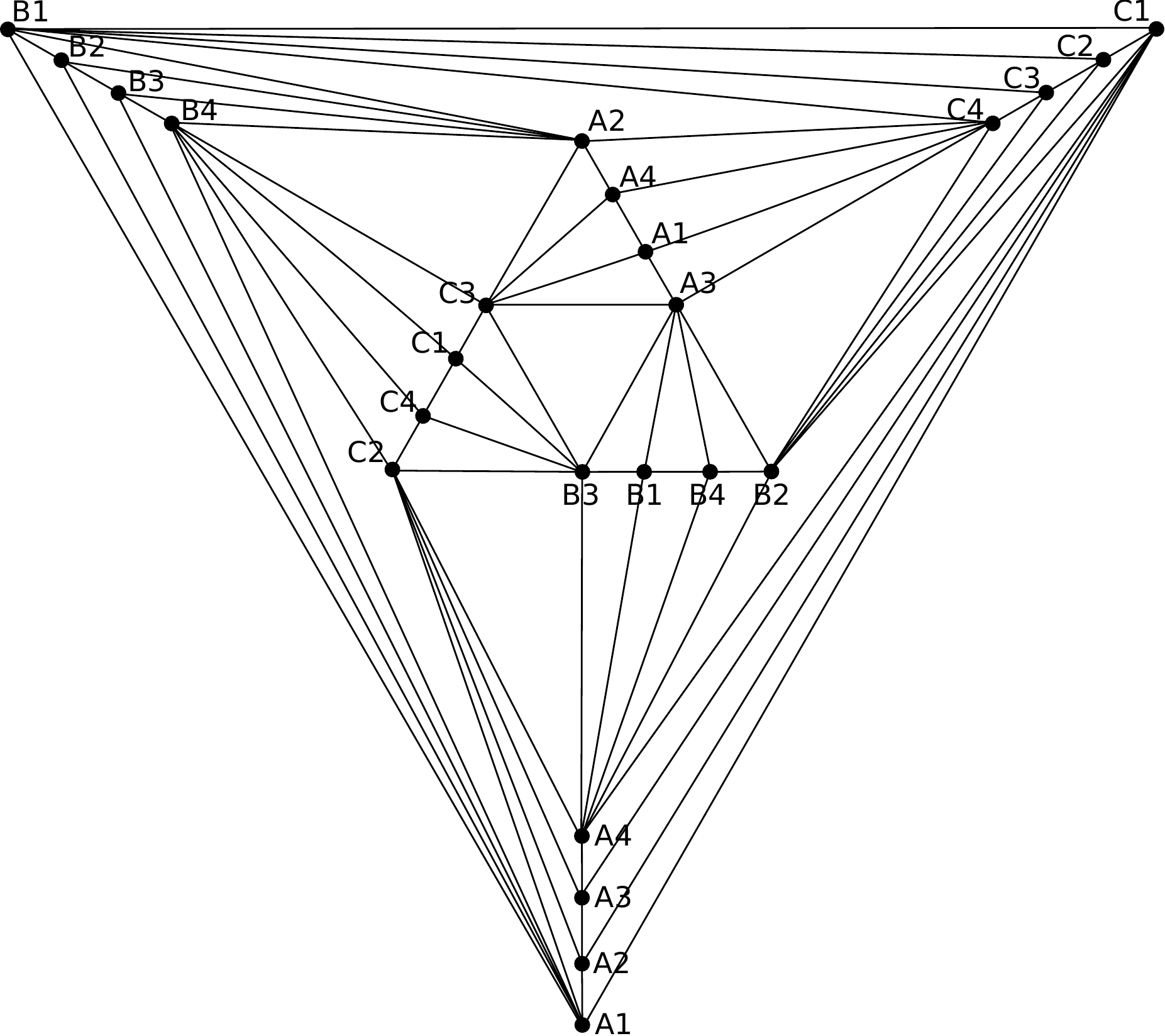}
			\end{center}
			\caption{Graph which proves Lemma \ref{lem:empire-planar-case-2}.}
			\label{fig:empire-graph-case-2}
		  \end{figure}
		  
  The claim is that this graph is the complete empire graph $J(12,2)$ and there are 4 conditions we need to check (from Lemma~\ref{lem:complete-empire-sufficient-conditions}):
    \begin{enumerate} [(a)]
      \item The graph is simple
      \item Its' vertices can be partitioned into exactly $12$ empires
      \item Each vertex empire contains at most $2$ vertices
      \item For any two empires there exist two vertices, one from each, which are adjacent
    \end{enumerate}
    Conditions (a) through (c) are trivially satisfied and by inspection we can see that condition (d) is also satisfied. Thus the graph in Figure \ref{fig:empire-graph-case-2} is indeed $J(12,2)$ and it is clearly planar and so we are done.
  \end{proof}
  
  We can then look at extending this idea to higher $m$ namely:
  \begin{lemma} \label{lem:empire-planar-case-even}
    There exists a $J(6m, m)$ graph which is planar for all $m$ even, $m \geq 2$.
  \end{lemma}
  \begin{proof}
    In a similar fashion to that of Lemma \ref{lem:empire-planar-case-2} we consider the complete graph $A = K_{2m}$ and we decompose it into $m$ edge disjoint Hamiltonian paths. We then take two more compies, $B, C$ and do the same. Then do the following:
    \begin{enumerate} [(i)]
      \item Take two paths from each of the decompositions of $A, B, C$ and apply Construction \ref{con:empire-graph-case-2} to these 6 paths to obtain a graph $G$
      \item Repeat the previous step (using paths which have not yet been used) until there are no paths in the decompositions of $A, B, C$ which are unused
      \item Then let the (disjoint) union of these $\frac{m}{2}$ graphs $G$ be called $G'$
    \end{enumerate}
    
    Then the claim is that $G'$ is in fact the required graph $J(6m,m)$. Again we need to check 5 conditions (from Lemma~\ref{lem:complete-empire-sufficient-conditions}):
    \begin{enumerate} [(a)]
      \item $G$ is simple
      \item Its' vertices can be partitioned into exactly $6m$ empires
      \item Each vertex empire contains at most $m$ vertices
      \item For any two empires there exist two vertices, one from each, which are adjacent
    \end{enumerate}
    
    Condition (a) is satisfied since we never create a loop or join two vertices twice.
    
    Condition (b) is satisfied since we take 3 copies of $K_{2m}$ giving us a total of $6m$ vertices each of which gets copied to create an empire.
    
    Condition (c) is satisfied since we use the $m$ Hamiltonian paths from the decomposition of $K_{2m}$ and as such each vertex only gets copied $m$ times and thus each empire contains at most (in fact in this case, exactly) $m$ vertices
    
    Condition (d) requires a little more work but the first thing to note is that Construction \ref{con:empire-graph-case-2} is symmetric in $A, B, C$ and in fact no where in this proof have we singled out $A, B$ or $C$, thus it is sufficient to check that every vertex in $A$ (say) touches every other vertex in $A$ and also every vertex in $B$. So firstly by Corollary \ref{cor:touch-every-other-once} we have that $A$ does indeed touch every other vertex in $A$ exactly once. Then in Construction \ref{con:empire-graph-case-2} we connect both endpoints of $A$ to all of $B$ so it is sufficient to check that each vertex of $A$ is an endpoint of at least one path, but by Lemma \ref{cor:endpoint-exactly-once} we have that this is indeed the case.
    
    Thus the graph constructed above is indeed $J(6m, m)$ and since it is the disjoint union of planar graphs it is also planar as required.
  \end{proof}
  
  We have now dealt with all even values of $m$ and also $m=3$ and so we now move onto higher odd values of $m$:
  \begin{lemma} \label{lem:empire-planar-case-odd}
    There exists a $J(6m, m)$ graph which is planar for all odd $m$ with $m \geq 5$.
  \end{lemma}
  \begin{proof}
    We would like to use an argument similar to that used in the proof of Lemma \ref{lem:empire-planar-case-even} but we do not have an even number of each of the $A, B, C$ paths. We can, however, modify the procedure slightly. So in a similar fashion we consider the complete graph $A = K_{2m}$ and we decompose it into $m$ edge disjoint Hamiltonian paths. We then take two more compies, $B, C$ and do the same. Then since $m$ is odd we have that $m = 2k + 1$ for some $k \geq 2$, we then use $2k$ of each of the $A$, $B$ and $C$ paths and apply the same construction as we used in Lemma \ref{lem:empire-planar-case-even}.
    We now need to add in the last remaining $A, B$ and $C$ paths, ensuring that the endpoints of $A$ are adjacent to the whole of the $B$ path and likewise the $B, C$ endpoints adjacent to the whole of the $C, A$ paths respectively. We use the following procedure (illustrated for the case $m=5$):
  \begin{enumerate}
	\item Select one path from $A, B, C$ and represent them without crossings in the plane (Figure \ref{fig:empire-graph-odd-1})
	\item Add extra edges by connecting both endpoints of $A$ with all the vertices of $B$ (Figure \ref{fig:empire-graph-odd-2})
	\item Add extra edges by connecting one endpoint of $B$ and $C$ with all the vertices of $C$ and $A$ respectively (Figure \ref{fig:empire-graph-odd-3})
  \end{enumerate}
			
  \begin{figure}[htb]
    \begin{center}
      \includegraphics[height=0.25\textheight]{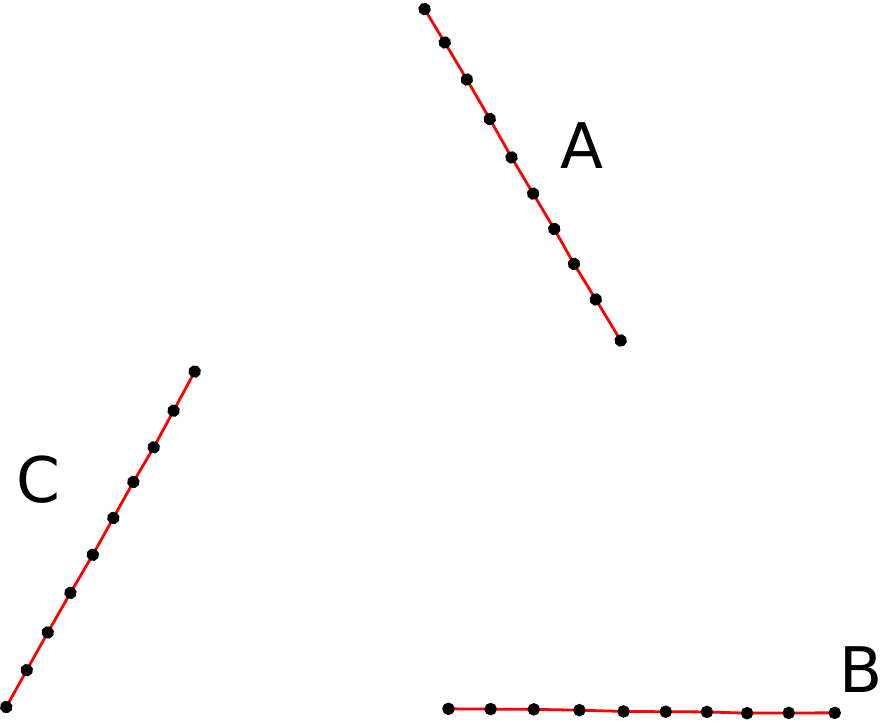}
	\end{center}
	\caption{Drawing one path from each of $A, B, C$ without crossings in the plane.}
	\label{fig:empire-graph-odd-1}
  \end{figure}
	
  \begin{figure}[htb]
    \begin{center}
      \includegraphics[height=0.25\textheight]{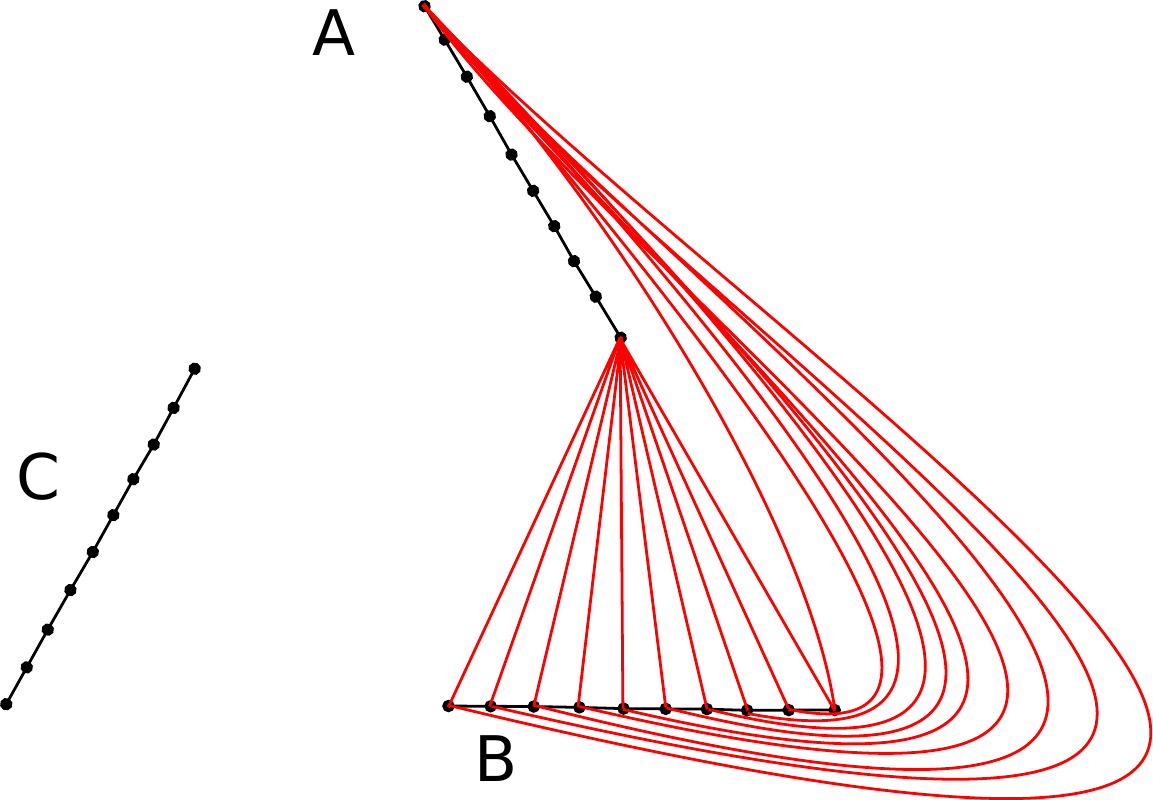}
	\end{center}
	\caption{Connect each endpoint of $A$ with all the vertices of $B$.}
	\label{fig:empire-graph-odd-2}
  \end{figure}
	
  \begin{figure}[htb]
    \begin{center}
      \includegraphics[height=0.25\textheight]{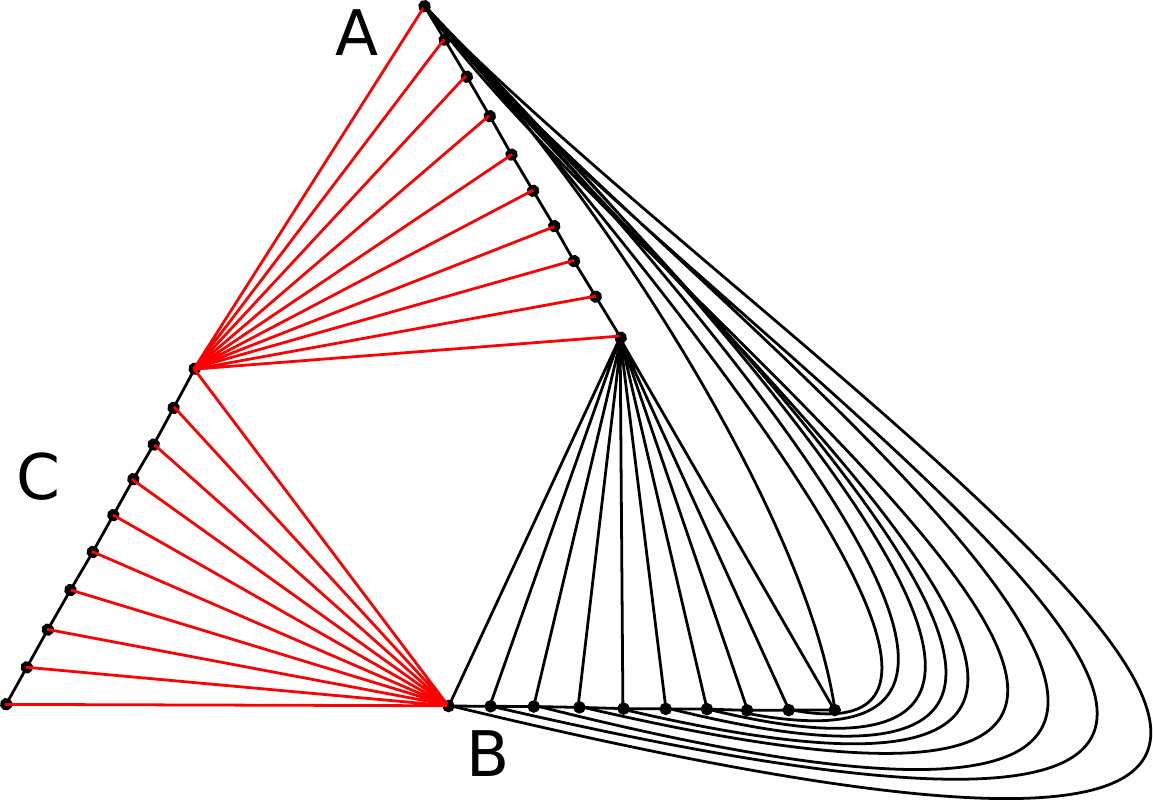}
	\end{center}
	\caption{Connect one endpoint of $B$ and $C$ with all the vertices of $C$ and $A$ respectively.}
	\label{fig:empire-graph-odd-3}
  \end{figure}
  
    We would like to connect the other endpoint of $B$ with all of the vertices of $C$ but we cannot do this on this graph without creating crossings in the plane, we can however connect all the vertices of $C$ to another copy of the endpoint from one of copies of Figure \ref{fig:empire-graph-case-2} which we have created. To do this let the other endpoint of $B$ be denoted by $b$, then we need to redraw (still without crossings) one of the copies of Figure \ref{fig:empire-graph-case-2} so that $b$ is on the exterior, Figure \ref{fig:empire-graph-odd-5} shows how this can be done.
    \begin{remark}
	  We have simply redrawn one edge to expose one more vertex from $B$ which we are taking to be $b$, it should be pointed out that we can always do this since at some point $b$ will be adjacent to an endpoint and we can deliberately chose that path to be drawn on the outside (since we made no specification as to which $B$ path went where) so that we can always do this.
	\end{remark}
	\begin{remark}
	  It should also be mentioned that we are taking a copy of Figure~\ref{fig:empire-graph-case-2} which isn't technically correct as all the paths should contain 10 vertices not 4 but the idea is the same.
	\end{remark}
		  
    So once we have exposed $b$ in this graph we can then connect it to all the vertices of $C$ from our original graph, Figure~\ref{fig:empire-graph-odd-6} shows how this can be done (note we have rotated the graph from the one seen in Figure~\ref{fig:empire-graph-odd-5}). Then finally we would like to connect the final endpoint of $C$ (call it $c$) with all the vertices of $A$, and since we are assuming $m \geq 5$ we know that we have another copy of Figure~\ref{fig:empire-graph-case-2} which we haven't yet used, so we can perform the same trick again to expose $c$. However to connect $c$ to all the vertices of $A$ from our original graph we need to place it inside one of the countries of Figure~\ref{fig:empire-graph-odd-3} and then connect it. Since by this point the graph would be too small to be useful we show in Figure~\ref{fig:empire-graph-odd-7} how this can be done with the large blue circle representing an entire copy of Figure~\ref{fig:empire-graph-case-2} with the edges being specifically connected to the copy of $c$. 
    \begin{remark}
	  Once again we should point out that we know we can expose $c$ by a similar argument to that of exposing $b$ and we can deliberately ensure that the required paths (where $b$ and $c$ are adjacent to an endpoint) are not in the same graph since again there is no requirements on how you pair up the paths from different graphs.
    \end{remark}		
		  	
      \begin{figure}[htb]
	    \begin{center}
	      \includegraphics[height=0.25\textheight]{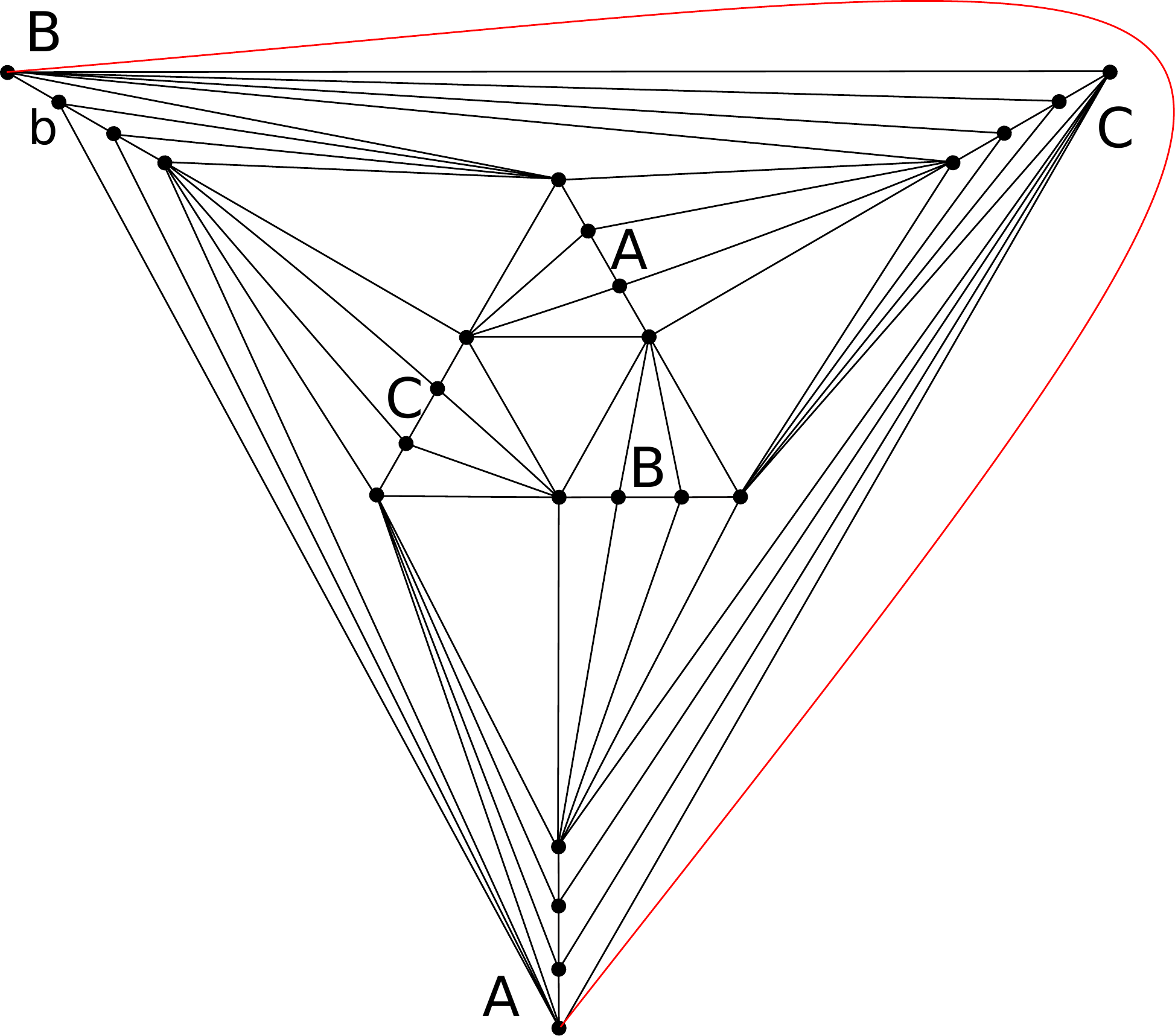}
		\end{center}
		\caption{Redraw a copy of Figure~\ref{fig:empire-graph-case-2} so that $b$ is on the exterior.}
		\label{fig:empire-graph-odd-5}
	  \end{figure}
		
      \begin{figure}[htb]
	    \begin{center}
	      \includegraphics[height=0.25\textheight]{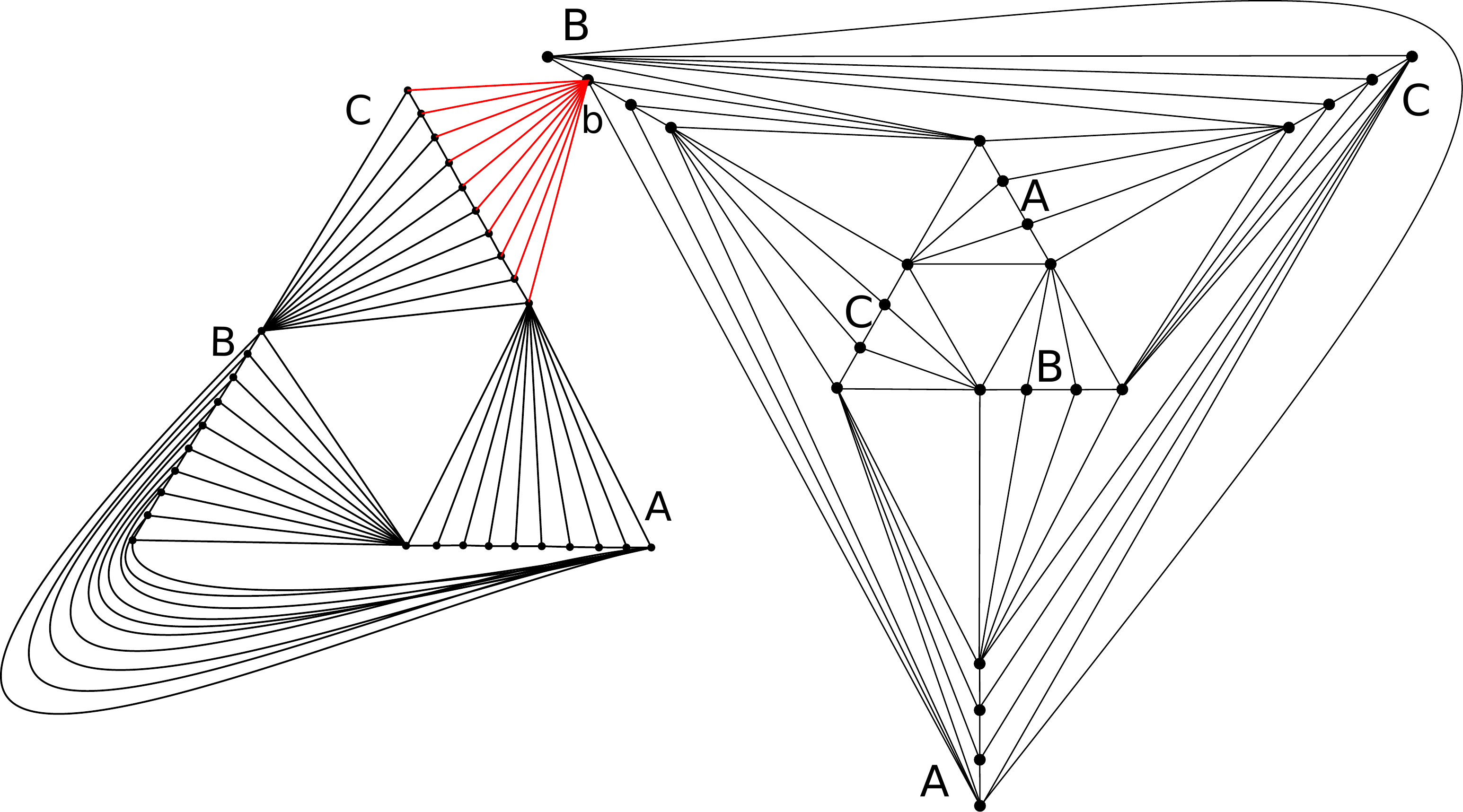}
		\end{center}
		\caption{Connect the exposed copy of $b$ to all vertices of $C$.}
		\label{fig:empire-graph-odd-6}
	  \end{figure}
		
      \begin{figure}[htb]
	    \begin{center}
	      \includegraphics[height=0.25\textheight]{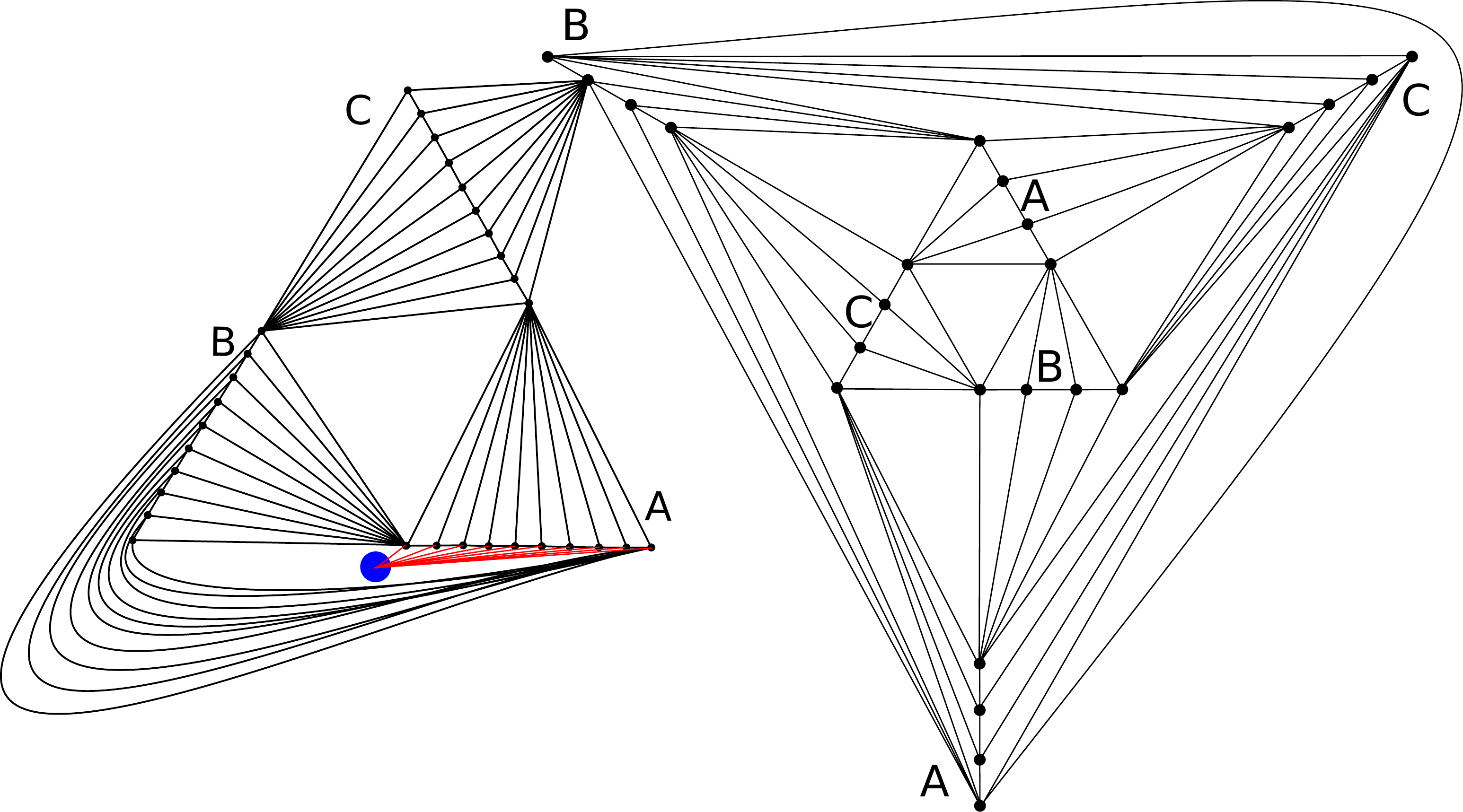}
		\end{center}
		\caption{Place another copy of Figure~\ref{fig:empire-graph-case-2} with $c$ exposed inside Figure~\ref{fig:empire-graph-odd-3} and connect it to all the vertices of $A$.}
		\label{fig:empire-graph-odd-7}
	  \end{figure}
		  
    Thus we have constructed a graph for all $m$ odd with $m \geq 5$, note that this is a single connected graph for $m=5$ and disconnected for all $m \geq 7$. We again need to check conditions (a) through (e) from Definition~\ref{def:complete-empire-graph}, however all the arguments from Lemma~\ref{lem:empire-planar-case-even} still hold since the basic construction principles are used, most importantly:
    \begin{itemize}
      \item The construction is still simple with appropriate degrees of vertices
      \item Paths taken from edge-disjoint Hamiltonian paths of $K_{2m}$ which guarantee the right number of empires
      \item Both endpoints of every $A$ path is adjacent to the whole of one of the $B$ paths a the same with $B, C$ endpoints adjacent to the whole of one of the $C, A$ paths respectively
    \end{itemize}
    Thus we can conclude we have constructed a $J(6m, m)$ graph and it is clearly planar as the union of planar graphs and thus we are done.
  \end{proof}
  
\section{Proof of Heawood's empire conjecture in the plane}
  Thus we are in a position to prove the main theorem which by this stage is reasonably trivial:
  \begin{theorem} \label{thm:heawood-equality-empire}
    $6m$ colours are sufficient to colour any $m$-pire map on a sphere for $m \geq 2$.
  \end{theorem}
  \begin{proof}
    Firstly we know from Theorem~\ref{thm:heawood-upper-empire} that you never require more than $6m$ colours thus to prove the statement it is sufficient to provide an example of a graph which requires $6m$ colours for each $m$. Secondly we recall that we can deal with $m$-pire graphs rather than maps. Thirdly we recall that any planar graphs can be embedded on the sphere so it is sufficient to find planar graphs for each $m$ which require $6m$ colours. Thus we can use Lemma~\ref{lem:complete-empire-graphs-number-colours} and say that we need only to show that there exists a planar $J(6m, m)$ for all $m \geq 2$ but this follows from Lemma~\ref{lem:empire-planar-case-even}, Lemma~\ref{lem:empire-planar-case-3} and Lemma~\ref{lem:empire-planar-case-odd}.
  \end{proof}
  
  \begin{remark}
    The observant reader will notice that there is an inconsistency with the proof provided by Wessel and our definition of an empire graph, namely that an empire graph must be connected. For the cases $m = 2, 3, 5$ the graphs produced by Wessel are indeed connected and fully satisfy the criteria. However, for all other m they are not since they are all the disjoint union of some connected graphs. This can be explained away by how one defines an empire graph; however. I would argue that a disconnected graph makes no physical sense because if you tried to convert back to an empire map you would not be able to. This is because you would have two countries for which there was no way of ``walking" from one to another, i.e. you would have to have bits missing from the world!
    
	This is not a major problem since we can make them connected by simply joining two vertices from different disjoint parts. Since we are not creating any vertices every condition will still be satisfied so long as we do not create any loops or multiple edges. However since we are joining two previously disconnected vertices this will also not be an issue.
	
	So, to describe the process explicitly, to convert one of the graphs produced by Wessel, $G$, into a valid empire graph as defined in this report you can do the following:
	\begin{enumerate}
	  \item Let $G = \sqcup_{i=1}^s G_i$ be the decomposition of G into path connected components
	  \item If $s = 1$ then we are done.
	  \item If $s > 1$ select any $G_i$ , $G_j$ with $i != j$
	  \item Pick any vertex $v \in G_i$ which is on the exterior, similarly for $w ∈ G_j$
	  \item Since $v, w$ are on the exteriors of their graphs you can then simply join them by an edge without create any crossings on the sphere
	  \item Go to Step 1 (with $s$ now one less that before)
	\end{enumerate}
  \end{remark}

  \chapter{Empire maps for higher genus surfaces} \label{chp:torus}
We now move away from the plane to higher genus surfaces e.g. the torus, and explore the colouring of empire maps on these surfaces to see if we can establish similar results to the ones we have achieved in the plane.

\section{Known results}
  First we define the the analogue of the chromatic number of a surface but for empire maps:
  \begin{definition}
    For a closed orientable surface $S$ of genus $g$ we define the \textbf{chromatic empire number}, $h_{g,m}$ as the maximum number of colours required to colour any $m$-pire map on $S$.
  \end{definition}
  
  This area started, as with everything else we have discussed, with Heawood. He provided an upper bound for higher genus surfaces which we provide without proof:
  \begin{theorem} \label{thm:genus-upper-bound}
    For an orientable surface $S$ with genus $g$ and chromatic empire number $h_{g,m}$ we have that:
    \begin{equation} \label{eqn:genus-upper-bound}
      h_{g,m} \leq \left\lfloor \frac{6m + 1 + \sqrt{(6m + 1)^2 + 24(2g - 2)}}{2} \right\rfloor
    \end{equation}
  \end{theorem}
	
  As before Heawood conjectured that this upper bound was sharp for all $g,m$ and it is this conjecture that we will be looking at extending:
  \begin{conjecture} \label{con:empire-genus-equality}
    For an orientable surface $S$ with genus $g$ with chromatic empire number $h_{g,m}$ we have that:
    \begin{equation} \label{eqn:genus-empire-equality}
      h_{g,m} = \left\lfloor \frac{6m + 1 + \sqrt{(6m + 1)^2 + 24(2g - 2)}}{2} \right\rfloor
    \end{equation}
  \end{conjecture}

\section{Simplifying the upper bound}
  A complete proof of Conjecture~\ref{con:empire-genus-equality} is not known but a large amount of work has been done to make inroads into the problem and here we bring together the known results and then looks towards trying to confirm the conjecture for some of the unknown cases.
  
  To help attack this problem we reproduce two Lemmas which were first given by Heawood and for which we provide a proof. The purpose of these Lemmas is to simplify the upper bound for $h_{g,m}$ under certain conditions:
  \begin{lemma} \label{lem:simplify-h-sphere}
    For $S$ a sphere with $h_{0,m}$ the chromatic empire number we have:
    \begin{displaymath}
      h_{0,m} \leq 6m
    \end{displaymath}
  \end{lemma}
  \begin{proof}
    We use proof by contradiction, so assume that for some $m$ we have that $h_{0,m} \geq 6m + 1$. Then we have the following:
    \begin{displaymath}
      \begin{aligned}
        6m + 1 &\leq \left\lfloor \frac{6m + 1 + \sqrt{(6m + 1)^2 - 48)}}{2} \right\rfloor \\
        &= \left\lfloor 3m + \frac{1}{2} + \frac{\sqrt{(6m + 1)^2 - 48)}}{2} \right\rfloor \\
      \end{aligned}
    \end{displaymath}
    Then we note that the floor function only makes numbers smaller giving us:
    \begin{displaymath}
      \begin{aligned}
        6m + 1 &\leq \left\lfloor 3m + \frac{1}{2} + \frac{\sqrt{(6m + 1)^2 - 48)}}{2} \right\rfloor \\
        \Rightarrow 6m + 1 &\leq 3m + \frac{1}{2} + \frac{\sqrt{(6m + 1)^2 - 48)}}{2} \\
        \Rightarrow 3m + \frac{1}{2} &\leq \frac{\sqrt{(6m + 1)^2 - 48)}}{2} \\
        \Rightarrow 6m + 1 &\leq \sqrt{(6m + 1)^2 - 48)}
      \end{aligned}
    \end{displaymath}
    We then note that since $m \geq 1$ both sides are positive and so we square both sides to get:
    \begin{displaymath}
      \begin{aligned}
        (6m + 1)^2 &\leq (6m + 1)^2 - 48 \\
        \Rightarrow 0 &\leq -48
      \end{aligned}
    \end{displaymath}
    which is clearly a contradiction and so $h_{0,m} < 6m + 1$ or since we are dealing with integers $h_{0,m} \leq 6m$ as required.
  \end{proof}
  
  \begin{lemma} \label{lem:simplify-h-higher-genus}
    For a surface $S$ with genus $g > 0$ we have that:
    \begin{displaymath}
      h_{g,m} \leq 6m + 1
    \end{displaymath}
    for $1 \leq g \leq \frac{1}{2}(m + 2)$.
  \end{lemma}
  \begin{proof}
     We can rearrange the the second half of the condition on $g$ to get (noting that the first half is satisfied by assumption):
    \begin{displaymath}
      m \geq 2g - 2
    \end{displaymath}
    We then use a similar contradiction argument as used in the proof of Lemma~\ref{lem:simplify-h-sphere} and assume that for some $m, g$ satisfying $m \geq 2g - 2$ we have $h_{g,m} \geq 6m + 2$ so we get:
    \begin{displaymath}
      \begin{aligned}
        6m + 2 &\leq \left\lfloor \frac{6m + 1 + \sqrt{(6m + 1)^2 + 24(2g - 2)}}{2} \right\rfloor \\
        \Rightarrow 6m + 2 &\leq \frac{6m + 1 + \sqrt{(6m + 1)^2 + 24(2g - 2)}}{2} \\
        \Rightarrow 3m + \frac{3}{2} &\leq \frac{\sqrt{(6m + 1)^2 + 24(2g - 2)}}{2} \\
        \Rightarrow 6m + 3 &\leq \sqrt{(6m + 1)^2 + 24(2g - 2)} \\
        \Rightarrow 6m + 3 &\leq \sqrt{(6m + 1)^2 + 24m} \\
        \Rightarrow (6m + 3)^2 &\leq (6m + 1)^2 + 24m \\
        \Rightarrow 36m^2 + 36m + 9 &\leq 36m^2 + 36m + 1 \\
        \Rightarrow 9 &\leq 1
      \end{aligned}
    \end{displaymath}
    which is clearly a contradiction and therefore $h_{g,m} < 6m + 2$ and again since we are working with integer values we have that $h_{g,m} \leq 6m + 1$ as required.
  \end{proof}
  
  We provide one more lemma which although it doesn't actually simplify the upper bound it does allow us to extend work to higher genus immediately. The idea is that if we already have an upper bound on $h_{g,m}$ for all values of $g,m$ and what we would like to do is find a way to provide a lower bound as well, thus reducing the unknown cases. Intuitively the idea is that if we can provide an lower bound on $h_{g,m}$ for a surface $S$ then this lower bound also applies to all surfaces $S'$ which have a higher genus than $S$. More formally:
  \begin{lemma} \label{lem:extend-results-higher-genus}
    For a fixed value of $m \geq 1$, given a surface $S$ with genus $g$ and empire chromatic number $h_{g,m}$, if $h_{g,m} \geq h$ for some $h$, then for any surface $S'$ with genus $g' \geq g$ and empire chromatic number $h_{g',m}$, we have that $h_{g',m} \geq h$.
  \end{lemma}
  \begin{proof}
    Since $h_{g,m} \geq h$ we know that there exists an empire map on $S$ which requires $h$ colours to colour properly. This empire map must have a dual graph $G'$ which can also be embedded on $S$. Then since any graph can be embedded on surfaces of higher genus we can say that $G'$ can also be embedded on $S'$. We can then convert this dual graph back to an empire map on $S'$ which also requires $h$ colours to colour properly and thus $h_{g',m} \geq h$.
  \end{proof}
  
\section{Summary of known results}
  What we will now do is to cover all the known results made towards Conjecture~\ref{con:empire-genus-equality}, most of these results we have already mentioned and in some cases proved but we provide them again here in the new notation and for completeness:
  \begin{corollary} \label{cor:summary-known-results}
    For a surface $S$ of genus $g$ with empire chromatic number $h_{g,m}$ we have that:
    \begin{displaymath}
      h_{g,m} = \left\lfloor \frac{6m + 1 + \sqrt{(6m + 1)^2 + 24(2g - 2)}}{2} \right\rfloor
   \end{displaymath}
   for the following cases:
   \begin{enumerate} [(a)]
     \item $g = 0$, $m \geq 1$
     \item $g = 1$, $m \geq 1$
     \item $g = 2$, $m \geq 1$
    \end{enumerate}
  \end{corollary} 
  \begin{proof}
    We discuss each case separately:
    \begin{enumerate} [(a)]
      \item Case $g=0$ (sphere), $m \geq 1$. We consider two cases, $m=1$ and $m > 1$. Firstly take $m=1$, this then becomes the Four Colour Problem which has been solved and so $h_{0,1}$ should equal 4 and indeed simply calculation shows that this is the case. \newline
     So consider $m > 1$ then we know from Lemma~\ref{lem:simplify-h-sphere} that $h_{0,m} \leq 6m$ but we also know from Theorem~\ref{thm:planar-empire-maps} that there exists a map which requires $6m$ colours and thus $h_{0,m} \geq 6m$ and so we get $h_{0,m} = 6m$ as required.
      \item Case $g=1$ (torus), $m \geq 1$. We do not provide a proof for the torus. It was first proved in a piecemeal fashion by Jackson and Ringel however the proof was long and indirect. In 1997 Walter Wessel \cite{wessel-torus} provided a shorter, elementary, uniform and constructive proof (again which we do not provide). \newline
     Note that by Lemma~\ref{lem:simplify-h-higher-genus} $h_{g,m} \geq 6m + 1$ for $m \geq 2g-2$ which in this case is always true so we have that for the case $g=1$ that $h_{1,m} = 6m + 1$ for all $m \geq 1$.
      \item Case $g=2$ (double torus), $m \geq 1$. We have shown (or stated!) in part (b) that for $g=1$ we have $h_{1,m} = 6m + 1$ so by Lemma~\ref{lem:extend-results-higher-genus} that $h_{2,m} \geq 6m + 1$. But by Lemma~\ref{lem:simplify-h-higher-genus} we have that $h_{2,m} \leq 6m + 1$ for $m \geq 2g - 2 = 2$ so this proves the result for all cases except the case $m=1$. \newline
     For the case $g=2, m=1$ we note that this involves no empires and is therefore the Map Colour Theorem which as we mentioned earlier (Section~\ref{sec:heawood-conjecture-normal}) has been proven and so we are done.
    \end{enumerate}
  \end{proof}
   
  So the it has been shown that Conjecture~\ref{con:empire-genus-equality} holds for the sphere, torus and double torus and in fact it is not known completely for any other surfaces. However there were two arguments that were used in the proof of Corollary~\ref{cor:summary-known-results} for the double torus which hold for all higher genus surfaces and we will state these separately as these are important results in there own right. The first is simply the Map Colour Theorem which we can restate in terms of our new notation:
  \begin{theorem} \label{thm:map-colour-new-notation}
    For all (orientable) surfaces $S$ with genus $g$ and empire chromatic number $h_{g,m}$ we have:
    \begin{displaymath}
      h_{g,1} = \left\lfloor \frac{6m + 1 + \sqrt{(6m + 1)^2 + 24(2g - 2)}}{2} \right\rfloor
    \end{displaymath}
  \end{theorem}

  The second important result is a corollary of the proof of Corollary~\ref{cor:summary-known-results} for the torus:
  \begin{corollary} \label{cor:equality-for-high-m}
    For all (orientable) surfaces $S$ with genus $g > 0$ and empire chromatic number $h_{g,m}$ we have:
    \begin{displaymath}
      h_{g,m} = 6m + 1
    \end{displaymath}
    for all $m \geq 2g - 2$
  \end{corollary}
  \begin{proof}
    This follows as an immediate consequence of Corollary~\ref{cor:summary-known-results}, Lemma~\ref{lem:extend-results-higher-genus} and Lemma~\ref{lem:simplify-h-higher-genus}.
  \end{proof}

\section{The triple-torus}
  So we now look to extend upon the known results and the most obvious place to start is with the case $g=3$. However before we can start working with a genus 3 surface we need a way of representing it. We use an approach similar to that used in Section~\ref{sec:classification-of-surfaces} and more specifically we will use the approach we used to represent the torus as in Figure~\ref{fig:making-torus-quotient}. To decide how we can find a similar representation for a triple-torus we use an approach outlined in Gay \cite[Chapter 10]{gay} which allows you to describe polygons as symbols. The first step is to create a way of converting between a series of symbols and a polygon, so we shall start by converting from the image we had of a torus into symbols:
  \begin{enumerate} [(i)]
    \item Start by taking an polygon for which each side has a direction and a label (note each label may appear at most twice)
    \item Pick any vertex
    \item Starting at this vertex go clockwise round the polygon writing down the label of each side in order with the caveat that if the direction of the edge is not going clockwise around the polygon you write it down as the inverse of the label
    \item Stop when you arrive back at your starting vertex
  \end{enumerate}
  
  \begin{figure}[htb]
	\begin{center}
	  \includegraphics[height=0.3\textwidth]{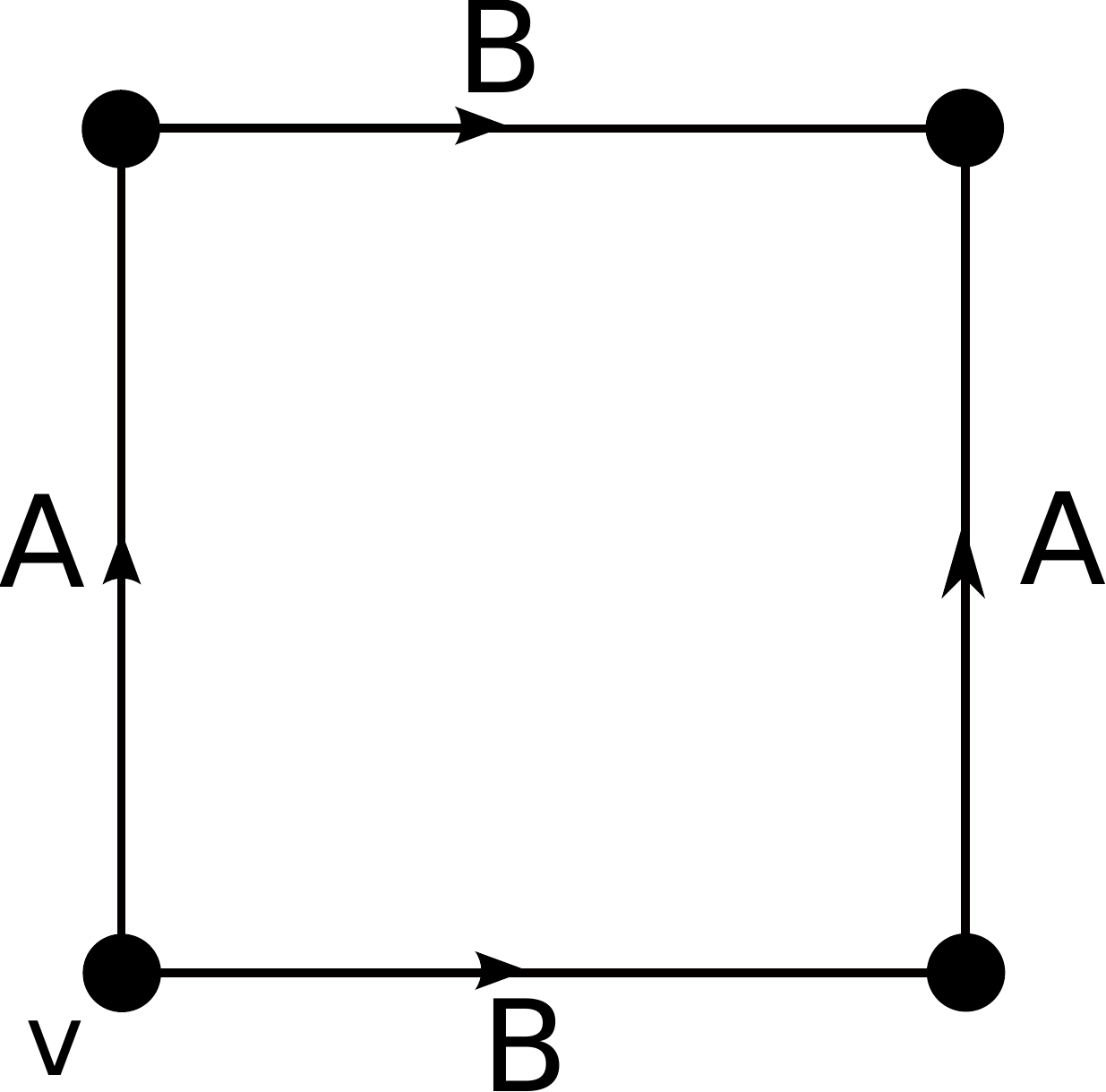}
    \end{center}
    \caption{The torus as the quotient of the square.}
    \label{fig:torus-quotient}
  \end{figure}
  
  \begin{example}
    So we take the representation of the torus as shown in Figure~\ref{fig:torus-quotient} and start at the vertex labelled $v$ and we go clockwise and therefore obtain the series of symbols: $ABA^{-1}B^{-1}$ \newline
    This is then a ``surface symbol" for the torus, it is not the only possible surface symbol for the torus but it does only represent the torus, for a rigorous proof of this see \cite[Chapter 10]{gay}. To convert back from a surface symbol we simply reverse the process as you would expect. \newline
  \end{example}
  
  So now we want to find a representation of the triple-torus, and again this is discussed in Gay \cite[Chapter 10]{gay} which states that for a double torus we simply combine two surface symbols for the torus i.e. $ABA^{-1}B^{-1}CDC^{-1}D^{-1}$ and for the triple torus we get: $ABA^{-1}B^{-1}CDC^{-1}D^{-1}EFE^{-1}F^{-1}$. For convenience we can also manipulate this to obtain equivalent surface symbols and there are many rules given in Gay for how to manipulate surface symbols but we are interested only in the following two rules: (Greek letters can represent a (possibly empty) sequence of symbols, standard letters represent a single symbol)
    \begin{rules} \label{rul:right-to-left}
      \begin{displaymath}
        \alpha X \beta \gamma X^{-1} \delta \sim \alpha X \gamma \beta X^{-1} \delta
      \end{displaymath}
      ``A sequence of symbols to the right of an $X$ can be moved to the left of an $X^{-1}$.
    \end{rules}
    \begin{rules} \label{rul:left-to-right}
      \begin{displaymath}
        \alpha \beta X \gamma X^{-1} \delta \sim \alpha X \gamma X^{-1} \beta \delta
      \end{displaymath}
      ``A sequence of symbols to the left of an $X$ can be moved to the right of an $X^{-1}$.
    \end{rules}
    
    We can then use these two rules to change our surface symbol for a triple-torus as follows:
    \begin{displaymath}
      \begin{aligned}
        ABA^{-1}B^{-1}CDC^{-1}D^{-1}EFE^{-1}F^{-1} &\sim ABA^{-1}CDC^{-1}B^{-1}D^{-1}EFE^{-1}F^{-1} \\
        &\sim ABCDC^{-1}A^{-1}B^{-1}D^{-1}EFE^{-1}F^{-1} \\
        &\sim ABCDA^{-1}B^{-1}C^{-1}D^{-1}EFE^{-1}F^{-1} \\
        &\sim ABCDEFE^{-1}A^{-1}B^{-1}C^{-1}D^{-1}F^{-1} \\
        &\sim ABCDEFA^{-1}B^{-1}C^{-1}D^{-1}E^{-1}F^{-1}
      \end{aligned}
    \end{displaymath}
    
    So we get the following two representations:
    \begin{displaymath}
      ABA^{-1}B^{-1}CDC^{-1}D^{-1}EFE^{-1}F^{-1} \sim ABCDEFA^{-1}B^{-1}C^{-1}D^{-1}E^{-1}F^{-1}
    \end{displaymath}
    
    If we then convert these back to polygons we get the two representations as shown in Figure~\ref{fig:triple-torus-quotient}. As it turns out we will be using the right-hand representation but both are equally valid as are any many other representations which we could arrive at, but these have the virture of simplicity and symmetry.
  
  \begin{figure}[htb]
	\begin{center}
	  \includegraphics[height=0.25\textheight]{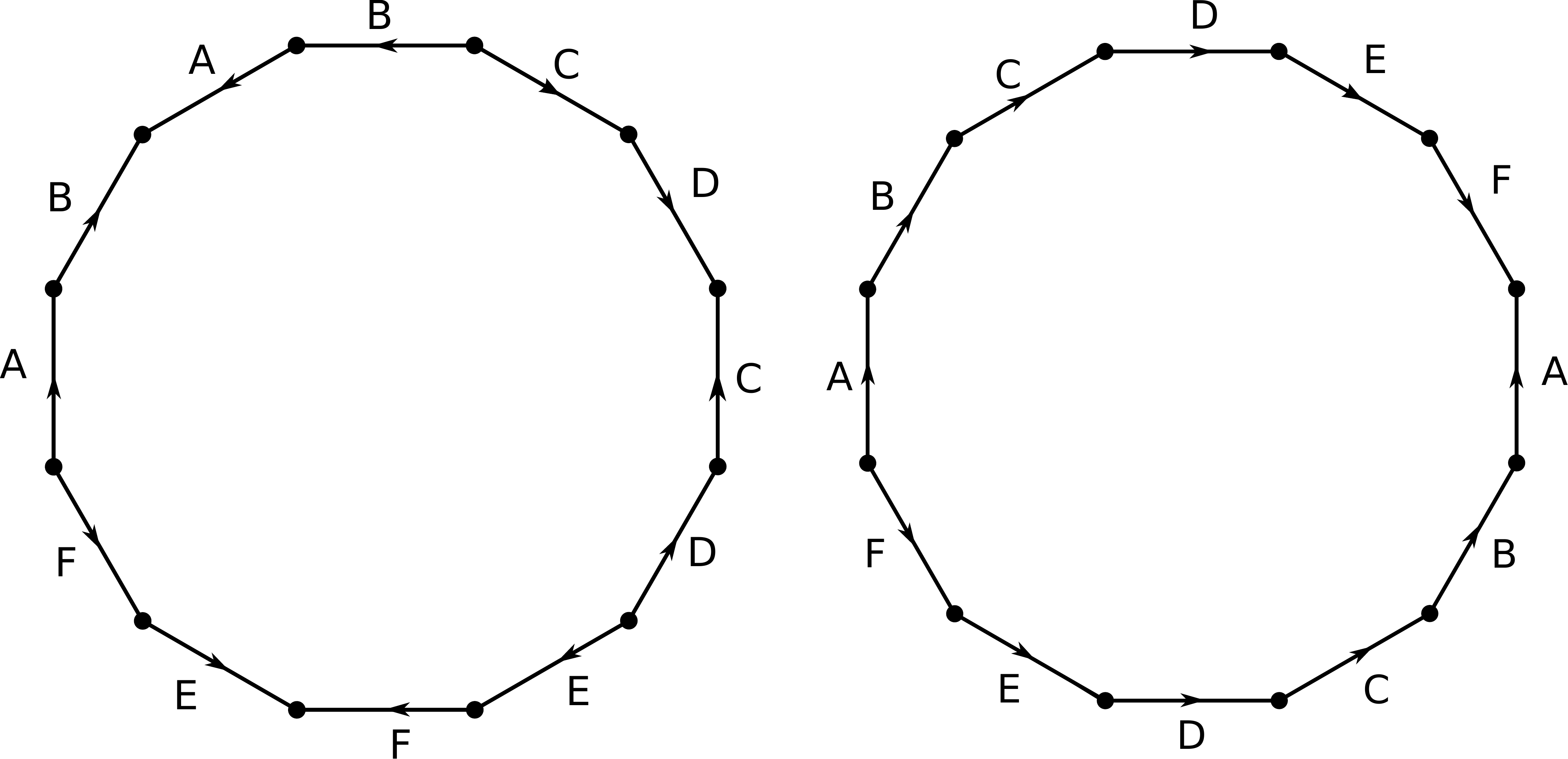}
    \end{center}
    \caption{The triple-torus as the quotient of a polygon.}
    \label{fig:triple-torus-quotient}
  \end{figure}
  
\section{Extending known results} \label{sec:new-proof}
  Now we can deal with the triple torus and have a representation which we can draw graphs on we can began to tackle the unknown cases for a genus 3 surface. Most of the work has been done; Theorem~\ref{thm:map-colour-new-notation} states that for $m=1$ we have equality. Also, Corollay~\ref{cor:equality-for-high-m} says that for $m \geq 2g - 2 = 4$ we have equality. Thus we have only two cases to prove: $m=2$ and $m=3$. We shall be considering the case $m=3$. In this case we get the following:
  \begin{displaymath}
    \begin{aligned}
      h_{3,2} &= \left\lfloor \frac{6 \cdot 2 + 1 + \sqrt{(6 \cdot 2 + 1)^2 + 24(2 \cdot 3 - 2)}}{2} \right\rfloor \\
              &= 14
    \end{aligned}
  \end{displaymath}
  
  \begin{remark}
    We know from Corollary~\ref{cor:equality-for-high-m} that on the torus, with $m = 2$, we have that $h_{1,2} = 13$. Then if we use Lemma~\ref{lem:extend-results-higher-genus} we can see that $h_{3,2} \geq 13$, so the only possibilities are $h_{3,2} =$ 13 or 14.
  \end{remark}

  This then leads us to the following theorem:  
  \begin{theorem} \label{thm:proof-case-g3-m2}
    For a surface $S$ with genus 3 we have that the chromatic 2-pire number is $h_{3,2} = 14$
  \end{theorem}
  \begin{proof}
    To prove this we simply need to provide a graph which requires 14 colours and indeed we can do this by producing a $J(14,2)$ graph which we can embed on a triple torus and I claim the graph given in Figure~\ref{fig:proof-g3-m2} is such a graph. An introduction to methods used when looking for such graphs is given in Appendix~\ref{app:tools}.
	
	\begin{figure}[htb]
	  \begin{center}
	  	\includegraphics[scale=0.6]{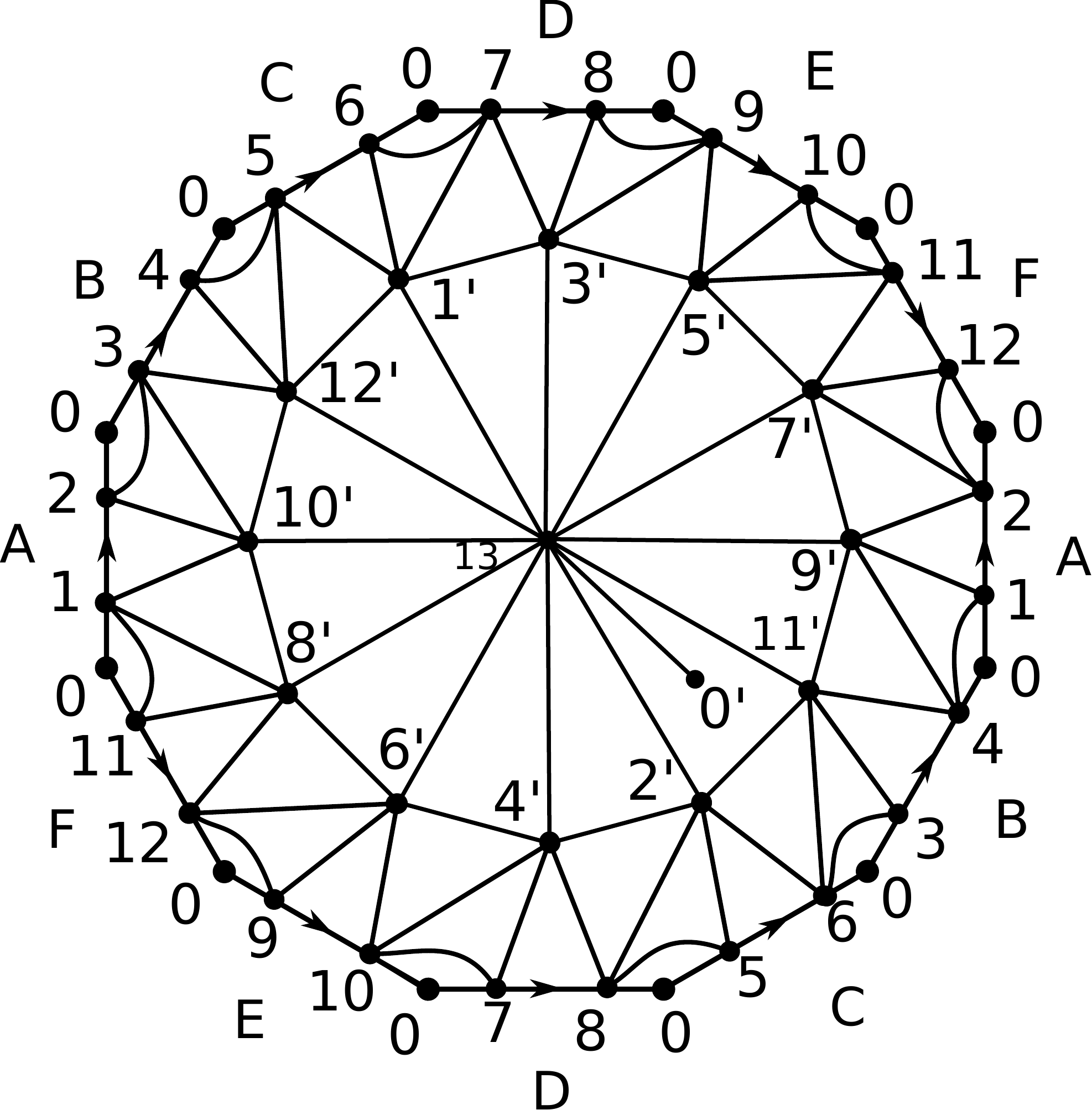}
	  \end{center}
	  \caption{Proof of Theorem~\ref{thm:proof-case-g3-m2},}
	  \label{fig:proof-g3-m2}
	\end{figure}
	
	\begin{remark}
	  It is important to realise that the graph is still embedded on the quotient of a polygon, as shown in Figure~\ref{fig:triple-torus-quotient}, so the sides labelled $A$ are glued together, similarly for $B, C$ etc. This means that all 12 vertices labelled 0 actually end up becoming a single vertex on the triple-torus, similarly the two vertices labelled 1 become the same vertex on the triple torus etc.
	\end{remark}

    We need to check that the graph provided is indeed a $J(14, 2)$ graph. To do this we check the four conditions given in Lemma~\ref{lem:complete-empire-sufficient-conditions} which were:
	\begin{enumerate} [(a)]
      \item The graph is simple
      \item The graph can be partitioned into exactly 14 vertex empires
      \item Each vertex empire contains at most 2 vertices
      \item For any two empires there exist two vertices, one from each, which are adjacent
	\end{enumerate}
    So we check each condition:
    \begin{enumerate} [(a)]
      \item Inspection shows that there are no loops or multiple edges
      \item If we let $0, 0'$ be a vertex empire, i.e. $0$ and $0'$ represent different countries from the same empire, and similarly $1, 1'$ and $2, 2'$ etc. we can see that there are exactly 14 vertex empires from 0 to 13
      \item Again inspection shows that there are exactly 2 vertices in each empire except the last empire where $13'$ isn’t present so that empire contains only one vertex, labelled 13
	  \item To show this we provide an adjacency table for all the vertices in the graph. It can be clearly seen that one of the vertices in each empire is adjacent to a vertex from every other empire: \newline
      \begin{tabular} {l|l}
        Vertex & Adjacent vertices \\ \hline
        $0$   & $1,2,3,4,5,6,7,8,9,10,11,12$   \\
        $0'$  & $13'$                          \\ \hline
        $1$   & $0, 2, 4, 8', 9', 10', 11$     \\
        $1'$  & $3', 5, 6, 7, 12', 13$         \\ \hline
        $2$   & $0, 1, 3, 7', 9', 10', 12$     \\
        $2'$  & $4', 5, 6, 8, 11', 13$         \\ \hline
        $3$   & $0, 2, 4, 6, 10', 11', 12'$    \\
        $3'$  & $1', 5', 7, 8, 9, 13$          \\ \hline
        $4$   & $0, 1, 3, 5, 9', 11', 12'$     \\
        $4'$  & $2', 6', 7, 8, 10, 13$          \\ \hline
        $5$   & $0, 1', 2', 4, 6, 8, 12'$      \\
        $5'$  & $3', 7', 9, 10, 11, 13$        \\ \hline
        $6$   & $0, 1', 2', 3, 5, 7, 11'$      \\
        $6'$  & $4', 8', 9, 10, 12, 13$        \\ \hline
        $7$   & $0, 1', 3', 4', 6, 8, 10$      \\
        $7'$  & $2, 5', 9', 11, 12, 13$        \\ \hline
        $8$   & $0, 2', 3', 5, 7, 9$           \\
        $8'$  & $1, 4', 6', 10', 11, 12, 13$   \\ \hline
        $9$   & $0, 3', 5', 6', 8, 10, 12$     \\
        $9'$  & $1, 2, 4, 7', 11', 13$         \\ \hline
        $10$  & $0, 4', 5', 6', 7, 9, 11$      \\
        $10'$ & $1, 2, 3, 8', 12', 13$         \\ \hline
        $11$  & $0, 1, 5', 7', 8', 10, 12$     \\
        $11'$ & $2', 3, 4, 6, 9', 13$          \\ \hline
        $12$  & $0, 2, 5, 6', 7', 8', 9, 11$   \\
        $12'$ & $1', 3, 4, 10', 13$            \\ \hline
        $13$  & $0,1,2,3,4,5,6,7,8,9,10,11,12$
      \end{tabular} \newline 
    \end{enumerate}
    
    Thus we have constructed an embedding of a $J(14, 2)$ on the triple-torus which by Lemma~\ref{lem:complete-empire-graphs-number-colours} requires 14 colours. Thus we have that $h_{3,2} \geq 14$ but by Theorem~\ref{thm:genus-upper-bound} $h_{3,2} \leq 14$ and thus $h_{3,2} = 14$ as required.
  \end{proof}

  \chapter{Conclusion} \label{chp:conclusion}

\section{Recap of Results Covered}
  There are four main results which we have discussed in this report:
  \begin{enumerate}
    \item \textbf{The Four Colour Theorem:} proven by Appel and Haken using computers in 1976
    \item \textbf{The Four Colour Theorem:} proven by Ringel and Youngs in 1968
    \item \textbf{Heawood’s Conjecture for Empire Maps in the Plane:} proven by Jackson and Ringel in 1984 with a constructive proof by Wessel in 1997 which is comprehensively reviewed in Chapter~\ref{chp:empire-conjecture}
    \item \textbf{Heawood’s Conjecture for Empire Maps on the Torus:} proven by Wessel in 1997
  \end{enumerate}

  These four results (and some trivial corollaries) currently hold the entirety of what we know about colouring empire maps on orientable surfaces.
  
  In this report we have introduced some mathematical grounding about such areas and discussed briefly some of the mathematics underpinning some of the key results.

  One thing to note when discussing this area is that so much work was originally done by Heawood in his paper of 1890 which is not easily accessible. Thus many of the results which are known and quoted must be taken as read with his particular methods not always known. Also for the smaller results original proofs must be found, as was largely the case in the smaller results of Chapter~\ref{chp:torus}.

  Also particularly important was the completion of the proof by Hutchinson. A result which was again originally proven by Heawood, but with Hutchinson leaving a key part as an exercise, and without Heawood’s paper to fall back on, an original proof had to be found. Thus there were 3 main bits of original work in this report:
  \begin{enumerate}
    \item Proof of an essential claim (left as an exercise) in a proof by Hutchinson. Proved in Section~\ref{sec:proof-of-claim}.
    \item Proof of a previously unknown case towards Heawood’s Empire Conjecture, specifically, 2-pire maps on the triple torus. Proved in Section~\ref{sec:new-proof}.
    \item Introduction of slack and proof that slack is always non-zero for $J(m, n)$ graphs. Proved in the appendix.
  \end{enumerate}
  
  I will conclude this section by making a few comments about reading around the subject. Firstly, since most of the work completed in this area (Heawood excluded!) was completed relatively recently, it is freely available online.

  Following on from this, my second comment is that the terminology in this area is not standardised. In each paper you will see different terminology for the same objects. In particular, the definitions of empire maps and empire graphs as given in this report (as triples, see Definition~\ref{def:empire-map}) were original definitions. The underlying object is a common one but rarely rigourously defined. Also the
concept of a collapsed empire graph in some places, e.g. in Hutchinson \cite{hutchinson}, is just referred to as an empire graph. So be aware when reading each paper that the terminology may vary.

  Lastly, as has been mentioned before in the report many of the results also hold for non-orientable surfaces. However, similar to the original proofs by Ringel and others, they are not constructive and so provide little guidance on constructing the limiting graphs.

\section{Unanswered Questions}
  The area of map colouring is still a reasonably fertile area of mathematics. Most obviously there are still an infinite number of unknown cases towards Heawood's Empire Conjecture. Further single results can be found in this area as was done in this report using techniques outlined in the Appendix. However a general result could be difficult, certainly by producing graphs, due to the irregularity of the upper bound, arising from the floor function. If one wanted to find a general result they might have to simplify the upper bound first, in a way similar to that in Lemma~\ref{lem:simplify-h-higher-genus}.

  Thus I believe there are a few main directions for further research:
  \begin{enumerate}
    \item Find a simpler proof to the Four Colour Theorem which doesn’t rely on computers
    \item Complete some more unknown cases towards Heawood’s Empire Conjecture
    \item Find a general solution of Heawood’s Empire Conjecture
    \item Investigate other extensions, e.g. on disconnected surfaces, as mentioned in Section~\ref{sec:history}
  \end{enumerate}

  \appendix
  \chapter{Investigative Tools} \label{app:tools}

This appendix contains some of the wisdom which I have gleaned whilst investigating this project and is aimed at those looking to extend the work in a similar direction to the one outlined. It is primarily centered around looking for $J(a, b)$ graphs which prove unknown cases of Conjecture~\ref{con:empire-genus-equality}.
  
\section{A Special Graph}
  We start by going back a step; before settling on the idea of $J(a, b)$ graphs one might naturally (as I did) first think about a more specific type of graph, namely complete uniform graphs, as defined in Section~\ref{sec:intro-empire-maps}. The conditions are stricter than $J(a, b)$ graphs:
  \begin{definition}
    A \textbf{complete uniform $m$-pire graph on $n$ empires}, is a complete $m$-pire graph, $(G, A, h)$, with the following extra conditions:
    \begin{enumerate} [(a)]
      \item There are exactly $nm$ vertices in $G$. Or equivalently, each vertex empire contains exactly $m$ vertices
      \item For any two vertex empires, there exists \emph{uniquely} a pair of vertices, one from each empire, which are adjacent
    \end{enumerate}
    For ease of notation these graphs shall be denoted $\bar{J}(n, m)$ graphs
  \end{definition}
  
  \begin{remark}
    Note that these $\bar{J}(n, m)$ graphs are still a property of graphs as in the case with $J(n, m)$ graphs but just much stricter conditions. In fact, two distinct $\bar{J}(n, m)$ graphs can only differ very slightly, namely one can alter the degree of each vertex within an empire. Although the total degree of all vertices in one empire must always stay the same.
  \end{remark}
  
  We start by producing some very simple lemmas about these graphs:
  \begin{lemma} \label{lem:complete-uniform-counting}
    A $\bar{J}(n, m)$ graph contains exactly $nm$ vertices and exactly $\frac{1}{2}n(n - 1)$ edges.
  \end{lemma}
  \begin{proof}
    The number of vertices is given in the definition. 
    
    We can calculate the number of edges by noting that condition (b) requires that there is exactly one edge between any two vertex empires, thus since each of the $n$ vertex empires is joined to each of the other $(n - 1)$ vertex empires we get $n(n - 1)$ edges, although we have clearly double counted giving us $\frac{1}{2} n(n - 1)$ edges.
  \end{proof}

\section{Slack}
  So we have a special subset of $J(n, m)$ graphs and as such can be seen as a base point from which all other $J(a, b)$ graphs can be obtained, the question being how much can they be changed? One way to explore this question is through a concept I called slack which is based around the result in Lemma~\ref{lem:simple-graph-three-sides}. The idea is that if we can embed a simple graph on a surface then we have that $2E \geq 3C$, but exactly how much bigger:
  \begin{definition}
    For a simple graph $G$ embedded on a surface $S$ we define the \textbf{slack}, $s(G, S)$ as:
    \begin{displaymath}
      s(G, S) = 2E - 3C
    \end{displaymath}
    where $E$ and $C$ are the number of edges and countries in the embedding of $G$ respectively.
  \end{definition}
  
  \begin{example}
    One of the benefits of considering slack is to show that a graph \textbf{cannot} be embedded in a surface $S$. For example we do like we did for proving $K_5$ isn't planar and assume that a graph $G$ \textbf{can} be embedded on a surface $S$. Then we calculate the slack $s = s(G, S)$, then if $s < 0$ we know that $G$ can't be embedded in $S$. Note that if $s \geq 0$ this doesn't guarantee that $G$ can be embedded.
  \end{example}
  
  So returning back to the problem of proving cases towards Conjecture~\ref{con:empire-genus-equality}. We consider the following scenario; let $g$ and $m$ both be fixed and let $h$ be defined as:
  \begin{displaymath}
    h = \left\lfloor \frac{6m + 1 + \sqrt{(6m + 1)^2 + 24(2g - 2)}}{2} \right\rfloor
  \end{displaymath}
  then what we would like to do is embed $\bar{J}(h, m)$ on a surface $S$ of genus $g$. So let's assume that we can embed a $\bar{J}(h, m)$ graph on $S$, what is the slack? Is it ever going to be less than zero? Well we can show that it is always non-negative:
  \begin{theorem}
    For a surface $S$ of genus $g > 0$ and for fixed $m \geq 1$, if we let $h$ be defined as:
    \begin{displaymath}
      h = \left\lfloor \frac{6m + 1 + \sqrt{(6m + 1)^2 + 24(2g - 2)}}{2} \right\rfloor
    \end{displaymath}
    if we let $G$ be a $\bar{J}(h, m)$ graph then $s(G, S) \geq 0$.
  \end{theorem}
  \begin{proof}
    Throughout this proof we let $C, E, V$ be the counties, edges and vertices in the embedding of $G$ respectively. Then the slack is $s = s(G, S) = 2E - 3C$ so we use proof by contradiction and assume that for some value of $g$ and some value of $m$ we have $s = 2E - 3C < 0$.
    
    Now we can use $G$ to calculate the Euler characteristic of $S$:
    \begin{displaymath}
      \begin{aligned}
        \chi(S) &= C - E + V \\
        \Rightarrow C &= E - V + \chi(S) \\
        \Rightarrow 3C &= 3E - 3V + 3\chi(S) \\
        \Rightarrow 3C &= 3E - 3V + 3(2 - 2g)
      \end{aligned}
    \end{displaymath}
    
    Then we can combine with the assumption to get:
    \begin{displaymath}
      2E - 3C = 2E - (3E - 3V + 3\chi(S)) = 3V - 3(2 - 2g) - E < 0
    \end{displaymath}
    
    But we know from Lemma~\ref{lem:complete-uniform-counting} how many vertices and edges there are, so substituting in we get:
    \begin{displaymath}
      3V + 3(2g - 2) - E = 3hm + 3(2g - 2) - \frac{1}{2}h(h-1) < 0
    \end{displaymath}
    
    We can then multiply through by -2 and we get a quadratic in $h$:
    \begin{displaymath}
      \begin{aligned}
        \Rightarrow h(h-1) - 6hm - 6(2g - 2) &> 0 \\
        \Rightarrow h^2 - h(6m + 1) - 6(2g - 2) &> 0 \\
        \Rightarrow (h - h_1)(h - h_2) &> 0
      \end{aligned}
    \end{displaymath}
    with $h_1$ and $h_2$ as follows:
    \begin{displaymath}
      h_1 = \frac{1}{2} \left( 6m + 1 - \sqrt{(6m + 1)^2 + 24(2g - 2)} \right)
    \end{displaymath}
    \begin{displaymath}
      h_2 = \frac{1}{2} \left( 6m + 1 + \sqrt{(6m + 1)^2 + 24(2g - 2)} \right)
    \end{displaymath}
    Then since the whole thing must be positive we know that the two factors must both be positive or both be negative. Then since we are considering $g > 0$ we get:
    \begin{displaymath}
      \begin{aligned}
        g > 0 &\Rightarrow 2g - 2 \geq 0 \\
        &\Rightarrow (6m + 1)^2 + 24(2g - 2) \geq (6m + 1)^2  \\
        &\Rightarrow h_1 = 6m + 1 - \sqrt{(6m + 1)^2 + 24(2g - 2)} \leq 0
      \end{aligned}
    \end{displaymath}
    So we get that $(h - h_1) \geq h > 0$ so we know that one factor is strictly positive and so the other factor must be strictly positive as well. But we have that $h = \lfloor h_2 \rfloor$ and so we get $h \leq h_2$ so $(h - h_2) \leq 0$ which isn't strictly positive and so we get a contradiction so $2E - 3C = s \geq 0$ as required.
  \end{proof}
  
  \begin{remark}
    The very important thing to note is that this does not say that we can always embed a necessary $\bar{J}(h, m)$ as this would have proven the conjecture for all cases! It merely says that one possible argument for showing it is impossible doesn't work!
  \end{remark}
  \begin{remark}
    One can also note that we only considered $g > 0$, the case $g=0$, the sphere, could be considered separately and one could also show that in this case the slack is always non-negative but since Corollary~\ref{cor:summary-known-results} has completely settled the case $g=0$ this seems a slightly pointless exercise!
  \end{remark}
  \begin{remark}
    From this proof we can also see that the slack is exactly zero if and only if $h = h_2$ i.e. $h$ is an integer so for example in the case $g=1$ we know that $h = 6m + 1$ and so the for all values of $m$ is zero.
  \end{remark}
  
\section{Using Slack}
  So we can continue our search for some $J(n, m)$ graph with a little more optimism! We have shown that there is always a possibility that we can embed some $\bar{J}(h, m)$ but what about a more general $J(h, m)$ graph. Well before I mentioned that $\bar{J}(h, m)$ was like a base graph and so we can obtain all the other $J(h, m)$ graphs by relaxing restrictions placed on $\bar{J}(h, m)$ such as:
  \begin{enumerate}
    \item Allowing fewer than $m$ vertices in each vertex empire
    \item Allowing more than one edge between any two vertex empires
    \item Allowing edges between two vertices in the same vertex empire
  \end{enumerate}
  These are all valid, however, the bottom two are slightly trivial for the following reasons. The second option of allowing more than one edge between any two vertex empires is counter-productive. Adding an extra edge also creates one extra country so the slack decreases by 1 so it gives you less freedom and gives you nothing you didn’t have before so there is not a lot of point considering this possibility.

  Considering the third option; if we allow the addition of edges between two vertices in the same vertex empire this is equivalent to letting two countries on the same empire neighbour each other. But this is the same as having one larger country, or in this case identifying the two vertices which is equivalent to the first change above.

  Thus we are left with one potential change, to have fewer than $m$ vertices in any one empire. What effect does this have on the slack? Well the number of edges stays the same as each empire must still be connected to every other empire. Similarly the surface doesn't change to if we look at the following formula from above:
  \begin{displaymath}
    C = E - V + \chi(S)
  \end{displaymath}
  we can see that removing one vertex means that the number of countries increases by one. Thus $2E - 3C$ will decrease by 3.
  
  So how does this help. Well we can notice a few things:
  \begin{enumerate}
    \item If $s(G, S) = 0$ then we have a triangular graph, i.e. every country is bounded by exactly 3 edges
    \item Thus if $s(G, S) = 2$ then we have one country with 5 bounding edges, or two with 4 bounding edges and the rest all triangular
    \item If $s(G, S) < 3$ you cannot reduce the number of vertices
    \item If $3 \leq s(G, S) < 6$ you can remove at most one vertex, etc.
  \end{enumerate}
  
  Thus we can learn something about the graph we haven't yet produced and know that we can exclude certain possibilities.
  \begin{example}
    The slack in the case $g=3, m=2$ is 5 and in the graph which proves the case (Figure~\ref{fig:proof-g3-m2}) we have one fewer vertex than you would get in $\bar{J}(14, 2)$. This reduces the slack to just 2 and you can see that there is indeed one country which has 5 bounding edges.
  \end{example}
  \begin{example}
    The slack in the case $g=3, m=3$ is only 2. Thus a similar approach of having a central vertex which is the only one in it's empire and is connected to all others in a ``wheel" will not work in this case. In fact every empire must have the full 3 vertices. Also we know that there must be 1 or 2 countries which will be non-triangular but that most of the countries will have to be triangular.
  \end{example}
  
  In general however, these graphs can be difficult to generalise and searching for them can be tedious. Hopefully with these tools the job might be a little quicker. Good luck!


\end{document}